\documentclass[11pt,reqno]{amsart} 
\usepackage[utf8]{inputenc} 
\usepackage[left=2.4cm,right=2.4cm,top=2.7cm,bottom=2.7cm]{geometry}
\usepackage{graphicx} 
\usepackage{color}
\usepackage{cite}
\usepackage{mathtools}
\usepackage{xcolor}

\renewcommand{\wedge}{\times}

\makeatletter

\renewcommand\subsubsection{\@secnumfont}{\bfseries}%
\renewcommand\subsubsection{\@startsection{subsubsection}{3}
  \z@{.5\linespacing\@plus.7\linespacing}{-.5em}%
  {\normalfont\bfseries}}
  
  \makeatother

\usepackage{array} 
\usepackage{paralist} 
\usepackage{verbatim} 
\usepackage{subfig} 

\usepackage{esint}
\usepackage{amsfonts}
\usepackage{amsmath}
\usepackage{amsthm}
\usepackage{amssymb}
\usepackage{hyperref}
\usepackage{dsfont}

\hypersetup{
	colorlinks=true,
	linkcolor=blue, 
	citecolor=blue,
	filecolor=blue,
	urlcolor=blue,
}

\newcommand{\mel}{\MoveEqLeft}

\newtheorem{theorem}{Theorem}[section]
\newtheorem{proposition*}{Proposition\textsuperscript{*}}

\newtheorem{corollary*}{Corollary\textsuperscript{*}}
\newtheorem{proposition}[theorem]{Proposition}

\newtheorem{lemma}[theorem]{Lemma}

\theoremstyle{definition}

\newtheorem{remark}[theorem]{Remark} 

\newtheorem{example*}{Example\textsuperscript{*}}

\numberwithin{equation}{section}

\usepackage{bbm}

\def\Limes#1#2 {\lim\limits_{#1\rightarrow #2}}
\renewcommand{\Re}{\operatorname{Re}}
\renewcommand{\Im}{\operatorname{Im}}

\newcommand{\htau}{\check{\tau}}

\DeclareMathOperator{\Id}{Id}

\def\eps{\epsilon}

\def\R{\mathbb{R}}
\def\C{\mathbb{C}}

\def\T{\mathbb{T}}
\def\Z{\mathbb{Z}}

\renewcommand{\P}{\mathcal{P}}

\def\N{\mathbb{N}}

\def\XXint#1#2#3{{\setbox0=\hbox{$#1{#2#3}{\int}$ }
\vcenter{\hbox{$#2#3$ }}\kern-.59\wd0}}

\renewcommand{\div}{\grad\cdot}
\renewcommand{\epsilon}{\varepsilon}

\def\scalar#1#2{\langle #1,#2 \rangle}
\def\de{\partial}
\renewcommand{\div}{\operatorname{div}}

\def\dx{\,\mathrm{d}x}

\def\dy{\,\mathrm{d}y}

\def\dl{\,\mathrm{d}l}

\newcommand{\be}{\beta}
\newcommand{\ep}{\varepsilon}

\newcommand{\ga}{\gamma}

\newcommand{\la}{\lambda}
\newcommand{\om}{\omega}
\newcommand{\si}{\sigma}
\newcommand{\te}{\theta}

\newcommand{\ze}{\zeta}
\newcommand{\Ga}{\Gamma}

\newcommand{\tu}{\widetilde{u}_\ep}
\newcommand{\tul}{\widetilde{u}_{\ep,l}}

\def\CC{\mathbb{C}}
\def\NN{\mathbb{N}}
\def\RR{\mathbb{R}}

\def\ZZ{\mathbb{Z}}

\def\TT{\mathbb{T}}

\def\bD{\mathbb{M}}

\newcommand{\cA}{{\mathcal A}}

\newcommand{\cC}{{\mathcal C}}

\newcommand{\cH}{{\mathcal H}}

\newcommand{\cK}{{\mathcal K}}

\newcommand{\cT}{{\mathcal T}}
\newcommand{\cX}{{\mathcal X}}
\newcommand{\cY}{{\mathcal Y}}

\newcommand\cZ{\mathcal Z}

\newcommand{\pd}{\partial}
\newcommand\minus\backslash

\newcommand\lan\langle
\newcommand\ran\rangle

\DeclareMathOperator{\supp}{supp}

\newcommand{\dd}{{\,\mathrm d}}
\DeclareMathOperator\Div{div}

\newcommand{\norm}[1]{\left\lVert#1\right\rVert}

\renewcommand\leq\leqslant
\renewcommand\geq\geqslant

\DeclareMathOperator{\BS}{BS}

\title[Vortex-sheet desingularization]{Vortex-sheet desingularization \\ for three-dimensional ideal fluids}
 
\author[A. Enciso]{Alberto Enciso}
\vspace{-1cm}
\address{   
\newline
\textbf{{\small Alberto Enciso}} 
\vspace{0.15cm}
\newline \indent Instituto de Ciencias Matem\'aticas, Consejo Superior de Investigaciones Cient\'\i ficas, 28049 Madrid, \indent Spain}
\email{aenciso@icmat.es}

 \author[A. J. Fern\'andez]{Antonio J.\ Fern\'andez}
 \address{ \vspace{-0.4cm}
\newline 
\textbf{{\small Antonio J. Fern\'andez}} 
\vspace{0.15cm}
\newline \indent Departamento de Matem\'aticas, Universidad Aut\'onoma de Madrid, 28049 Madrid, Spain}
 \email{antonioj.fernandez@uam.es}

 \author[D. Meyer]{David Meyer}
 \address{ \vspace{-0.4cm}
\newline 
\textbf{{\small David Meyer}} 
\vspace{0.15cm}
\newline \indent Max Planck Institute for Mathematics in the Sciences, Inselstrasse 22-26, 04103 Leipzig,
Germany}
 \email{david.meyer@mis.mpg.de}

\keywords{}
\subjclass[2020]{}

\begin{document}
\begin{abstract}
We prove a desingularization theorem for analytic vortex sheets of the 3D incompressible Euler equations. Starting from an analytic solution of the corresponding Birkhoff-Rott system, we construct, for every sufficiently small thickness parameter  $\varepsilon>0 $, an exact Euler vorticity supported on a tubular neighborhood of width $ O(\varepsilon) $ around the sheet, and defined on a time interval that does not shrink to~0 as $\ep\to0$. We show that, as $\ep \to 0$, these vorticities converge, in the sense of distributions, to the prescribed vortex sheet. In particular, we conclude that analytic 3D vortex sheet motions arise as limits of exact Euler flows with lifespan bounded from below independently of~$ \varepsilon $. The proof hinges on the study of vorticities defined in terms of a time-dependent foliation by almost parallel surfaces and of divergence-free vector fields tangent to these surfaces.
\end{abstract}

\maketitle

\section{Introduction}
We consider the desingularization of vortex sheets for the 3D incompressible Euler equations. In vorticity form, these equations read as
\begin{equation}\label{E.vorticity_intro}
    \pd_t \om + (u\cdot\nabla)\om = (\om\cdot\nabla)u\,,
    \quad
    \om=\operatorname{curl}u\,,\quad \div u=0\,,
\end{equation}
so the velocity $u$ is recovered from the vorticity $\om$ via the Biot--Savart law. These equations are supplemented with some initial condition $\om(\cdot,0)=\om_0$, where the initial vorticity~$\om_0$ is a divergence-free vector field.

A vortex sheet is, formally, a weak solution to the Euler equations whose vorticity is supported on a time-dependent surface~$\Gamma(t)$; as a consequence of this, the velocity has a jump discontinuity across the interface $\Gamma(t)$. The study of vortex sheets is classical in fluid mechanics,  both from the analytical and the numerical point of view. In particular, the Kelvin--Helmholtz instability has been known for over a century~\cite{Kelvin, Helmholtz}; see also \cite{Majda, Sa}. However, a rigorous connection between vortex sheets and three-dimensional Euler flows remains open. The aim of this paper is to establish such a link.

For simplicity, we will present the main ideas in the simplest setting, in which the system is $1$-periodic in the first two variables, and the moving surface is a graph. This corresponds to studying the Euler equations on $\mathbb M:= \T^2\times\RR$, where $\T:=\R/\Z$. The same approach works, with minor modifications, for general vortex sheets (say, with $\Ga(0)$ being any closed embedded surface in~$\R^3$); details are given in Section~\ref{S.closedSurface}.

We write points in~$\mathbb M$ as $x=(x',x_3)$ with $x'\in\T^2$ and $x_3\in\RR$. Given $T > 0$, a vortex sheet is described by a height function $h:\T^2\times[0,T]\to\RR$ and a vector field density $\varpi:\T^2\times[0,T]\to\RR^3$. The singular vorticity $\om^{\mathrm{sing}}$ corresponding to a vortex sheet then has the form
\begin{equation}\label{E.omega_sing_intro}
    \om^{\mathrm{sing}}(x,t):= \varpi(x',t) \, \delta\bigl(x_3-h(x',t)\bigr)\,.
\end{equation}
The associated velocity field $u^{\mathrm{sing}}$ is recovered using the Biot--Savart operator on~$\mathbb M$, which reads as
\begin{equation}
\BS[w](x,t):=\frac1{4\pi}\sum_{k\in\ZZ^2\times\{0\}}\int_{[-\tfrac12,\tfrac12]^2\times\R}  w(y,t) \wedge\frac{x-y+k}{|x-y+k|^3}\,\dy \,.\label{BS3D}
\end{equation}
Note that this formula, which is well defined without a principal value when~$w$ has zero mean, coincides with the usual Biot--Savart integral when $w$ is regarded as a vector field on $\mathbb{R}^3$ that is periodic in two directions. The incompressibility of~$u^{\mathrm{sing}}$ imposes geometric constraints on the sheet strength: if $\om^{\mathrm{sing}}$ is given by \eqref{E.omega_sing_intro}, then $\varpi$ is a tangent vector field on $\Ga(t)=\{(s,h(s,t))\}$ satisfying $\operatorname{div}_s\varpi'(s,t)=0$. Recall that the normal part of the velocity is continuous across the surface~$\Gamma(t)=\{ (x',h(x',t))\}$, while the tangential component exhibits a jump.

It is well known, at least at a formal level, that smooth three-dimensional vortex sheets are governed by the Birkhoff--Rott system,
\begin{equation}\label{E.BR_graph_intro}
\left\{
\begin{aligned}
    & \pd_t h(s,t) + \tilde{u}'(s,t)\cdot \nabla_s h(s,t) = \tilde{u}^3(s,t), \\
    & \pd_t \varpi(s,t) + \partial_{s_1} (\tilde{u}^1 \varpi)(s,t) + \partial_{s_2} (\tilde{u}^2 \varpi)(s,t) 
    = (\varpi'(s,t)\cdot\nabla_s)\tilde{u}(s,t).
\end{aligned}
\right.
\end{equation}
Here
\[
\tilde{u}(s,t) = \frac{1}{4\pi}\, {\rm PV} \sum_{k \in \ZZ^2 \times \{0\}} \int_{[-\tfrac12,\tfrac12]^2} \varpi(s-\varsigma,t) \times
\frac{\ga_0(s,t)-\ga_0(s-\varsigma,t)+k}
{\lvert \ga_0(s,t)-\ga_0(s-\varsigma,t)+k\rvert^3}
\dd \varsigma\,,
\]
is the averaged velocity trace at the sheet point $\ga_0(s,t):=(s,h(s,t))$, given by  $\tilde u=(u^++u^-)/2$ on $\Ga(t)$ with $u^\pm$ being the one-sided traces of $u^{\mathrm{sing}}$ on the interface\footnote{The last term is well defined since $\varpi$ is purely tangential.}.
The first equation is the kinematic condition for the graph, while the second describes the tangential sheet-density evolution in terms of tangential derivatives of~$\tilde u$.
In spite of the emergence of the aforementioned Kelvin--Helmholtz instabilities, the Birkhoff--Rott equation is locally well posed for analytic data $(\varpi^0,h^0)$ by the classical theorem of Sulem, Sulem, Bardos, and Frisch~\cite{SS}. This is however no longer the case if one moves to Sobolev spaces, as already proven by Caflisch and Orellana \cite{CO}.

Regarding the desingularization of vortex sheets, that is, the approximation of Birkhoff--Rott solutions by smooth Euler flows with concentrated vorticities, the contrast between two and three dimensions is sharp. 

In two dimensions, beginning with the approximation by vortex patches carried out by Benedetto and Pulvirenti \cite{BP}, one now has a substantial theory, including the approximation of graph-type sheets by exponentially concentrated smooth vorticities proven by Caflisch, Lombardo and Sammartino \cite{Caflisch}, and our recent smooth compactly supported regularization~\cite{EFM}. 

By contrast, no desingularization theorem for unsteady three-dimensional Euler vortex sheets has been available up to now. Indeed, Caflisch, Lombardo and Sammartino explicitly note \cite[p.~2138]{Caflisch} that their construction does not extend to 3D Euler, and that establishing such a result would require overcoming the vanishing lifespan of the corresponding smooth approximations. The main difficulty is structural. This is because desingularization arguments for concentrated vorticity{---both for vortex sheets and vortex filaments or points---}typically rely on the existence of globally well-posed dynamics for the fluid evolution. This is of course compatible with the 2D Euler equations, and actually underlies a number of significant recent gluing constructions carried out in this context~\cite{DDMW,DDMP}. However, for 3D Euler there is no comparable global well-posedness theory. Therefore, any viable approach must start off directly with the singular vortex sheet dynamics. This is the approach developed here.

Our theorem starts from an analytic solution $(h,\varpi)$ to the Birkhoff--Rott system~\eqref{E.BR_graph_intro} on some time interval and constructs, for each sufficiently small $\ep>0$, an exact Euler vorticity supported on a tubular neighborhood of thickness $O(\ep)$ around the surface $\Ga(t)$. The existence time is independent of~$\ep$, and the smooth vorticity $\om_\ep(\cdot,t)$ converges to $\om^{\mathrm{sing}}(\cdot,t)$ as $\ep\to0$, in the distributional sense.

Since the width of the layer is $O(\ep)$, the corresponding Euler solutions have large norms for small~$\ep$ after the natural Euler scaling. Thus, the result may also be viewed as a long-lifespan construction for a class of large solutions (of a very particular layered structure). Outside the classes of stationary solutions, and symmetric flows without swirl\footnote{Note that for axisymmetric and helical flows without swirl global well-posedness is known under mild assumptions. See for instance \cite{Danchin1} and the references therein. We also refer to \cite{GuoZhao2025} and the references therein for some recent results. }, there are very few known examples of long-lifespan 3D Euler flows~\cite{Danchin2,RenTian2024,EncisoPeraltaSalasTorres2023,GuoHuangPausaderWidmayer2023,GuoPausaderWidmayer2023, LiZhou2024}.

The key ingredient of the proof is the precise way in which layers appear in the construction. Specifically, the desingularized vorticity $\om_\ep(x,t)$ is built from a family of nearby graphs
\[
    x_3=\nu_{\ep,l}(x',t)\,,
\]
with sheet densities $\varpi_{\ep,l}:\T^2\times[0,T]\to\RR^3$, labeled by $l\in[-1,1]$. The spacing between neighboring leaves is encoded by the derivative
$
    \pd_l \nu_{\ep,l}
$,
which will remain under control in our construction. The corresponding (nonsingular) Euler vorticity $\om_\ep$ can be written as 
\begin{equation}\label{E.omeps_intro}
    \om_\ep(x,t)
    = \int_{-1}^1
    \varpi_{\ep,l}(x',t) \,
    \delta\bigl(x_3-\nu_{\ep,l}(x',t)\bigr)\,\dl\,.
\end{equation}
For the solutions constructed below, the leaves remain at mutual distance comparable to~$\ep$, so $\om_\ep$ is supported on an $O(\ep)$-neighborhood of the limiting sheet. The point is that $\om_\ep$ is an exact solution of the 3D Euler equations.

We can now informally state the main result. A more precise statement can be found in Theorem~\ref{main thm} below. 
 
\begin{theorem}\label{T.main_intro}
Let $(h,\varpi)$ be an analytic solution of the three-dimensional Birkhoff--Rott system~\eqref{E.BR_graph_intro} on some time interval. Then, there exists an $\ep_0>0$ such that, for every $0<\ep\leq \ep_0$, one can find 
\[
    \nu_{\ep,l}:\T^2\times[0,T)\to\RR\,, \quad \varpi_{\ep,l}:\T^2\times[0,T) \to \R^3 \,,    \quad \textup{with } l\in[-1,1],
\]
defined on an $\ep$-independent time interval $[0,T)$, such that the vorticity
\[
    \om_\ep(x,t)
    = \int_{-1}^1
    \varpi_{\ep,l}(x',t)\,
    \delta\bigl(x_3-\nu_{\ep,l}(x',t)\bigr)\dl\in C^\infty(\mathbb M\times[0,T),\RR^3)\,,
\]
is an exact solution to the 3D Euler equations on~$\mathbb M \times [0,T)$. Moreover, 
\[
    \supp\om_\ep(\cdot,t)
    \subset
    \bigl\{x\in \mathbb M:\ |x_3-h(x',t)|\lesssim \ep\bigr\}\,,
\]
and
\[
    \om_\ep(\cdot,t)\rightharpoonup \om^{\mathrm{sing}}(\cdot,t)\,, \quad \textup{as } \ep \to 0\,,
\]
for all $t\in[0,T)$. Here, $\om^{\mathrm{sing}}$ is given by~\eqref{E.omega_sing_intro}.
\end{theorem}
 
The paper, which is devoted to proving the more precise version of this result given in Theorem~\ref{main thm}, is organized as follows. Section~\ref{S.setup} introduces the graph formulation, the layered ansatz, and the precise desingularization theorem (Theorem~\ref{main thm}). Section \ref{S.effective} reduces the problem to a system for the leaf variables. In other words, we derive the effective system of evolution equations. In Section \ref{S.localExistence} we prove the local well-posedness of the effective system, with a lifespan which is independent of $\ep$, and then complete the proof of Theorem \ref{main thm}. The kernel estimates required for the local well-posedness are postponed to Sections \ref{S.Kernel} and \ref{S.prop44}. Finally, Section~\ref{S.closedSurface} explains how the same argument extends to normal graphs over an arbitrary analytic closed surface in~$\RR^3$. 

\section{Geometric setup and main result} \label{S.setup}

For $\rho>0$, we set $\T^2_\rho:=\T^2\times B_\rho(0)$ and denote by $\cH(\T^2_\rho)$ the space of holomorphic functions~$f$ on the strip
\[
\big\{z=s+i\beta\in\CC^2 : s,\,\beta\in\R^2,\;  |\beta|\leq \rho\big\}\,,
\]
which satisfy the periodicity condition
\[
f(z_1,z_2)=f(z_1+1,z_2)=f(z_1,z_2+1)\,,
\]
for all $z=(z_1,z_2)$ in the strip. Further, we denote by 
$$
\cX_\rho:=\big\{f\in \cH(\T^2_\rho):  f(\T^2)\subset\R\big\}\,,
$$
the space of periodic holomorphic functions $f$ on the strip that arise as complexifications of real-valued analytic functions on $\T^2$. We regard $\cX_\rho$ as a Banach space, equipped with the norm 
\begin{align} \label{E.norm rho}
\norm{f}_{\rho}:=\sup_{|\beta|\leq \rho}\norm{f(\cdot+i\beta)}_{C^{\frac{1}{2}}(\T^2)}.
\end{align}

We will also need to keep track of families of functions $f_l\in \cX_{\rho}$ that depend on a real parameter $l\in[-1,1]$. In this case, we will say that the family $f_l$ is in the space $C^r\mathcal \cX_\rho$, for some reals $r,\rho>0$, if $l\mapsto f_l(\cdot)$ is of class $C^{\lfloor r \rfloor}((-1,1),\mathcal \cX_\rho)$ and the norm
\[
\|f_l\|_{C^r\mathcal \cX_\rho}:=\sum_{j=0}^{\lfloor r\rfloor}\sup_{|l|\leq 1}\|\partial_l^jf_l\|_\rho+\sup_{l,l'\in [-1,1]}\frac{\|{\de_l^{\lfloor r\rfloor}(f_l-f_{l'})\|}_{\rho}}{|l-l'|^{r-\lfloor r\rfloor}}
\]
is finite.

We consider an initial vorticity for
\begin{equation} \label{eq 3dEuler}
    \left\{
    \begin{aligned}
        \, & \pd_t \omega + (u \cdot \nabla) \omega = (\omega \cdot \nabla) u \quad &&\textup{in } \bD \times (0,T)\,, \\
        & u:= \BS[w] &&\textup{in } \bD \times [0,T)\,, \\
        & \omega(\cdot,0) = \omega_0 && \textup{in } \bD\,,
    \end{aligned}
    \right.
\end{equation}
which is given as an infinite superposition of almost parallel vortex sheets. To make this precise, we start by considering a family of functions $\nu_{\ep,l}^0\in \cX_{\rho_0}$, for some fixed $\rho_0>0$, depending on a small parameter $\ep>0$ and of class $C^r$ for some $1 \leq r \leq \infty$ with respect to the parameter $l \in (-1,1)$. We assume that the derivative of $\nu_{\ep,l}^0$ with respect to the parameter~$l$ is small and reasonably close to a constant, in that the functions
\begin{align}\label{E.theta0}
\theta_{\eps,l}^0(z):=\frac{\de_l\nu_{\eps,l}^0(z)}{\ep}-1
\end{align}
belong to $\cX_{\rho_0}$ and are bounded as
\begin{align}
\sup_{|l|\leq 1}\norm{\theta_{\ep,l}^0}_{\cX_{\rho_0}}<\frac{1}{3}\,.\label{theta small}
\end{align}
Note that this bound does not mean that the functions $\nu_{\eps,l}^0$ themselves are of order~$\ep$; instead, they should be understood as small perturbations of functions that do not depend on~$l$.
For technical reasons, we also assume that the imaginary parts of these functions are small, in that
\begin{align}
\sup_{|l| \leq 1}\sup_{|\beta|\leq \rho_0} \left\{ \norm{\Im \nu_{\eps,l}^0(\cdot+i\beta)}_{C^1(\T^2)}+\norm{\Im \theta_{\eps,l}^0(\cdot+i\beta)}_{C^0(\T^2)} \right\} \leq B \,, \label{im smal}
\end{align}
for some suitably small constant~$B$. Since $\nu_{\eps,l}^0(\cdot+i\beta)$ and $\theta_{\eps,l}^0(\cdot+i\beta)$ are real-valued for $\beta=0$, this can always be achieved by taking a small enough~$\rho_0$. 

We then define a family of surfaces in~$\TT^2_{\rho_0}\times\R$ as the image of the maps
\begin{align}
\gamma_{\ep,l}^0(s):=\left(s,\nu_{\ep,l}^0(s)\right)\,, \label{def gamma}
\end{align}
with small $\ep>0$ and $l\in[-1,1]$. Note that condition \eqref{theta small} ensures that the surface $\gamma_{\ep,l}^0(\T^2)\subset\R^3$, of dimension 2, does not intersect $\gamma_{\ep,l'}^0(\T^2)$ for any $l\neq l'$.

As a consequence of this, if the map $l\mapsto \gamma^0_{\ep,l}$ is of class $C^r((-1,1),\cX_{\rho_0}^3)$ for some $1\leq r\leq\infty$, \eqref{theta small} also guarantees that
\[
(l,s)\mapsto \gamma^0_{\ep,l}(s)
\]
defines a $C^r$-diffeomorphism~$\Gamma^0_\ep$ from $(-1,1) \times \T^2$ onto the set
\[
\cT^0_\ep:=\bigcup_{l\in(-1,1)}\gamma^0_{\ep,l}(\T^2)\,,
\]
which is a thin tubular neighborhood (of width at most $C\ep$) around the surface $\gamma_{\ep,0}^0(\T^2)$ in $\bD$. Note that this surface is given by the graph of the analytic real-valued function $\nu^0_{\ep,0}|_{\T^2}$. 

We next introduce a parametrized family of holomorphic vector fields $\widetilde{\omega}_{\eps,l}^0 \in \cX_{\rho_0}^3$. Together with the above diffeomorphism~$\Gamma^0_\ep$, we will use these vector fields to define an initial vorticity $\om^0_\ep:\bD \to\RR^3$, supported on~$\cT^0_\ep$, by means of the formula
\begin{equation} \label{def omega0}
\omega_\eps^0(x)=\begin{cases}
\widetilde{\omega}_{\ep,l}^0(s) &\text{if $x=\gamma_{\ep,l}^0(s)\in\cT^0_\ep$}\,,\\
0 &\text{if } x\not\in\cT^0_\ep\,.
\end{cases}
\end{equation}
We impose the following conditions on this family of holomorphic vector fields:
\begin{enumerate}
    \item {Support:} $\widetilde{\omega}_{\eps,l}^0 \equiv 0$ for all $|l| \geq 1/2$. \label{itemone}
    \item {Regularity:} The map $l \mapsto \widetilde{\omega}_{\eps,l}^0$ is of class $C^r((-1,1), \cX_{\rho_0}^3)$ for some real $r > 1$, which we will assume is not an integer.
    \item {Tangency:} For all $s\in\T^2$ and all $l\in(-1,1)$, the vector  $\widetilde{\omega}_{\eps,l}^0(s)$ is tangent to the surface $\gamma_{\eps,l}^0(\T^2)\subset\bD$ at the point $\gamma_{\ep,l}^0(s)$. \label{itemthree}
    \item {Divergence-free:} With $\om^0_\ep$ defined as above, $\Div\om^0_\ep =0$ in~$\bD$. \label{itemfour}
    \item Zero mean: $\int_{\bD} \om^0_\ep\dd x=0$. \label{itemfive}
\end{enumerate}
For later purposes, we also set the ``vorticity density''
$
\varpi_{\ep,l}^0:= \ep(1+\te_{\ep,l}^0) \widetilde{\omega}_{\ep,l}^{\,0}\,.
$

See also figure \ref{fig1} below for a sketch.

\begin{figure}[h]
\includegraphics[width=0.42\textwidth]{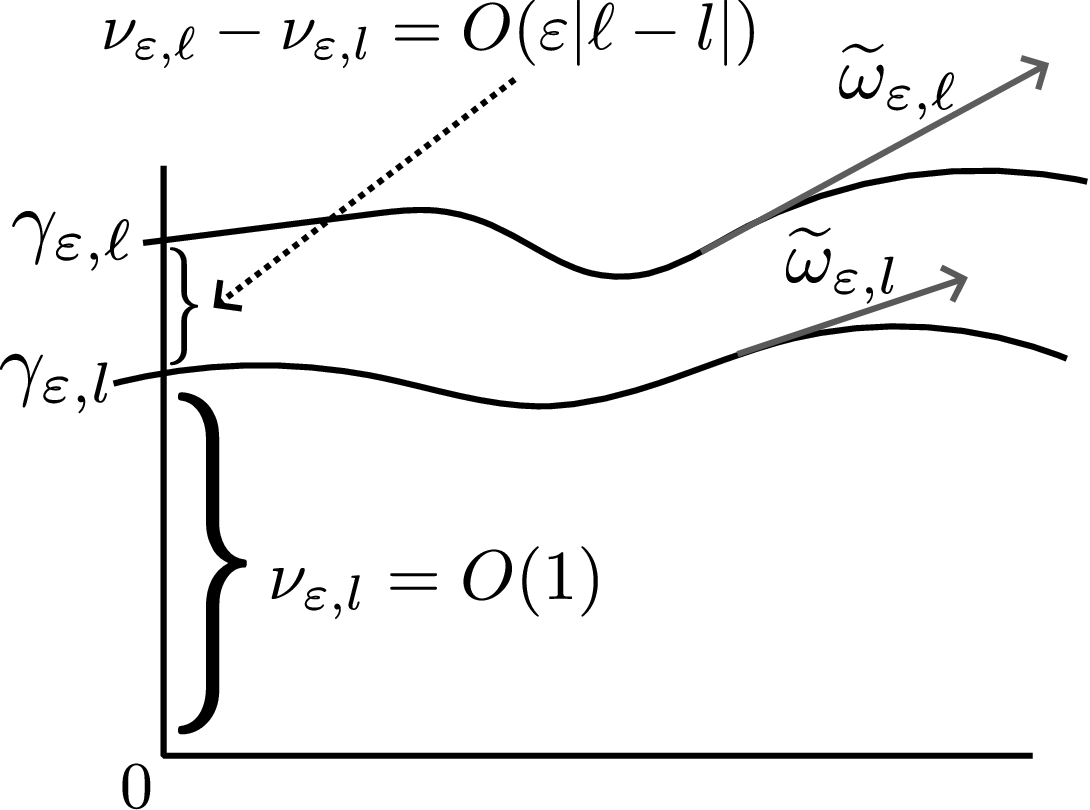}
\caption{A sketch of the two-dimensional cross-section of the arrangement of some of the quantities\label{fig1}}
\end{figure}

\begin{remark} \label{R.initialdata}
    It is easy to construct functions~$\widetilde{\om}^0_{\ep,l}$ with these properties. The only point that is perhaps unclear is that~\ref{itemthree} and~\ref{itemfour} do not clash. To see that this is the case, parametrize a vector field with the tangency condition~\ref{itemthree} as
    \[
    \widetilde{\om}^0_{\ep,l}(s)= a^1_{\ep,l}(s)\, \partial_{s_1}\gamma^0_{\ep,l}(s)+ a^2_{\ep,l}(s)\, \partial_{s_2}\gamma^0_{\ep,l}(s)\,,
    \]
    with arbitrary functions $a^j_{\ep,l}\in\cX_{\rho_0}$ whose dependence on~$l$ is of class~$C^r$, and which vanish for $|l|\geq1/2$. In~$\cT^0_\ep$, the fact that
    \[
    \Div \om_\ep^0= 0\ \Longleftrightarrow\ \frac1{J_{\ep,l}^{\, 0}}\partial_{s_1}(J_{\ep,l}^{\, 0}\,a^1_{\ep,l})+\frac1{J_{\ep,l}^{\, 0}}\partial_{s_2}(J_{\ep,l}^{\, 0}\,a^2_{\ep,l})=0 \qquad \text{for all } (l,s)\in(-1,1)\times \T^2\,,
    \]
    where $J_{\ep,l}^{\, 0}:=\partial_l\gamma^0_{\ep,l}(s)\cdot \partial_{s_1}\gamma^0_{\ep,l}(s)\times \partial_{s_2}\gamma^0_{\ep,l}(s) $ is the Jacobian, provides the only constraint for the coefficients, other than the soft integral condition \ref{itemfive}. A useful concrete example to keep in mind is 
    \[
    \nu_{\eps,l}^0(z):=\nu_0(z)+\eps l\,,\qquad \widetilde{\omega}_{\eps,l}^0(z)=\eta(l)\,\widetilde{\omega}_0(z)\,,
    \]
    where the vorticity is supported in a small tubular neighborhood of the graph of $\nu_0|_{\T^2}$, where $\widetilde{\omega}_0$ is the complexification of an analytic tangent vector field tangent to this surface and divergence-free with respect to the induced metric, and where $\eta$ is a cutoff function.
\end{remark}

\begin{remark}
    The above representation of~$\om^0_\ep$ is not unique. For instance, one can rescale $l$ to get a different family of functions $\gamma_{\ep,l}^0,\ \widetilde{\om}^0_{\ep,l}$ defining the same~$\om^0_\ep$.
\end{remark}

Under the conditions \ref{itemone}--\ref{itemfive}, the initial vorticity~\eqref{def omega0} is of class~$C^r(\bD)$, with $r > 1$ a non-integer. It is therefore well-known that, for any $\ep>0$, there exist some maximal time of existence $T_\ep\in(0,\infty]$ and a unique local strong solution~$\om_\ep$ to the Euler equation~\eqref{eq 3dEuler} in $\bD\times(0,T_\ep)$ with initial vorticity~$\om^0_\ep$. Furthermore, $\om_\ep \in C_{\mathrm{w}}([0,T_\ep),C_c^r(\bD))$, where the subscript ``w'' denotes weak-$*$ continuity. We will denote by
\[
u_\ep:= \BS[\om_\ep]
\]
the corresponding fluid velocity field, which is of class $L^\infty([0,T_\ep),C^{r+1}(\bD))\cap W^{1,\infty}([0,T_\ep), C^r(\bD))$.

It is also standard that, by a classical formula due to Cauchy (see e.g.\ \cite[Formula 2.116]{Majda}), the unique solution to \eqref{eq 3dEuler} is given by 
\begin{align}
\om_\ep(x,t) :={\rm D} X_t( X_t^{-1}(x))  \om_\eps^0( X_t^{-1}(x))\,,\label{Cauchy}
\end{align}
where $ X_t$ denotes the particle-trajectory map at time~$t$, namely the homeomorphism defined by the value at time~$t$ of the unique solution to the ODE
\begin{equation} \label{E.flowEquation}
\frac{\partial}{\partial t}  X_t(x) = u_\ep( X_t(x),t)\,, \qquad X_0(x) = x\,.
\end{equation}
Equivalently, the vorticity at time~$t$ is the push-forward of~$\om_\ep^0$ by~$X_t$, that is,
\[
\om_\ep(\cdot, t)=X_{t\#} \om_\ep^0\,.
\]
Therefore, the support of~$\om_\ep(\cdot,t)$ is contained in the set
\begin{equation}
\cT_\ep(t):= X_t(\cT^0_\ep)\,.\label{def teps}
\end{equation}

Let us now focus on the structure of the solution~$\om_\ep$ at small but positive times. For this, let us respectively denote by
\[
\Pi':\bD\to\T^2\,,\qquad  \Pi:\bD\to \RR
\]
the projectors mapping $\bD=\T^2\times\RR$ to its factors.  
It follows from the regularity of the solution~$\om_\ep$ that for all small enough~$t\geq0$, there exists a diffeomorphism $\sigma_t$ of~$\TT^2$ such that
\begin{equation} \label{E.sgraph}
  \Pi'[ X_t(\ga_{\ep,l}^0(\si_t(s)))] = s \,, \quad \textup{for all } s \in \TT^2\,.
\end{equation}
Note that $\si_0$ is the identity. Using this map, we reparametrize the transported surfaces as 
\begin{equation} \label{def gamma evolved}
    \gamma_{\ep,l}(s,t) :=  X_t(\ga_{\ep,l}^0(\si_t(s)))\,,
\end{equation}
and stress that, for small~$t$, the surface~$\gamma_{\ep,l}(\TT^2\times\{t\})$ is a graph over~$\T^2$. In fact,
\[
\gamma_{\ep,l}(s,t)=\left(s,\nu_{\eps,l}(s,t)\right)\,,
\]
with
\begin{equation}  \label{def nu eps l}
\nu_{\eps,l}(s,t) :=\Pi[ X_t(\ga_{\ep,l}^0(\si_t(s)))]\,.
\end{equation}
Observe that, by the continuity of the flow map,  the intersection $\gamma_{\eps,l}(\T^2\times\{t\})\cap \gamma_{\eps,l'}(\T^2\times\{t\})$ is empty if $l\neq l'$, so the map
\begin{equation}\label{E.diffeogamma}
(l,s)\mapsto \gamma_{\ep,l}(s,t)
\end{equation}
is a $C^{r}$-diffeomorphism $(-1,1)\times\TT^2\to \mathcal{T}_\eps(t)$. 
This property holds, in fact, as long as 
$$
\sup_{(l,s)\in(-1,1)\times\T^2}|\theta_{\eps,l}(s,t)|<1\,,
$$
where
\begin{equation} \label{E.thetaepsl}
\theta_{\eps,l}(s,t):=\frac{\partial_l \nu_{\ep,l}(s,t)}\ep-1\,.
\end{equation}
We can use this function to write
\begin{equation}  \label{def nu eps l2}
\nu_{\ep,l}(s,t) - \nu_{\ep,\ell}(s,t) := \ep \int_\ell^l \big( 1+ \theta_{\eps,\mu}(s,t) \big) \dd \mu\,.
\end{equation}
Since the support of $\om_\ep(\cdot,t)$ is contained in~$\cT_\ep(t)$, for small enough~$t$, there exists a unique family of functions $\widetilde{\omega}_{\eps,l}(\cdot,t):\T^2\rightarrow \R^3$, which are of class~$C^r$, such that
\begin{equation}\label{E.ansatzomega}
    \om_\ep(x,t)=\begin{cases}
\widetilde{\omega}_{\ep,l}(s,t) &\text{if $x=\gamma_{\ep,l}(s,t)\in\cT_\ep(t)$}\,,\\
0 &\text{if } x\not\in\cT_\ep(t)\,.
\end{cases}
\end{equation}
The vorticity should be interpreted as the compactly supported time-dependent distribution that acts on $\Phi\in C^0(\bD,\RR^3)$ as
\begin{equation*} 
\langle  \om_\ep(\cdot,t), \Phi\rangle := \int_{\TT^2} \int_{-1}^1 \widetilde{\om}_{\ep,l}(s,t) \cdot \Phi(\gamma_{\ep,l}(s,t)) \,J_{\ep,l}(s,t) \dd l \dd s\,,
\end{equation*}
where $J_{\ep,l}$ is the Jacobian given by
\begin{equation} \label{E.JacobianRemark}
    J_{\ep,l} := \pd_l \ga_{\ep,l} \cdot (\pd_{s_1} \ga_{\ep,l }\times  \pd_{s_2} \ga_{\ep,l}) = \pd_l \nu_{\ep,l} = \ep (1+ \te_{\ep,l})\,.
\end{equation}
For later purposes, we also set the ``vorticity density''
$$
\varpi_{\ep,l}:= J_{\ep,l}\, \widetilde{\omega}_{\ep,l} = \ep(1+\te_{\ep,l}) \widetilde{\omega}_{\ep,l}\,.
$$

Having this notation and this geometric setting at hand, we can now give a more precise version of our main result:

\begin{theorem}\label{main thm}
Let $\rho_0 > 0$ be a real number and $r > 1$ a non-integer. For $\ep \in (0,1)$, consider families of functions $(\nu_{\ep,l}^0)_\ep\subset  C^r \cX_{\rho_0} $ and $ (\widetilde{\omega}_{\ep,l}^0)_\ep \subset C^r \cX_{\rho_0}^3$ satisfying the bounds~\eqref{E.theta0}-\eqref{theta small}, the structural conditions \ref{itemone}-\ref{itemfive}, and such that
\[
C_0:=\sup_{0<\ep<1}\sup_{|l|\leq 1} \left(\, \|\nu_{\ep,l}^{0}\|_{ \cX_{\rho_0}}+ \|\nabla\nu_{\ep,l}^{0}\|_{ \cX_{\rho_0}^2} + \|\varpi_{\ep,l}^0\|_{\mathcal \cX_{\rho_0}^3}\,\right)<\infty\,.\label{def C0}
\]
Then, there exist $\ep_0 \in (0,1)$ such that, for every $\ep \in (0,\ep_0]$, there exist a constant $B_0 > 0$ and  a time $T > 0$, depending on $C_0$ but not on $\ep$, such that, if the bound~\eqref{im smal} holds for some $B<B_0$, one has:
\begin{itemize}
    \item[\textbf{(i)}] The strong solution $\om_\ep$ to the Euler equations~\eqref{eq 3dEuler} with initial datum \eqref{def omega0} exists up to time~$T$ and $\omega_\ep \in C_{\mathrm{w}}([0,T), C_c^r(\bD))$. Furthermore, it is of the form~\eqref{E.ansatzomega} for all $t \in [0,T)$, with $\widetilde{\om}_{\ep,l}(s,t)$ tangent to the surface defined by \eqref{def gamma evolved}. \smallbreak  
    \item[\textbf{(ii)}] For all $t \in [0,T)$, the support of $\om_\ep(\cdot,t)$ is contained in the tubular neighborhood~$\cT_\ep(t)$ given in \eqref{def teps}, of width less than~$ C_1 \ep$, for some $C_1 > 0$.
    \smallbreak
    \item[\textbf{(iii)}] If the initial data converges as
    $$
    \nu_{\ep,0}^0 \to h^0 \quad \textup{in }\cX_{\rho_0}\,, \quad \textup{and} \quad \int_{-1}^1 \varpi_{\ep,l}^0 \dd l \to \varpi^0 \quad \textup{in }\cX_{\rho_0}^3\,,\ \textup{ as } \ep \to 0^{+}\,,
    $$
    then the solution converges to the vortex sheet in the following sense: let $(h, \varpi)$ be the unique analytic solution to \eqref{E.BR_graph_intro} with initial data $(h^0, \varpi^0)$. Then, it follows that, as $\ep \to 0^{+}$,
    \begin{align*}
    & {\textbf{(a)}}\ \nu_{\eps,l}\rightarrow h\quad \text{in $L^\infty([0,T],C^k(\T^2))$ for all $l\in (-1,1)$ and all $k\in \N$}\,,\\
    & {\textbf{(b)}}\ \int_{-1}^1\varpi_{\eps,l}\dd l\rightarrow \varpi \quad \text{in $L^\infty([0,T],C^k(\T^2))$ for all $k\in \N\,.$}
    \end{align*}
In particular, $\omega_\eps(\cdot,t)$ converges distributionally to  $\omega^{\mathrm{sing}}(\cdot,t)$  given in \eqref{E.omega_sing_intro}.    
\end{itemize}
Furthermore, if $ (\widetilde{\omega}_{\ep,l}^0)_\ep \subset C^r \cX_{\rho_0}^3$ and $({\nu_{\ep,l}^0})_\ep \subset C^r \cX_{\rho_0}$ for all $r>1$, then $\omega_\ep \in C^\infty([0,T)\times\bD)$. 
\end{theorem}

\subsection{Sketch of the proof}  \label{S.sketch}

In one phrase, our approach to establishing the existence part of Theorem \ref{main thm} is to rephrase \eqref{eq 3dEuler} as an evolution system for the leaf variables $\nu_{\ep,l}$, $\te_{\ep,l}$ and $\widetilde{\om}_{\ep,l}$ (see \eqref{def nu eps l}, \eqref{E.thetaepsl} and \eqref{E.ansatzomega}), and then to show existence, uniqueness and boundedness of solutions to this system, with a lifespan independent of $\ep$. 

First, in Lemma \ref{eff system}, we derive an effective system of equations for $\nu_{\ep,l},\, \theta_{\ep,l}$ and $\widetilde{\omega}_{\ep,l}$.  Once we have this system, \eqref{eq:3.2}--\eqref{eq:3.5}, in Lemma \ref{weak sol}, we show that if we have a solution to \eqref{eq:3.2}--\eqref{eq:3.5}, not necessarily coming from a solution to \eqref{eq 3dEuler}, then we can construct a solution to \eqref{eq 3dEuler} from such data. Moreover, this solution will keep the desired structure, see \eqref{E.ansatzomega}. 

Taking this connection into account, the proof of Theorem \ref{main thm} (i) and (ii) is reduced to establishing an existence result for \eqref{eq:3.2}--\eqref{eq:3.5} with a lifespan independent of $\ep$. Section \ref{S.localExistence} is mainly devoted to the proof of such a result. The proof relies on an abstract Cauchy-Kovaleskaya type theorem, namely Nishida's Theorem (cf. Theorem \ref{Nishida}).  

The key step to show that we can apply Theorem \ref{Nishida} to \eqref{eq:3.2}--\eqref{eq:3.5} is to prove uniform bounds on the velocity $u_\ep$, restricted to the surfaces $\gamma_{\eps,l}$, and taken as a function of $s$ in the analytic function spaces $\cX_\rho$.
The tangential nature of the initial vorticity and the fact that this structure is propagated under the evolution (cf.\ Lemma \ref{tan lemma}) are essential here; otherwise the $l$-derivative of $u_\ep$ (which is of order $\eps^{-1}$ and diverges to the jump of the velocity in the limit) would contribute to the vortex stretching, and would destroy the uniform time of existence.
The other crucial observation is that we can use the prescribed structure of the vorticity as a superposition of vortex sheets to decompose the velocity into the contributions of the individual sheets. Therefore, we only need to show the boundedness of a variant of the usual Birkhoff--Rott operator (see Proposition \ref{main est}) which describes the velocity that a vortex sheet induces on an almost parallel nearby sheet. In order to streamline the presentation, we will postpone many technical details to Sections \ref{S.Kernel} and \ref{S.prop44}.

Once we have the local well-posedness for \eqref{eq:3.2}--\eqref{eq:3.5} with a lifespan which is independent of $\ep$, we conclude Section \ref{S.localExistence} by proving the distributional convergence of the constructed solutions to the solution of the Birkhoff--Rott equation. Roughly speaking, the argument here is that the solution must converge to a weak solution whose vorticity is a measure concentrated on a curve by compactness, which must also be analytic by our a priori bounds and lower semicontinuity. The uniqueness of analytic solutions to the Birkhoff-Rott equation then ensures that the limit is the solution to Birkhoff-Rott.

\section{Back and forth from 3D Euler to an effective system for vortex sheets} \label{S.effective}

This section is devoted to showing the connection between \eqref{eq 3dEuler} and the effective system for the leaf variables mentioned in Subsection \ref{S.sketch}.

It will be absolutely crucial for us that the tangency property~\ref{itemthree} in Section \ref{S.setup} is propagated under the evolution. More precisely, we have:

\begin{lemma}\label{tan lemma}
Suppose that the (strong) solution~$\om_\ep$ to \eqref{eq 3dEuler} with initial datum $\omega_\ep^0$ as in \eqref{def omega0} exists up to some positive time~$T_\star$, and that the map~\eqref{E.diffeogamma} remains a diffeomorphism for all $t\in[0,T_\star]$. Then, for all $(l,s,t)\in (-1,1) \times \TT^2\times [0,T_\star)$, the vector $
\widetilde{\omega}_{\eps,l}(s,t)\in\RR^3$ is tangential to the 
surface $\gamma_{\eps,l}(\T^2,t)\subset\bD$ at the point $\gamma_{\ep,l}(s,t)$.
\end{lemma}
\begin{proof}
With some abuse of notation, let $T_{\gamma_{\eps,l}(s,t)}$ be the tangent space to the surface $\gamma_{\eps,l}(\T^2,t)$ at the point $\gamma_{\eps,l}(s,t)$.
We first note that the preimage of some point $x=\gamma_{\eps,l}(s,t)$ under the flow map $ X_t$ is precisely \begin{align}
 X_t^{-1}(x)=\gamma_{\eps,l}^0(\sigma_t(s))\,.\label{preim}
\end{align}
By \eqref{def gamma evolved}, we can write \begin{align*}
T_{\gamma_{\eps,l}(s,t)}={\rm D} X_t(\gamma_{\eps,l}^0(\sigma_t(s))) T_{\gamma_{\eps,l}^0(\sigma_t(s))}\,.
\end{align*}
Now, we know that $\widetilde{\omega}_{\eps,l}^0(\sigma_t(s))\in T_{\gamma_{\eps,l}^0(\sigma_t(s))}$
by the condition~\ref{itemthree} and hence also that
\begin{align*}
{\rm D} X_t(\gamma_{\eps,l}^0(\sigma_t(s)))\widetilde{\omega}_{\eps,l}^0(\sigma_t(s))\in T_{\gamma_{\eps,l}(s,t)}\,.
\end{align*}
However, \eqref{preim} and \eqref{Cauchy} yield \begin{align*}
{\rm D} X_t(\gamma_{\eps,l}^0(\sigma_t(s)))\widetilde{\omega}^0_{\eps,l}(\sigma_t(s))=\widetilde{\omega}_{\eps,l}(s,t)\,,
\end{align*}
establishing the lemma.
\end{proof}

Having this property at hand, we now derive the desired effective system. In addition to the notation introduced in Section \ref{S.setup}, we will use the notation \begin{align} \label{E.tul}
\tul(s,t):=u_{\eps}(\gamma_{\eps,l}(s,t),t)\,,
\end{align}
and denote by
\[
\tul'(s,t):= \Pi' (\tul(s,t)) = \left(u^1_\eps(\gamma_{\eps,l}(s,t),t),u^2_\eps(\gamma_{\eps,l}(s,t),t)\right)^{\rm T},
\]
the projection on the horizontal plane. Here and in what follows, superscripts are used to denote the components of a vector field. We will further write $\nabla_s$ to denote the gradient in the $s$-variable, and write $\tu'\cdot\nabla_s$ with the obvious meaning. We use the same convention for the divergence and the Jacobian matrix in the $s$-variable.  

We will also need an orthonormal basis of tangent vectors to the surface $\gamma_{\ep,l}(\T^2,t)$, such as
\begin{equation*}
\tau_{1}:=\frac{\htau_{1,\ep,l}}{|\htau_{1,\ep,l}|}\,,\quad \textup{and} \quad \tau_{2}:=\frac{\htau_{2,\ep,l}|\htau_{1,\ep,l}|^2-\htau_{1,\ep,l}\,(\htau_{1,\ep,l}\cdot \htau_{2,\ep,l})} {|\htau_{2,\ep,l}|\htau_{1,\ep,l}|^2-\htau_{1,\ep,l}\,(\htau_{1,\ep,l}\cdot \htau_{2,\ep,l})|}\,,
\end{equation*}
with
\[
\htau_{1,\ep,l}(s,t):=\partial_{s_1}\gamma_{\ep,l}(s,t)=\left(1,0,\partial_{s_1}\nu_{\ep,l}(s,t)\right)^{\rm T}\,,\qquad \htau_{2,\ep,l}(s,t):=\partial_{s_2}\gamma_{\ep,l}(s,t)=\left(0,1,\partial_{s_2}\nu_{\ep,l}(s,t)\right)^{\rm T}\,.
\]
 
\begin{lemma}\label{eff system}
Assume that the unique strong solution~$\om_\ep$ to \eqref{eq 3dEuler} with initial datum $\omega_\ep^0$ as in \eqref{def omega0} exists up to some time $T_\star > 0$, has the structure \eqref{E.sgraph}-\eqref{E.ansatzomega} and that it holds that \begin{align}
    \sup_{|l|\leq 1}\norm{\theta_{\eps,l}}_{L^\infty(\T^2\times[0,T_\star])}<1\,.\label{thetaok}
\end{align} Then, for $t \in (0, T_\star)$, it holds that
\begin{align}
        \partial_t \nu_{\varepsilon,l} + (\tul' \cdot \nabla_s) \nu_{\ep,l}   &= \widetilde{u}^3_{\ep,l}\,, \label{eq:3.2} \\
    \partial_t \theta_{\varepsilon,l} + \nabla_s \cdot \big(\tul' (1+\theta_{\ep,l}) \big) &= 0\,, \label{eq:3.3}
\end{align}
and the vorticity evolves according to
\begin{equation}\label{eq:3.5}
    \partial_t \widetilde{\omega}_{\varepsilon,l} + \widetilde{u}^1_{\ep,l} \,\partial_{s_1} \widetilde{\omega}_{\varepsilon,l} + \widetilde{u}^2_{\ep,l} \, \partial_{s_2} \widetilde{\omega}_{\varepsilon,l} = \mathcal A_{\ep} \widetilde{\omega}_{\varepsilon,l}\,.
\end{equation}
Here,
\begin{equation*} 
  \mathcal A_{\ep}\widetilde{\omega}_{\varepsilon,l} := (\widetilde{\omega}_{\varepsilon,l} \cdot \tau_1) D_{\tau_1} \tul + (\widetilde{\omega}_{\varepsilon,l} \cdot \tau_2) D_{\tau_2} \tul\,,
\end{equation*}
where 
\begin{equation*}
    D_{\tau_1}\tul := \frac{1}{|\partial_{s_1}\ga_{\ep,l}|} \partial_{s_1}\tul\,, \quad 
    D_{\tau_2}\tul:= \frac{1}{|\partial_{s_2}\ga_{\ep,l} - ( \partial_{s_2}\ga_{\ep,l} \cdot\tau_1) \tau_1|} \left( \partial_{s_2} \tul- \frac{ \partial_{s_2} \ga_{\ep,l} \cdot \tau_1}{|\partial_{s_1}\ga_{\ep,l}|}\, \partial_{s_1} \tul \right)\,,
\end{equation*}
are the directional derivatives of $\tul$ along $\tau_j$.
\end{lemma}

\begin{remark} \label{R.effectiveSystem} In terms of the ``vorticity density" 
    \begin{align}
    \varpi_{\eps,l}:=\eps(1+\theta_{\eps,l})\widetilde{\omega}_{\eps,l}\,,\label{def pi}
    \end{align}   
    an elementary calculation shows that \eqref{eq:3.5} is equivalent to
    \begin{align}
    \de_t\varpi_{\eps,l}+\de_{s_1}(\tul^1\,\varpi_{\eps,l})+\de_{s_2}(\tul^2 \,\varpi_{\eps,l})=\mathcal A_{\eps} \varpi_{\eps,l}\label{eff3'}\,,
    \end{align}
    provided that the factor $1+\theta_{\eps,l}$ does not vanish. 
    
    Likewise, let us stress that, in terms of $q_{\ep,l} := 1+\theta_{\ep,l}$, \eqref{eq:3.3} is equivalent to
    $$
    \pd_t q_{\ep,l} + \nabla_s \cdot (\widetilde{u}_{\ep,l}'\, q_{\ep,l} ) = 0\,.
    $$

    We also remark that it follows from the fact that $\check{\tau}_{1,\eps,l}$ and $\check{\tau}_{2,\eps,l}$ are not parallel that none of the denominators is $0$.
\end{remark}

\begin{proof}
Let us derive the equations one by one. 

\subsubsection*{First equation} We can differentiate the identity~\eqref{def gamma evolved} with respect to $t$ to get
\begin{align}(0,0,\partial_t\nu_{\ep,l}(s,t))^{\rm T}&=\partial_t\gamma_{\ep,l}(s,t)= \pd_t X_t[\gamma_{\ep,l}^0(\sigma_t(s))]+ {\rm D}X_t[\gamma_{\ep,l}^0(\sigma_t(s))]\, {\rm D}\gamma_{\ep,l}^0(\sigma_t(s))\, \pd_t\sigma_t(s\notag)\\
&=\tul(s,t) + {\rm D} X_t[\gamma_{\ep,l}^0(\sigma_t(s))]\, {\rm D}\gamma_{\ep,l}^0(\sigma_t(s))\, \pd_t\sigma_t(s)\,.\label{E.Ideq1}
\end{align}
Here we have used the flow equation \eqref{E.flowEquation}. Likewise, we can differentiate the identity \eqref{def gamma evolved} with respect to~$s$, obtaining
\begin{equation}\label{E.Ideq2}
\left[
\begin{aligned}
    &\Id \\
     \nabla_s&\nu_{\ep,l}(s,t)
\end{aligned}
\right] = {\rm D}_s\gamma_{\ep,l}(s,t)={\rm D}X_t[\gamma_{\ep,l}^0(\sigma_t(s))]\, {\rm D} \gamma_{\ep,l}^0(\sigma_t(s))\, {\rm D}_s\sigma_t(s)\,,
\end{equation}
where $\Id$ is the $2\times2$ identity matrix. Then, denoting the components of the flow map by $X_t=(X_t',X^3_t)$, and using the first two rows in \eqref{E.Ideq1}--\eqref{E.Ideq2}, we obtain that
\begin{align*}
{\rm D}_s\sigma_t&=\left[{\rm D} X_t'[\gamma_{\ep,l}^0(\sigma_t(s))]{\rm D}\gamma_{\ep,l}^0(\sigma_t(s))\right]^{-1}\,,\\
\pd_t\sigma_t &=-\left[{\rm D} X_t'[\gamma_{\ep,l}^0(\sigma_t(s))]{\rm D}\gamma_{\ep,l}^0(\sigma_t(s))\right]^{-1}\,\tul'= -{\rm D}_s \sigma_t \, \tul'\,.
\end{align*}
The time derivative of $\nu_{\ep,l}$ is therefore 
\begin{align*}
    \pd_t\nu_{\ep,l}(s,t)& =\tul^3(s,t)-\nabla X_t^3[\gamma_{\ep,l}^0(\sigma_t(s))]{\rm D} \gamma_{\ep,l}^0(\sigma_t(s)) {\rm D}_s\si_t(s)\tul'(s,t) \\
    & = \tul^3(s,t)- \nabla_s\nu_{\ep,l}(s,t) \tul'(s,t)\,,
\end{align*}
which is~\eqref{eq:3.2}. To pass to the last identity, we have used the last row in ~\eqref{E.Ideq2}.

\subsubsection*{Second equation} Let us denote by $a,b$ indices that range over $\{1,2\}$, and by $i,j$ indices ranging over~$\{1,2,3\}$. Also, to simplify the notation, we will use here the Einstein summation convention. First, subtracting~\eqref{eq:3.3} for two different values of~$l$ $(l \neq l')$, we get
\begin{equation*}
    \pd_t(\nu_{\ep,l}-\nu_{\ep,l'})+\widetilde{u}^a_{\ep,l}\pd_{s_a}\nu_{\ep,l}-\widetilde{u}^a_{\ep,l'}\pd_{s_a}\nu_{\ep,l'}-\widetilde{u}^3_{\ep,l}+\widetilde{u}^3_{\ep,l'}=0\,.
\end{equation*}
Now, note that
\begin{align*}
\pd_{s_a}\int_{\nu_{\ep,l'}}^{\nu_{\ep,l}}& u^a_\ep(s,r,t) \dd r =u^a_\ep(s,\nu_{\ep,l},t)\pd_{s_a}\nu_{\ep, l}-  u^a_\ep(s,\nu_{\ep,l'},t)\pd_{s_a}\nu_{\ep, l'}+\int_{\nu_{\ep,l'}}^{\nu_{\ep,l}} \pd_{s_a}u^a_\ep(s,r,t)\dd r\\
&= u^a_\ep(s,\nu_{\ep,l},t)\pd_{s_a}\nu_{\ep, l}-  u^a_\ep(s,\nu_{\ep,l'},t)\pd_{s_a}\nu_{\ep, l'}-\int_{\nu_{\ep,l'}}^{\nu_{\ep,l}} \pd_r u^3_\ep(s,r,t)\dd r\\
&= u^a_\ep(s,\nu_{\ep,l},t)\pd_{s_a}\nu_{\ep, l}-  u^a_\ep(s,\nu_{\ep,l'},t)\pd_{s_a}\nu_{\ep, l'}-\widetilde{u}^3_{\ep,l}+\widetilde{u}^3_{\ep,l'}\,,
\end{align*}
where we have used that $u_\ep$ is divergence-free to pass to the second line. Also, by the definition of~$\theta_{\ep,l}$ and \eqref{thetaok}, one can change variables in the integral to arrive at
\[
\int_{\nu_{\ep,l'}}^{\nu_{\ep,l}} u^a_\ep(s,r,t)\dd r=\ep\int_{l'}^l u^a_\ep(s,\nu_{\ep,\lambda},t)\,(1+\theta_{\ep,\lambda}) \dd\lambda,
\]
which results in
\[
\pd_t\left(\frac{\nu_{\ep,l}-\nu_{\ep,l'}}{\ep} \right) + \pd_{s_a}\int_{l'}^l u^a_\ep(s,\nu_{\ep,\lambda},t)\,(1+\theta_{\ep,\lambda}) \dd\lambda=0\,.
\]
Dividing by $l-l'$ and sending $l'\to l$, we get \eqref{eq:3.3}.

\subsubsection*{Third equation}  

It follows from Euler's equation~\eqref{eq 3dEuler} that the time derivative of the vector field $(x,t)\mapsto \om_\ep(X_t(x),t)$ is given by
\begin{align*}
\pd_t[\om_\ep(X_t(x),t)]&= \pd_t\om_\ep(X_t(x),t)+ (u_\ep\cdot \nabla)\om_\ep(X_t(x),t)=(\om_\ep\cdot \nabla) u_\ep(X_t(x),t)\,.
\end{align*}
Also, combining \eqref{E.ansatzomega} and \eqref{E.sgraph}, we get that
\begin{align*}
    \pd_t[\omega_\ep(X_t(\gamma_{\ep,l}^0(s)),t)] =   \pd_t[ \widetilde{\om}_{\ep,l}(\sigma_t^{-1}(s),t)] = \pd_t \widetilde{\om}_{\ep,l}(\sigma_t^{-1}(s),t) + {\rm D}_s \widetilde{\omega}_{\ep,l}(\sigma_t^{-1}(s),t) \tul'(\sigma_t^{-1}(s),t)
\end{align*}
Hence, it follows that
\begin{equation} \label{E.TEconclusion}
\pd_t \widetilde{\om}_{\ep,l}(s,t) + {\rm D}_s  \widetilde{\om}_{\ep,l}(s,t) \tul'(s,t) = {\rm D}u_\ep(\gamma_{\ep,l}(s,t),t)  \widetilde{\om}_{\ep,l}(s,t)\,.
\end{equation}
Taking into account Lemma \ref{tan lemma}, we can replace ${\rm D} u_\ep$ by its tangential projection. This is precisely the role of the matrix $\cA_\ep$ in the statement. In other words,
$$
\cA_\ep \widetilde{\om}_{\ep,l}(s,t) := {\rm pr}[{\rm D}u_\ep](\gamma_{\ep,l}(s,t),t) \widetilde{\om}_{\ep,l}(s,t)\,.
$$
Plugging this definition into \eqref{E.TEconclusion} we get \eqref{eq:3.5}, as desired. 
\end{proof}

As explained in Subsection \ref{S.sketch}, once we have the effective system \eqref{eq:3.2}--\eqref{eq:3.5}, we need to show that we can ``come back'' to \eqref{eq 3dEuler}. In the next lemma we show that,  if we have a solution to \eqref{eq:3.2}--\eqref{eq:3.5} in $(0,T_\star)$, not necessarily coming from a solution to \eqref{eq 3dEuler}, then we can construct a solution to \eqref{eq 3dEuler}, as long as the data remains sufficiently regular and $\sup_{|l| \leq 1}\norm{\theta_{\eps,l}(\cdot,t)}_{L^\infty(\TT^2)}<1$, for all $t \in (0,T_\star)$. 

\begin{lemma}\label{weak sol}
Let $r > 1$ be a non-integer, and let $(\tul(\cdot,t), \nu_{\eps,l}(\cdot,t),\theta_{\eps,l}(\cdot,t),\widetilde{\omega}_{\eps,l}(\cdot,t))$ be $C^r$~functions, continuously differentiable in $t$, uniformly bounded in $l$, and satisfying \eqref{eq:3.2}--\eqref{eq:3.5} on some time interval $(0,T_\star)$. Assume that:
\begin{itemize}
    \item[\rm (i)] The initial data $(\nu_{\eps,l}^0,\,\theta_{\eps,l}^0,\,\widetilde{\omega}_{\eps,l}^0)$ is of the form discussed in Section~\ref{S.setup},
    \item[\rm (ii)] $\sup_{|l| \leq 1}\norm{\widetilde{\omega}_{\eps,l}}_{L^\infty(\T^2 \times (0,T_\star))}<\infty$,
    \item[\rm (iii)] $\sup_{|l| \leq 1}\norm{\theta_{\eps,l}}_{L^\infty(\T^2 \times (0,T_\star))}<1$,
\end{itemize}
and let
\begin{equation} \label{E.omegaeback}
\om_\ep(x,t):=\begin{cases}
    \widetilde{\om}_{\ep,l}(s,t) & \text{if $x=\gamma_{\eps,l}(s,t)$ for some }(l,s)\,,\\
    0 & \text{otherwise}\,.
\end{cases}
\end{equation}
If its associated velocity field $u_\ep:=\BS[\om_\ep]$ satisfies
\begin{equation} \label{E.ueback}
u_\ep(s,\nu_{\ep,l}(s,t),t)= \tul(s,t)\,, \quad \textup{for all }(s,l,t)\in\T^2\times(-1,1)\times(0,T_\star)\,,
\end{equation}
then $\om_\ep\in C_{\rm w}([0,T_\star),C_c^r(\bD))$ is a divergence-free vector field solving \eqref{eq 3dEuler}.
\end{lemma}

\begin{proof} First, observe that, for any test vector field $\Phi \in C(\bD, \RR^3)$,
\begin{equation}
\begin{aligned}
\scalar{\omega_{\ep}(\cdot,t)}{\Phi} & = \int_{\TT^2} \int_{-1}^1 \widetilde{\om}_{\ep,l}(s,t) \cdot \Phi(\gamma_{\ep,l}(s,t)) \,J_{\ep,l}(s,t) \dd l \dd s  \\
& =\int_{\TT^2} \int_{-1}^1 \varpi_{\ep,l}(s,t) \cdot \Phi(\gamma_{\ep,l}(s,t)) \dd l \dd s\,, \quad \textup{for all } t \in [0,T_\star)\,,
\end{aligned} \label{int form varpi}
\end{equation}
where $J_{\ep,l}$ is the Jacobian given in \eqref{E.JacobianRemark}, and $\varpi_{\ep,l}$ was defined in \eqref{def pi}. Hence, we have that $\omega_\ep \in C^1([0,T_\star), \mathcal{D}'(\bD,\RR^3))$. Moreover, it is straightforward to check that $\omega_\ep$ defines a time-dependent vector field on $\bD$ which is bounded and has compact support, and that $\widetilde{\om}_{\ep,l}(s,t) \subset \RR^3$ is tangential to the surface $\ga_{\ep,l}(\T^2,t) \subset \bD$ at the points $\ga_{\ep,l}(s,t) = (s, \nu_{\ep,l}(s,t))$.

Next, assume that $u_\ep := \BS[\omega_\ep]$ satisfies \eqref{E.ueback}. Then, on the one hand, it follows that
\begin{align*}
    & \scalar{\pd_t \om_\ep(\cdot,t)}{\Phi} =  \int_{\TT^2} \int_{-1}^1 \pd_t \varpi_{\ep,l}(s,t) \cdot \Phi(\ga_{\ep,l}(s,t)) \dl \dd s \\
    & \quad + \int_{\TT^2} \int_{-1}^1 \varpi_{\ep,l}(s,t) \cdot \pd_3\Phi(\ga_{\ep,l}(s,t)) \pd_t \nu_{\ep,l}(s,t) \dd l \dd s\,, \quad \textup{for all } t \in (0,T_\star) \textup{ and all } \Phi \in C^1(\bD,\R^3)\,.
\end{align*}
Using the tangentiality property of $\widetilde{\omega}_{\ep,l}$, \eqref{eq:3.2} and \eqref{eff3'}, and integrating by parts, we get
\begin{align*}
    & \scalar{\pd_t \om_\ep(\cdot,t)}{\Phi} = \int_{\TT^2} \int_{-1}^1 \varpi_{\ep,l}(s,t) {\rm D} u_\ep(\ga_{\ep,l}(s,t)) \Phi(\ga_{\ep,l}(s,t)) \dl \dd s \\
    & \quad +  \int_{\TT^2} \int_{-1}^1 \varpi_{\ep,l}(s,t) \cdot {\rm D}\Phi(\ga_{\ep,l}(s,t)) \tul(s,t) \dd l \dd s\,, \quad \textup{for all } t \in (0,T_\star) \textup{ and all } \Phi \in C^1(\bD,\R^3)\,. 
\end{align*}

On the other hand, note that  
\begin{align*}
    & \scalar{\div(\om_\ep(\cdot,t) u_\ep(\cdot,t))}{\Phi} = - \scalar{\om_\ep(\cdot,t)}{[{\rm D}\Phi\, u_\ep] (\cdot,t)} \\
    & \quad = - \int_{\TT^2} \int_{-1}^1 \varpi_{\ep,l}(s,t) \cdot {\rm D}\Phi(\ga_{\ep,l}(s,t)) u_\ep(\ga_{\ep,l}(s,t),t) \dl \dd s\,, 
\end{align*}
and that
\begin{align*}
    & \scalar{{\rm D} u_\ep(\cdot,t) \om_\ep(\cdot,t)}{\Phi} = \scalar{\om_\ep(\cdot,t)}{{\rm D}u_\ep(\cdot,t) \Phi} \\
    & \quad = \int_{\TT^2} \int_{-1}^1 \varpi_{\ep,l}(s,t) \cdot {\rm D}u_\ep(\ga_{\ep,l}(s,t),t) \Phi(\ga_{\ep,l}(s,t)) \dl \dd s \,,
\end{align*}
for all $t \in (0,T_\star)$ and all  $\Phi \in C^1(\bD,\R^3)$.  Combining these identities, we conclude that
$$
\scalar{\pd_t \om_\ep(\cdot,t) + (u_\ep \cdot \nabla) \om_\ep(\cdot,t) - (\omega_\ep \cdot \nabla) u_\ep (\cdot,t)}{\Phi}  = 0 \,,  \quad \textup{for all } t \in (0,T_\star) \textup{ and all } \Phi \in C^1(\bD,\R^3)\,. 
$$

\noindent Now, using the identity above and the fact that $\omega_0$ given by \eqref{def omega0} and $u_\ep(\cdot,t)$ are divergence-free vector fields, the incompressibility of $\omega_\ep(\cdot,t)$ for all $ t \in (0,T_\star)$ immediately follows. Indeed, combining these facts, one can check that $f_\ep := \div \omega_\ep$ is a distributional solution to
$$
\pd_t f_\ep + u_\ep \cdot \nabla f_\ep = 0 \quad \textup{in } \bD \times (0,T_\star)\,, \quad  \quad  f_\ep(\cdot,0) = 0 \quad \textup{in } \bD\,.
$$
Furthermore, since $\widetilde{\omega}_{\ep,l}(s,t)$ is tangential to the surface $\ga_{\ep,l}(\TT^2,t)$ at the points $\ga_{\ep,l}(s,t) = (s,\nu_{\ep,l}(s,t))$, and $\widetilde{\omega}_{\eps,l}(\cdot,t)$ is $C^r$, for some non-integer $r > 1$, by assumption, this (distributional) divergence is a compactly supported bounded function on $\bD$. Also, by standard elliptic regularity, we know that $u_\ep(\cdot,t)$ is log-Lipschitz. Hence, it follows from \cite[Theorem II.2] {diPernaLions} that
$$
\div \omega_\ep(\cdot,t) = 0 \quad \textup{for all } t \in (0,T_\star)\,.
$$
At this point, we conclude that $\om_\ep \in C^1([0,T_\star), L^{\infty}(\bD))$ is a compactly supported distributional solution to \eqref{eq 3dEuler}. 

Finally, since $\om_\ep^0 \in C_c^r(\bD)$, using for instance the weak-strong uniqueness result given in \cite[Corollary 7.18]{Bahouri11}, the result follows from standard elliptic regularity. 
\end{proof}

\section{Local existence for the effective system and proof of Theorem \ref{main thm}} \label{S.localExistence}

This section is mainly devoted to showing that the effective system we have derived in Lemma  \ref{eff system} is locally well-posed, with a lifespan which is independent of $\ep$. Note that, having Lemma \ref{weak sol} at hand, this essentially concludes the proof of Theorem \ref{main thm} (i) and (ii). The proof of this local well-posedness relies on an abstract Cauchy-Kovalevskaya theorem, namely Nishida's theorem. We conclude the section showing the convergence to the vortex sheet. 

\begin{theorem}[Nishida~\cite{Nishida}]\label{Nishida}
Let $  \{Y_\rho\}_{\rho\in[0,\infty)}$ be a scale of Banach spaces, and let $B^R_\rho:=\{w\in Y_\rho: \|w\|_{Y_\rho}<R\}$ denote the ball in~$Y_\rho$ of radius $R$. Assume that, for any $0\leq \rho'<\rho$, $Y_\rho$ is a linear subspace of $ Y_{\rho'}$ and $\| w \|_{Y_{\rho'}} \leq \| w \|_{Y_\rho}$ for all $w\in Y_\rho$. Consider the initial value problem 
\begin{equation}\label{E.Nishida}
    w'(t) = F(w(t), t)\,, \qquad 
    w(0) = w_0\,,
\end{equation}
and assume the following conditions on $F$ and $w_0$,  for  some fixed positive numbers $R ,\delta ,\rho_0 $:
\begin{itemize}
    \item[(i)]  For every $0 < \rho' < \rho < \rho_0$, the map $F: B_\rho^R \times(-\delta,\delta)\to Y_{\rho'}, \ (w, t) \mapsto F(w, t),$ is continuous.
    \item[(ii)] For any $0<\rho' < \rho < \rho_0$ and all $w^1, w^2 \in B_\rho^R$, and for any $t\in (-\delta,\delta)$, $F$ satisfies
    \begin{equation*}
        \| F(w^1, t) - F(w^2, t) \|_{Y_{\rho'}} \leq \frac{C}{\rho - \rho'}\, \| w^1 - w^2 \|_{Y_\rho}\,,
    \end{equation*}
    where $C > 0$ is a fixed constant independent of $t, w^1, w^2, \rho$ or $\rho'$.
    \item[(iii)] $w^0 \in B_{\rho_0}^{R_0}$ for some $0 < R_0 < R$ and $F(w^0, \cdot ): (- \delta,\delta)\to Y_\rho$, $\,t \mapsto F(w^0,t)$, is a continuous function for every $0<\rho < \rho_0$. Moreover,  it satisfies, with a fixed constant $K > 0$,
    \begin{equation*}
        \sup_{|t|<\delta}\| F(w^0, t) \|_{Y_\rho} \leq \frac{K}{\rho_0 - \rho}.
    \end{equation*}
\end{itemize}
Under these hypotheses, there is a positive constant $\delta_0$ (depending only on $R_0$, $R$, $C$ and $K$) such that there exists a unique function $w(t)$ which, for every $0\leq \rho < \rho_0$, is a continuously differentiable function of $t\in (-(\rho_0 - \rho)\delta_0,\,(\rho_0 - \rho)\delta_0)$ with values in $B_\rho^R$ that satisfies the initial value problem~\eqref{E.Nishida}.
\end{theorem}

We aim to show that this theorem is applicable to the effective system \eqref{eq:3.2}--\eqref{eq:3.5}. To that end, we first set 
$$
\nu_\ep[l,s](t):= \nu_{\ep,l}(s,t)\,, \quad \theta_{\ep}[l,s](t) := \theta_{\ep,l}(s,t)\,, \quad \textup{and} \quad \varpi_{\ep}[l,s](t):= \varpi_{\ep,l}(s,t):=\eps(1+\theta_{\eps,l})\widetilde{\omega}_{\eps,l}(s,t)\,,
$$
and define
\begin{equation} \label{E.defw}
\begin{aligned}
    w[l,s](t) & = \big(w_1[l,s](t),\, w_2[l,s](t),\, w_3[l,s](t),\, w_4[l,s](t),\, w_{5-7}[l,s](t)\big) \\
    &:=\big(\nu_\ep[l,s](t),\, \pd_{s_1} \nu_\ep[l,s](t),\, \pd_{s_2} \nu_\ep[l,s](t),\, 1+\theta_\ep[l,s](t),\, \varpi_{\ep}[l,s](t) \big)\,.
\end{aligned}
\end{equation}
Here, and throughout the paper, we use the shortened notation $w_{5-7} = (w_5,w_6,w_7)$. 

In order to use Theorem \ref{Nishida} to solve the effective system, we use the following scale of Banach spaces for $k \in \NN$ and $\rho\geq 0$:
$$
\cY_\rho^k:=\bigg\{f: [-1,1] \times \T^2_\rho \rightarrow \C^k\bigg|\ l\mapsto f_j(l, \cdot)\in L^\infty([-1,1],\cX_\rho)\, \textup{ \& }\sup_{|l| \leq 1} \|f_j(l, \cdot)\|_\rho 
 <\infty\,,\ 1\leq j\leq k\bigg\}\,,
$$
endowed with the norm
$$
\|f\|_{\cY_\rho^k} := \sum_{j=1}^k \, \sup_{|l| \leq 1} \|f_j(l, \cdot)\|_\rho \,.
$$
We understand $L^\infty([-1,1],\cX_\rho)$ as the space of functions which are measurable with respect to the joint variable in $[-1,1] \times \TT_\rho^2 $ and for which $\sup_{|l| \leq 1} \norm{f(l,\cdot)}_{\cX_\rho}<\infty$. Note that this is different from the Bochner space due to the lack of separability. However, this kind of measurability is enough to ensure that all the integrals in the paper exist.

Of course, the variables $(w_1,\dots, w_7)$ in \eqref{E.defw} are not independent, and we cannot expect the effective system to behave well for arbitrary data in $\cY^7_\rho$. We introduce the set of admissible data 
\begin{align*}
&\mathcal{Z}_\rho:=\Bigg\{ w\in\cY^7_\rho\,\bigg|\, \nabla_s w_1[l,\cdot]=(w_2[l,\cdot],w_3[l,\cdot]) \textup{ and } w_1[l,\cdot]-w_1[\ell,\cdot]=\eps \int_\ell^l w_4[\mu,\cdot] \dd \mu,\  \forall\,  l,\ell \Bigg\}\,.
\end{align*}
Note that for every $\rho \geq 0$, $\cZ_\rho$ is a Banach subspace of $\cY_\rho^7$ endowed with the same norm as $\cY_\rho^7$. We will also use the semi-norm
\begin{align}
\norm{w}_{\cZ_\rho^{\rm Im}}:=\sum_{j=2}^4\sup_{|l|\leq 1}\sup_{|\beta|\leq \rho}\norm{\Im w_j(l,\cdot+i\beta)}_{L^\infty(\T^2)}\,.\label{we norm 0}
\end{align}

At this point, our first goal is to reformulate the effective system \eqref{eq:3.2}--\eqref{eq:3.5} as an initial value problem of the form \eqref{E.Nishida} in the scale of Banach spaces $\cZ_\rho$. To that end, using the notation introduced in Section \ref{S.effective}, we decompose the velocity $\tul(s,t)$ (see \eqref{E.tul}) into the contributions of the individual sheets in \eqref{E.ansatzomega}. First, note that
\begin{align*}
    \tul(s,t) & = \BS [\omega_\ep](\gamma_{\ep,l}(s,t),t) = \frac{1}{4\pi} \sum_{k \in \ZZ^2 \times \{0\}} \int_{[-\frac12,\frac12]^2 \times \R}  \omega_\ep(y,t) \times \frac{\gamma_{\ep,l}(s,t) - y + k}{|\gamma_{\ep,l}(s,t) - y + k|^3}\dd y \\
    & = \frac{1}{4\pi}  \sum_{k \in \ZZ^2 \times \{0\}} \int_{-1}^1 \int_{[-\frac12,\frac12]^2}  \varpi_{\ep,\ell}(s-\varsigma,t) \times \frac{\gamma_{\ep,l}(s,t) - \gamma_{\ep,\ell}(s-\varsigma,t) + k}{|\gamma_{\ep,l}(s,t) - \gamma_{\ep,\ell}(s-\varsigma,t) + k|^3}  \dd \varsigma \dd\ell \\
    & = \frac{1}{4\pi} \int_{-1}^1 \int_{[-\frac12,\frac12]^2} \varpi_{\ep,\ell}(s-\varsigma,t) \times \frac{\gamma_{\ep,l}(s,t) - \gamma_{\ep,\ell}(s-\varsigma,t)}{|\gamma_{\ep,l}(s,t) - \gamma_{\ep,\ell}(s-\varsigma,t)|^3} \dd \varsigma \dd\ell \\ 
    & \quad + \sum_{k \in (\ZZ^2 \times \{0\}) \setminus \{0\}} \int_{-1}^1 \int_{[-\frac12,\frac12]^2}  \varpi_{\ep,\ell}(s-\varsigma,t) \times \frac{\gamma_{\ep,l}(s,t) - \gamma_{\ep,\ell}(s-\varsigma,t) + k}{|\gamma_{\ep,l}(s,t) - \gamma_{\ep,\ell}(s-\varsigma,t) + k|^3}  \dd \varsigma \dd\ell\,.
\end{align*}
Also, let us denote by $\phi$ a smooth, non-negative, radially symmetric cutoff function such that
\begin{equation} \label{E.phicutoff}
\phi \equiv 1 \quad \textup{in } B_{\frac18}(0)\,, \quad \textup{ and } \quad \phi \equiv 0 \quad \textup{in } [-\tfrac12,\tfrac12]^2 \setminus B_{\frac14}(0)\,.
\end{equation}
Then, we have that
\begin{align*}
    \tul(s,t)  & = \frac{1}{4\pi} \int_{-1}^1 \int_{[-\frac12,\frac12]^2} \phi(\varsigma) \varpi_{\ep,\ell}(s-\varsigma,t) \times \frac{\gamma_{\ep,l}(s,t) - \gamma_{\ep,\ell}(s-\varsigma,t)}{|\gamma_{\ep,l}(s,t) - \gamma_{\ep,\ell}(s-\varsigma,t)|^3}  \dd \varsigma \dd\ell \\ 
    & \quad + 
    \frac{1}{4\pi} \int_{-1}^1 \int_{[-\frac12,\frac12]^2} (1-\phi(\varsigma)) \varpi_{\ep,\ell}(s-\varsigma,t) \times \frac{\gamma_{\ep,l}(s,t) - \gamma_{\ep,\ell}(s-\varsigma,t)}{|\gamma_{\ep,l}(s,t)  - \gamma_{\ep,\ell}(s-\varsigma,t)|^3}  \dd \varsigma \dd\ell\\
    & \quad + \sum_{k \in (\ZZ^2 \times \{0\}) \setminus \{0\}} \int_{-1}^1 \int_{[-\frac12,\frac12]^2}  \varpi_{\ep,\ell}(s-\varsigma,t) \times \frac{\gamma_{\ep,l}(s,t) - \gamma_{\ep,\ell}(s-\varsigma,t) + k}{|\gamma_{\ep,l}(s,t) - \gamma_{\ep,\ell}(s-\varsigma,t) + k|^3} \dd \varsigma \dd\ell\,.
\end{align*}
Also, we set
\begin{equation}
    K_\ep^{c}[\nu,g,\varpi](s) := \frac{1}{4\pi} \int_{[-\frac12,\frac12]^2} \phi(\varsigma)\varpi(s-\varsigma) \times\frac{\Xi^\ep[\nu,g](s,\varsigma)}{|\Xi^\ep[\nu,g](s,\varsigma)|^3} \dd\varsigma\,,
\end{equation}
where 
\begin{gather*}
\Xi^\ep[\nu,g](s,\varsigma):= \varsigma
    +  [\nu(s)-\ep g(s)]\,e_3-\nu(s-\varsigma)e_3\,, \quad \textup{for } s \in \TT^2 \textup{ and } \varsigma \in \Big[-\frac12,\frac12\Big]^2\,,
\end{gather*}
and 
\begin{equation}\begin{aligned}
K_\eps^f[v,w](l,s):=\ & \frac{1}{4\pi}\int_{-1}^1\int_{-[\frac{1}{2},\frac{1}{2}]^2} \big(1-\phi(\varsigma)\big)w(s-\varsigma,\ell) \times \frac{\varsigma+v(s,l)e_3-v(s-\varsigma,\ell)e_3}{|\varsigma+v(s,l)e_3-v(s-\varsigma,\ell)e_3|^3} \dd\varsigma\dd\ell\\
&+ \frac{1}{8\pi}\sum_{k\in \Z^2 \setminus \{0\}}\int_{[-\frac{1}{2},\frac{1}{2}]^2} w(s-\varsigma,\ell)\times \bigg(\frac{\varsigma-k+v(s,l)e_3-v(s-\varsigma,\ell)e_3}{|\varsigma-k+v(s,l)e_3-v(s-\varsigma,\ell)e_3|^3}\\
&+\frac{\varsigma+k+v(s,l)e_3-v(s-\varsigma,\ell)e_3}{|\varsigma+k+v(s,l)e_3-v(s-\varsigma,\ell)e_3|^3}\bigg)\dd\varsigma\dd\ell \label{def kfar}\,.
\end{aligned}\end{equation}
Having this notation at hand, we can finally decompose $\tul(s,t)$ as 
\begin{equation} \label{E.tulDecomposition}
    \tul(s,t) = \int_{-1}^1 K_\ep^c\Big[\nu_{\ep,\ell},\,  \int_{l}^\ell (1+\theta_{\ep,\mu} ) \dd \mu, \, \varpi_{\ep,\ell}\Big](s,t) \dd \ell + K_\ep^f[\nu_\ep, \varpi_{\ep}](s,l)\,.
\end{equation}

Now, extending the action of these operators to functions that depend on $z \in \T_\rho^2$ in the obvious way, we set
\begin{align}
U_\ep(w)(l,z):=  \int_{-1}^1 K_\ep^c\Big[w_1[\ell,\cdot],\,  \int_{l}^\ell w_4[\mu,\cdot] \dd \mu, w_{5-7}[\ell,\cdot] \Big](z) \dd\ell + K_\ep^f[w_1, w_{5-7}](l,z)\,. \label{velo split}
\end{align}
Note that we will use superscripts to denote the components of this velocity. In addition, we denote by $U_\ep' = \Pi(U_\ep)$ its projection to the horizontal plane. We then consider the initial value problem
\begin{equation} \label{E.initialValueProblem}
w'(t) = F(w(t))\,, \quad 
    w(0) = w^0\,,
\end{equation}
where
\begin{equation}
\begin{aligned}
    & F_1(w) := U_\ep^{3}(w) - U_\ep^{1}(w) w_2-U_\eps^2(w)w_3\,, \\
    & F_2(w) := \pd_{s_1} \big(U_\ep^{3}(w) - U_\ep^{1}(w) w_2-U_\eps^2(w)w_3 \big) \,,  \\
    & F_3(w) := \pd_{s_2} \big(U_\ep^{3}(w) - U_\ep^{1}(w) w_2-U_\eps^2(w)w_3 \big) \,,  \\
    & F_4(w) :=  - \nabla_s\cdot\Big(U_\ep'(w) w_4\Big)\,,\\
    & F_{5-7}(w) := - \pd_{s_1} \Big( U_\ep^{1}(w) w_{5-7} \Big)-\pd_{s_2} \Big( U_\ep^{2}(w) w_{5-7} \Big)+ A_\ep(w) w_{5-7}\,.\label{def F}
\end{aligned}
\end{equation}
and
$$
w^0[l,s] = (\nu_{\ep,l}^0(s),\, \pd_{s_1} \nu_{\ep,l}^0(s),\, \pd_{s_2}\nu_{\ep,l}^0(s),\, 1+\theta_{\ep,l}^0(s),\, \varpi_{\ep,l}^0(s))\,.
$$
Here, as in Lemma \ref{eff system}, 
\begin{align}
A_\ep(w) w_{5-7} = (w_{5-7} \cdot \tau_1) D_{\tau_1} U_\ep + (w_{5-7} \cdot \tau_2) D_{\tau_2} U_\ep\,,\label{def Aeps}
\end{align}
with
\begin{equation}\label{def taus}
\hat{\tau}_1 := (1,0,w_2)^{\rm T}\,, \quad \tau_1 := \frac{\hat{\tau}_1}{|\hat{\tau}_1| } \,, \quad \hat{\tau}_2 = (0,1,w_3)^{\rm T}\,, \quad\textup{and}\quad \tau_2 := \frac{\hat{\tau}_2|\hat{\tau}_1|^2- \hat{\tau}_1 (\hat{\tau}_1\cdot\hat{\tau}_2)}{|\hat{\tau}_2|\hat{\tau}_1|^2- \hat{\tau}_1 (\hat{\tau}_1\cdot\hat{\tau}_2)|}\,.
\end{equation}
and
$$
D_{\tau_1} U_\ep := \frac{1}{|\hat{\tau_1}|} \pd_{s_1} U_\ep \,, \quad D_{\tau_2} U_\ep := \frac{1}{|\hat{\tau}_2 - (\hat{\tau}_2 \cdot \tau_1) \tau_1|} \left( \pd_{s_2} U_\ep - \frac{\hat{\tau_2}\cdot \tau_1}{|\hat{\tau}_1|} \pd_{s_1} U_\ep \right)\,.
$$

We should first stress that \eqref{E.initialValueProblem} is essentially a reformulation of the effective system we derived in Lemma \ref{eff system}, now in terms of the new unknown $w$. Likewise, note that, if $w \in \cZ_\rho$ and $F(w) \in \cY^7_{\rho'}$ for some $0 < \rho' < \rho$, then $F(w)\in \mathcal{Z}_{\rho'}$. Indeed, this is trivial for the coupling between $w_1$ and  $w_2,w_3$, and for the coupling between $w_1$ and $w_4$, it follows from the derivation of \eqref{eq:3.3}. 

Having this functional setting at hand, in the next results, whose proofs are postponed to Section \ref{S.Kernel}, we prove the required estimates to verify that \eqref{E.initialValueProblem} fits the framework of Theorem \ref{Nishida}.

\begin{proposition} \label{main est}
For every $0 \leq \rho \leq \rho_0$, if $w^1, w^2 \in \cZ_\rho$ satisfy 
\begin{equation} \label{E.cC1} 
    \mathcal{C}_1 := \max \left\{ 1,\ \norm{w^1}_{\cY_\rho^7},\ \norm{w^2}_{\cY_\rho^7} \right\} < \infty\,, 
\end{equation}
as well as
\begin{equation} \label{E.smallness main est} 
    \max\left\{\norm{w^1}_{\cZ_\rho^{\Im}},\,\norm{w^2}_{\cZ_\rho^{\Im}}\right\} \leq \frac{1}{300 \mathcal{C}_1^4}\,,
\end{equation}
and    
\begin{equation} \label{E.smallness main est 2} 
     \sup_{|l| \leq 1} \sup_{|\be| \leq \rho} \max\left\{\norm{\Re w_4^1[l, \cdot + i\be]-1}_{L^\infty(\TT^2)},\, \norm{\Re w_4^2[l, \cdot + i\be]-1}_{L^\infty(\TT^2)}\right\} \leq \frac23\,,
\end{equation}
then
\begin{align}\label{bd K}
    \sup_{l,\, \ell \in [-1,1]} \, \bigg\|K_{\ep}^c\bigg[w_1^i[\ell,\cdot],\,\int_l^\ell w_4^i[\mu,\cdot] \dd \mu,\,w_{5-7}^i[\ell,\cdot]\bigg](\cdot) \bigg\|_{\cX_\rho^3}
    \lesssim 1 \,, \quad \textup{\textit{for} } i \in \{1,2\}\,,
\end{align}
and
\begin{align}\begin{aligned} \label{bd K lip}
     \sup_{l,\, \ell \in [-1,1]}  \, \bigg\| & K_{\ep}^c\bigg[w_1^1[\ell,\cdot],\,\int_l^\ell w_4^1[\mu,\cdot] \dd \mu,\,w_{5-7}^1[\ell,\cdot]\bigg](\cdot)\\
     &  -K_{\ep}^c\bigg[w_1^2[\ell,\cdot],\,\int_l^\ell w_4^2[\mu,\cdot] \dd \mu,\,w_{5-7}^2[\ell,\cdot]\bigg](\cdot) \bigg\|_{\cX_\rho^3}
\lesssim \norm{w^1-w^2}_{\cY_\rho^7} \,, 
\end{aligned}\end{align}
with implicit constants that depend only on $\cC_1$ but not on $\rho$ or $\ep$. 
\end{proposition}

\begin{proposition} \label{est kfar}
    For every $0 \leq \rho \leq \rho_0$, if $w^1, w^2 \in \cZ_\rho$ satisfy \eqref{E.cC1} and
    \begin{align}
    \sup_{|l| \leq 1} \sup_{|\beta|\leq \rho} \max \left\{ \norm{\Im w_1^1[l,\cdot+i\beta]}_{L^\infty(\TT^2)}, \ \norm{\Im w_1^2[l,\cdot+i\beta]}_{L^\infty(\TT^2)}  \right\} \leq \frac{1}{8}\,,\label{im small2}
    \end{align}
then
    \begin{align}\label{bd Kfar}
\big\|K_{\ep}^f\big[w_1^i,w_{5-7}^i\big] \big\|_{\cY_\rho^3} \lesssim 1\,,  \quad \textup{\textit{for} } i \in \{1,2\}\,,
\end{align}
and
    \begin{align}\label{bd Kfar lip}
 \big\|K_{\ep}^f\big[w_1^1,w_{5-7}^1\big]-K_{\ep}^f\big[w_1^2,w_{5-7}^2\big]\big\|_{\cY_\rho^3} \lesssim \norm{w^1-w^2}_{\cY_\rho^7}\,,  
\end{align}
with implicit constants that depend only on $\mathcal{C}_1$ given in \eqref{E.cC1} but not on $\rho$ or $\eps$. 
\end{proposition}

\medbreak

By the next result, $F$ fulfills the assumptions in Theorem \ref{Nishida} if the smallness assumptions in the two previous propositions hold. The proof of this result can be found in Section \ref{S.prop44}.

\begin{proposition}
    \label{bd F}
Assume that, for some $0 < \rho \leq \rho_0$, $w^1,w^2\in \cZ_\rho$ satisfy \eqref{E.cC1}-\eqref{E.smallness main est 2} and \eqref{im small2}. Then, it follows that
\begin{align*}
&\norm{F(w^i)}_{\cY_{\rho'}^7}\lesssim \frac{1}{\rho-\rho'}\,, \quad \textup{\it for } i \in \{1,2\}\,,\\
\end{align*}
and
$$
\norm{F(w^1)-F(w^2)}_{\cY_{\rho'}^7}\lesssim \frac{1}{\rho-\rho'} \norm{w^1-w^2}_{\cY_\rho^7}\,,
$$
uniformly in $0 \leq \rho' < \rho \leq \rho_0$. Furthermore, it holds that 
\begin{align}
\norm{U_\eps(w^i)}_{L^\infty(\TT^2)}\lesssim 1\,, \quad \textup{for } i \in \{1,2\}\,. \label{u bd}
\end{align}
Here, $F$ is as in \eqref{def F}, and the implicit constants depend on $\mathcal{C}_1$ given in \eqref{E.cC1} but not on $\rho,\rho'$ or $\ep$. 
\end{proposition}

\subsection{Proof of Theorem \ref{main thm}}

We now have all the ingredients to prove Theorem \ref{main thm}. 

Having Proposition \ref{bd F} at hand, the proof of the existence part of Theorem \ref{main thm} is straightforward.

\begin{proof}[Proof of Theorem \ref{main thm} (i) and (ii)] We set
\begin{equation} \label{E.changetov}
v := w - w^0\,,
\end{equation}
and consider the initial value problem
\begin{equation} \label{E.ODEv}
\pd_t v = F (v+w^0) =: G(v)\,, \quad v(0) = 0\,.
\end{equation}
Moreover, we observe that for all $R \leq 1$, if $v \in B_\rho^R$,  then, with $C_0$ as defined in Theorem \ref{main thm}
$$
\|w\|_{\cY_\rho^7} = \|v+w^0\|_{\cY_\rho^7} < R + \|w^0\|_{\cY_\rho^7} < 4(C_0 + 1)\,.
$$
Hence, we choose and fix $\mathcal{C}_1 = 4(C_0 +1)$ in \eqref{E.cC1}.

Having this choice of $\cC_1$ at hand, one can easily check that \eqref{E.smallness main est}, \eqref{E.smallness main est 2} and \eqref{im small2} hold provided that $v^1, v^2 \in B_\rho^R$ for some 
$$
R \leq \frac{1}{600 \cC_1^4}\,,
$$
and that, in \eqref{im smal}, we choose 
$$
B \leq \frac{1}{600 \cC_1^4}\,.
$$

Since $v(0) = 0$, taking into account Proposition \ref{bd F}, we infer that Theorem \ref{Nishida} can be applied to \eqref{E.ODEv}. Undoing the change of unknown in \eqref{E.changetov}, we get the existence result for \eqref{E.initialValueProblem} we were looking for, with a lifespan depending only on $C_0 > 0$, and not on $\ep$. At this point, using Lemma \ref{weak sol}, we get the existence part of Theorem \ref{main thm}, namely Theorem \ref{main thm} (i).

Concerning Theorem \ref{main thm} (ii), it also follows from the construction. Indeed, it follows from \eqref{E.omegaeback} and the fact that
$$\norm{\gamma_{\eps,1}-\gamma_{\eps,-1}}_{L^\infty(\T^2)}\leq \eps\norm{\int_{-1}^1(1+\theta_{\eps,\mu})\dd\mu}_{L^\infty(\T^2)}\lesssim \eps\,.$$ 
Note that in the last inequality we are using \eqref{E.smallness main est 2}. 
\end{proof}

Finally, we prove the distributional convergence of the constructed solutions to the solution to the Birkhoff-Rott equation. 

\begin{proof}[Proof of Theorem \ref{main thm} (iii)]
Let $\delta > 0$ be fixed but arbitrary. Combining Lemma \ref{lem lsc} with the Arzel\`a-Ascoli theorem, we get that, up to a subsequence, there exists $h \in \cX_{\rho'}$ such that
$$
\nu_{\ep,0} \to h \quad \textup{in } L^{\infty}([0,T-\delta], \cX_{\rho'})\,, \quad \textup{as } \ep \to 0^{+}\,,
$$
where $0 < \rho' < \rho_0$ depends on $\delta$. Arguing similarly, we get $\varpi \in \cX_{\rho'}^3$ such that, up to a subsequence,
$$
\int_{-1}^1 \varpi_{\ep,l} \dd l \to \varpi \quad  \textup{in } L^{\infty}([0,T-\delta], \cX_{\rho'})\,, \quad \textup{as } \ep \to 0^{+}\,.
$$
Note as well that, for all $l \in [-1,1]$ we have that, up to a subsequence, 
$$
\nu_{\ep,l} \to h \quad \textup{in } L^{\infty}([0,T-\delta], \cX_{\rho'})\,, \quad \textup{as } \ep \to 0^{+}\,.
$$
This immediately follows from the following chain of inequalities:
$$
\|\nu_{\ep,l}(\cdot,t) - \nu_{\ep,0}(\cdot,t)\|_{\cX_{\rho'}} \leq \ep \norm{\int_0^l (1+\theta_{\eps,\mu}(\cdot,t))\dd\mu}_{\cX_{\rho'}}\lesssim \eps\,, \quad \textup{for all } t \in [0,T-\delta]\,.
$$

 At this point, taking $T > 0$ slightly smaller, we get convergence up to a subsequence on the full interval and conclude that Theorem \ref{main thm} (iii) (a) and (b) hold up to a subsequence, once we have shown that $\varpi,h$ are indeed the solution to the Birkhoff-Rott equations.  

Setting $\ga_0(s,t) = (s, h(s,t))$, and
 $$
 \scalar{\omega(\cdot,t)}{\Phi} := \int_{\T^2} \varpi(s,t) \cdot \Phi(\gamma_0(s,t)) \dd s \,, 
 $$
 one can check (see \eqref{int form varpi}) that $\om(\cdot,t)$ is a weak-$*$ limit of $\om_\ep(\cdot,t)$ in the space of bounded $\R^3$-valued Radon measures. Moreover, it is a weak solution to 3D Euler supported on $\gamma_0(\cdot,t)$. Indeed, this can be seen as follows: since $(\om_\ep(\cdot,t))_\ep$ is bounded in $L^1(\bD)$, and the Biot-Savart operator maps $L^1(\bD)$ compactly to $L_{loc}^1(\bD)$, we get that $(u_\ep(\cdot,t))_\ep$ converges strongly in $L_{loc}^1(\bD)$. Moreover, $u_\ep(\cdot,t)$ is uniformly bounded on $\supp \omega_\eps(\cdot,t)$ by \eqref{u bd}. In $ \bD \setminus \supp \omega_\eps(\cdot,t)$, we have that $u_\ep(\cdot,t)$ is div- and curl-free (therefore in particular harmonic), and decaying at $\infty$. Hence, it follows from the maximum principle that each component of $u_\ep(\cdot,t)$ attains its maximum modulus on $\supp\omega_\eps (\cdot,t)$. Then, we have that $(u_\ep(\cdot,t))_\ep$ is bounded in $L^{\infty}(\bD)$. At this point, by interpolation, we can pass to the limit in the weak formulation of 3D Euler and get that $\omega(\cdot,t)$ is a weak solution to 3D Euler supported on $\gamma_0(\cdot,t)$.  By the uniqueness of vortex sheets in analytic spaces \cite{SS}, we then conclude that $\omega = \omega^{\rm sing}$, and the result follows. Note that the convergences along the whole sequences, instead of subsequences, follow from the uniqueness of the limit. 
\end{proof}

\section{Kernel estimates} \label{S.Kernel}
This section is devoted to the proofs of the Propositions \ref{main est} and \ref{est kfar}. We split the proof into several parts, but first, we introduce some notation. For $z=(z_1,z_2,z_3)\in\CC^3$, we denote its modulus by
\[
|z|:=\sqrt{|z_1|^2+|z_2|^2+|z_3|^2}\in[0,\infty)\,.
\]
Also, when $\Re\{ z_1^2+z_2^2+z_3^2\}>0$ (as will be the case throughout the paper), we use the main branch of the square root on $\C\backslash (-\infty,0)$ to define
\begin{align} \label{E.complexificationModulus}
&|z|_{\C}:=\sqrt{z_1^2+z_2^2+z_3^2}\in \C\,,
\end{align}
and denote the corresponding quadratic form by \begin{align}
\scalar{z}{z'}_\C:=z_1z_1'+z_2z_2'+z_3z_3'\in \C\,.
\end{align}
Let us stress that \eqref{E.complexificationModulus} is not a norm. Moreover, a direct calculation shows that \begin{align} \label{E.complex-modulus}
||z|_{\C}|\leq |z|\,.
\end{align}
For further reference, we also note that, whenever $\min\{\Re |z|_{\C},\Re |z'|_{\C}\}>0$, we can estimate \begin{align}
||z|_{\C}-|z'|_{\C}|= \frac{\left||z|_{\C}^2-|z'|_{\C}^2\right|}{||z|_{\C}+|z'|_{\C}|}=\frac{\left|\scalar{z-z'}{z+z'}_{_{\C}}\right|}{||z|_{\C}+|z'|_{\C}|}\lesssim |z-z'|\frac{|z|+|z'|}{\min(\Re |z|_{\C},\Re |z'|_{\C})}\,, \label{bas est}
\end{align}
and \begin{align}
||z|_{\C}^3-|z'|_{\C}^3|=\frac{\left|\scalar{z-z'}{z+z'}_{\C}(|z|_{\C}^2+|z|_{\C}|z'|_{\C}+|z'|_{\C}^2)\right|}{||z|_{\C}+|z'|_{\C}|}\lesssim |z-z'|\frac{|z|^3+|z'|^3}{\min(\Re |z|_{\C},\Re |z'|_{\C})}\,.\label{bas est2}
\end{align}

Having this notation at hand, in the next subsection, we prove several technical lemmas which will be key to prove Propositions \ref{main est} and \ref{est kfar}.

\subsection{Preliminary lemmas} \label{S.preliminaryLemmas} We first provide a technical result whose proof is direct calculus.  

\begin{lemma}\label{Lem double diff} $ $
\begin{itemize}
\item[a)] Let $p,q,r,t\in \C^3 \setminus \{0\}$ be such that $\min\{\Re|p|_{\C}, \Re|q|_{\C},\Re|r|_{\C},\Re|t|_{\C}\}>0$, then \begin{equation}\begin{aligned}
&\left|\frac{p}{|p|_{\C}^3}-\frac{q}{|q|_{\C}^3}-\frac{r}{|r|_{\C}^3}+\frac{t}{|t|_{\C}^3}-\frac{p-q-r+t}{|p|_{\C}^3}-\frac{3p\scalar{p}{q-p+r-t}_{\C}}{|p|_{\C}^5}\right|\\
& \qquad \lesssim \frac{\max\{|p|,|q|,|r|,|t|\}^7}{\min\{\Re|p|_{\C},\Re|q|_{\C},\Re|r|_{\C},\Re|t|_{\C}\}^{11}}\big(|p-q|+|r-t|\big)\big(|p-r|+|q-t|\big)\,.\label{dd est}
\end{aligned}\end{equation}

\item[b)] Also, let $p',q',r',t'\in \C^3 \setminus \{0\}$ be such that $\min\{\Re|p'|_{\C}, \Re|q'|_{\C},\Re|r'|_{\C},\Re|t'|_{\C}\}>0$, then \begin{equation}\begin{aligned}
&\Bigg|\bigg(\frac{p}{|p|_{\C}^3}-\frac{q}{|q|_{\C}^3}-\frac{r}{|r|_{\C}^3}+\frac{t}{|t|_{\C}^3}-\frac{p-q-r+t}{|p|_{\C}^3}-\frac{3p\scalar{p}{q-p+r-t}_{\C}}{|p|_{\C}^5}\bigg)\\
&\quad-\bigg(\frac{p'}{|p'|_{\C}^3}-\frac{q'}{|q'|_{\C}^3}-\frac{r'}{|r'|_{\C}^3}+\frac{t'}{|t'|_{\C}^3}-\frac{p'-q'-r'+t'}{|p'|_{\C}^3}-\frac{3p'\scalar{p'}{q'-p'+r'-t'}_{\C}}{|p'|_{\C}^5}\bigg)\Bigg|\\
&  \lesssim \frac{\max\{|p|,|q|,|r|,|t|,|p'|,|q'|,|r'|,|t'|\}^{13}}{\min\{\Re|p|_{\C},\Re|q|_{\C},\Re|r|_{\C},\Re|t|_{\C},\Re|p'|_{\C},\Re|q'|_{\C},\Re|r'|_{\C},\Re|t'|_{\C}\}^{17}}\\
&  \quad \boldsymbol{\cdot}\Big(\big(|p-p'-q+q'|+|r-r'-t+t'|\big)\big(|p-r|+|q-t|\big)\\
& \qquad  +\big(|p-q|+|r-t|\big)\big(|p-p'-r+r'|+|q-q'-t+t'|\big)\Big)\\
& \quad+\frac{\max\{|p|,|q|,|r|,|t|,|p'|,|q'|,|r'|,|t'|\}^{12}}{\min\{\Re|p|_{\C},\Re|q|_{\C},\Re|r|_{\C},\Re|t|_{\C},\Re|p'|_{\C},\Re|q'|_{\C},\Re|r'|_{\C},\Re|t'|_{\C}\}^{17}}\\
& \qquad\boldsymbol{\cdot} \big(|p-p'|+|q-q'|+|r-r'|+|t-t'|\big)\big(|p-q|+|r-t|\big)\big(|p-r|+|q-t|\big)\,.\label{td est}
\end{aligned}\end{equation}
\end{itemize}
\end{lemma}

\begin{proof}
\textbf{a)} First of all, note that
\begin{equation} \label{2 dd est}
\begin{aligned}
&\frac{p}{|p|_{\C}^3}-\frac{q}{|q|_{\C}^3}-\frac{r}{|r|_{\C}^3}+\frac{t}{|t|_{\C}^3}=\frac{p-q}{|p|_{\C}^3}-\frac{r-t}{|r|_{\C}^3}+q\frac{|q|_{\C}^3-|p|_{\C}^3}{|p|_{\C}^3|q|_{\C}^3}-t\frac{|t|_{\C}^3-|r|_{\C}^3}{|r|_{\C}^3|t|_{\C}^3}\\
& \quad =\frac{p-q-r+t}{|p|_{\C}^3}-(r-t)\frac{|p|_{\C}^3-|r|_{\C}^3}{|p|_{\C}^3|r|_{\C}^3}+q\frac{|q|_\C^3-|p|_{\C}^3}{|p|_\C^3|q|_{\C}^3}-t\frac{|t|_\C^3-|r|_{\C}^3}{|t|_\C^3|r|_{\C}^3}\,.
\end{aligned}
\end{equation}
Using \eqref{bas est2}, the second term on the right-hand side can be estimated by the right-hand side in \eqref{dd est}.

Likewise, we can further rewrite the last term as \begin{align*}
t\frac{|t|_\C^3-|r|_{\C}^3}{|t|_\C^3|r|_{\C}^3}=q\frac{|t|_\C^3-|r|_{\C}^3}{|p|_\C^3|q|_{\C}^3}+(t-q)\frac{|t|_\C^3-|r|_{\C}^3}{|t|_\C^3|r|_{\C}^3}+q(|t|_\C^3-|r|_{\C}^3)\left(\frac{1}{|r|^3_\C |t|^3_\C}  - \frac{1}{|p|^3_\C |q|^3_\C}\right)\,,
\end{align*}
and stress that, again using \eqref{bas est2}, the last two terms on the right-hand side can be estimated by the right-hand side in \eqref{dd est}. 

Next, observe that 
\begin{align*}
& q\frac{|q|_\C^3-|p|_{\C}^3}{|p|_\C^3|q|_{\C}^3} - q\frac{|t|_\C^3-|r|_{\C}^3}{|p|_\C^3|q|_{\C}^3} = -q\frac{|p|_{\C}^3-|q|_{\C}^3-|r|_{\C}^3+|t|_{\C}^3}{|p|_\C^3|q|_{\C}^3} \\
& \quad =-p\frac{|p|_{\C}^3-|q|_{\C}^3-|r|_{\C}^3+|t|_{\C}^3}{|p|_{\C}^6}+(p-q)\frac{|p|_{\C}^3-|q|_{\C}^3-|r|_{\C}^3+|t|_{\C}^3}{|p|_\C^3|q|_{\C}^3}\\
&\qquad-p\frac{|p|_{\C}^3-|q|_{\C}^3-|r|_{\C}^3+|t|_{\C}^3}{|p|_{\C}^3}\left(\frac{1}{|q|_{\C}^3}-\frac{1}{|p|_{\C}^3}\right)\,.
\end{align*}
Arguing exactly as we did before, we can estimate the last two terms by the right-hand side in \eqref{dd est}. We then focus on the first term on the right-hand side. We rewrite the numerator as 
\begin{align}\begin{aligned}\label{est dd norms}
& |p|_{\C}^3-|q|_{\C}^3-|r|_{\C}^3+|t|_{\C}^3=\frac{\scalar{p-q}{p+q}_{\C}(|p|_{\C}^2+|p|_{\C}|q|_{\C}+|q|_{\C}^2)}{|p|_{\C}+|q|_{\C}} \\
& \qquad-\frac{\scalar{r-t}{r+t}_{\C}(|r|_{\C}^2+|r|_{\C}|t|_{\C}+|t|_{\C}^2)}{|r|_{\C}+|t|_{\C}}\\
&\quad =\frac{\scalar{p-q-r+t}{p+q}_{\C}(|p|_{\C}^2+|p|_{\C}|q|_{\C}+|q|_{\C}^2)}{|p|_{\C}+|q|_{\C}}\\
& \qquad-\frac{\scalar{r-t}{r+t}_{\C}\left((|r|_{\C}^2+|r|_{\C}|t|_{\C}+|t|_{\C}^2)-(|p|_{\C}^2+|p|_{\C}|q|_{\C}+|q|_{\C}^2)\right)}{|r|_{\C}+|t|_{\C}}\\
&\qquad-\scalar{r-t}{r+t}_{\C}(|p|_{\C}^2+|p|_{\C}|q|_{\C}+|q|_{\C}^2)\left(\frac{1}{|r|_{\C}+|t|_{\C}}-\frac{1}{|p|_{\C}+|q|_{\C}}\right)\\
&\qquad -\frac{\scalar{r-t}{r+t-p-q}_{\C}(|p|_{\C}^2+|p|_{\C}|q|_{\C}+|q|_{\C}^2)}{|p|_{\C}+|q|_{\C}}.
\end{aligned}
\end{align}
The resulting terms from all summands here except the first can be estimated against the right-hand side in \eqref{dd est} using \eqref{bas est}.

Finally, concerning the first one, observe that
\begin{align*} 
    & \frac{\scalar{p-q-r+t}{p+q}_{\C}(|p|_{\C}^2+|p|_{\C}|q|_{\C}+|q|_{\C}^2)}{|p|_{\C}+|q|_{\C}} = 2|p|_{\C} \scalar{p-q-r+t}{p}_{\C} + \frac{2|q|_{\C}^2 \scalar{p-q-r+t}{p}_{\C}}{|p|_{\C} + |q|_{\C}} \\
    & \qquad + \frac{\scalar{p-q-r+t}{q-p}_{\C}(|p|_{\C}^2+|p|_{\C}|q|_{\C}+|q|_{\C}^2)}{|p|_{\C}+ |q|_{\C}} \\
    & \quad = 3 |p|_{\C}  \scalar{p-q-r+t}{p}_{\C} + \frac{\scalar{p-q-r+t}{p}_{\C}(|q|_{\C}^2-|p|_{\C}^2)}{|p|_{\C}+|q|_{\C}} + \frac{\scalar{p-q-r+t}{p}_{\C} |q|_{\C}(|q|_{\C}-|p|_{\C})}{|p|_{\C}+|q|_{\C}}\\
    & \qquad + \frac{\scalar{p-q-r+t}{q-p}_{\C}(|p|_{\C}^2+|p|_{\C}|q|_{\C}+|q|_{\C}^2)}{|p|_{\C}+ |q|_{\C}}\,.
\end{align*}
The corresponding terms from the last three summands can be estimated by the right-hand side in \eqref{dd est} using again \eqref{bas est}. At this point, a) follows by combining all these estimates.

\medbreak
\textbf{b)} This can be shown by splitting the double differences as earlier and showing that all the error terms contain a difference between the original variables and the variables with a $'$. For instance, the difference of the error terms in \eqref{2 dd est} can be split as \begin{align*}
\mel(r-t)\frac{|p|_{\C}^3-|r|_{\C}^3}{|p|_{\C}^3|r|_{\C}^3}-(r'-t')\frac{|p'|_{\C}^3-|r'|_{\C}^3}{|p'|_{\C}^3|r'|_{\C}^3}\\
&=(r-r'-t+t')\frac{|p|_{\C}^3-|r|_{\C}^3}{|p|_{\C}^3|r|_{\C}^3}+(r'-t')\frac{|p|_{\C}^3-|r|_{\C}^3-|p'|_{\C}^3+|r'|_{\C}^3}{|p|_{\C}^3|r|_{\C}^3}\\
&\quad+(r'-t')\frac{|p'|_{\C}^3-|r'|_{\C}^3(|p|_{\C}^3-|p'|_{\C}^3)|r|_{\C}^3+|p'|_{\C}^3(|r|_{\C}^3-|r'|_{\C}^3)}{|p|_{\C}^3|r|_{\C}^3|p'|_{\C}^3|q'|_{\C}^3}\,.
\end{align*}
The first and the last term can easily be estimated against the right-hand side of \eqref{td est} by using \eqref{E.complex-modulus} and \eqref{bas est2}. The numerator of the middle term can be rewritten in the same manner as in \eqref{est dd norms}, which yields the desired estimate. We omit the details for the remaining terms.
\end{proof}

Next, using classical Fourier multiplier theorems, we estimate some singular integrals in $C^\frac12(\T^2,\C)$. We deal with a general Calder\'on-Zygmund kernel depending on a complex matrix $M$ and a complex parameter $b$. Note that $\phi$ in the statement is the cutoff function defined in \eqref{E.phicutoff}.

\begin{lemma}\label{SingInt}
Let \begin{align*}
M:=\begin{pmatrix} 1 & 0\\ 0 & 1 \\ a_1 & a_2\end{pmatrix}\in \C^{3\times 2} \,,
\end{align*}
with $a_1,a_2\in \C$, and $b\in \C$. Assume that, for some  real number $A_0 \geq 1$, 
\begin{equation}\begin{aligned}
|\Im a_1|+|\Im a_2|\leq \frac{1}{100}\,, \quad 0\leq |\Im b|\leq \frac{1}{100A_0}|\Re b| \,,\quad\mathrm{and}\quad |a_1|+|a_2|\leq  A_0\,, \label{arg comp modulus}
\end{aligned}\end{equation}
and consider convolution kernels $f: \T^2 \to \C$ defined for $x \in [-\tfrac12,\tfrac12]^2$ by
 \begin{align*}
f(x):=\frac{x_1^{m_1}x_2^{m_2}\delta^r}{|Mx+\delta be_3|_{\C}^{r+m_1+m_2+2}}\phi(x)\, \mathds{1}_{c\leq |x|\leq \frac12}\,,
\end{align*}
with $\delta \in (0,1)$, $c\in [0,1/2]$, and $r,m_1,m_2\in \N \cup \{0\}$. Also, assume that  $r\neq 0$ or that $m_1+m_2$ is odd. Then, it follows that
\begin{equation} \label{E.convolutionHolder}
\norm{f *g}_{C^{\frac12}(\T^2)} \lesssim |b|^{-r}\, \norm{g}_{C^{\frac12}(\T^2)} \,, \quad \textup{for all } g \in C^{\frac12}(\T^2)\,.
\end{equation}
Moreover, for $k_1, k_2, k_3 \in \N \cup \{0\}$, the derivatives of $f$ with respect to the parameters $a_1,a_2$ and $b$ satisfy that
\begin{equation} \label{E.convolutionDerivativeHolder}
\norm{({\rm D}^{k_1}_{a_1}{\rm D}^{k_2}_{a_2} {\rm D}_b^{k_3} f) * g}_{C^{\frac12}(\T^2)}  \lesssim |b|^{-r-k_3}\, \norm{g}_{C^{\frac12}(\T^2)} \,, \quad \textup{for all } g \in C^{\frac12}(\T^2)\,.
\end{equation}
The implicit constants here depend only on $k_1, k_2, r,m_1,m_2$ and $A_0$, but not on $\delta$.
\end{lemma}
 
\begin{proof}
We start by proving a basic lower bound on the denominator. First, note that
\begin{align*}
& \Re|Mx+\delta be_3|_{\C}^2
=x_1^2+x_2^2+(\delta\Re b+\Re a_1 x_1+\Re a_2 x_2)^2-(\delta\Im b+\Im a_1 x_1+\Im a_2 x_2)^2\\
& \quad \geq |x|^2+(\delta\Re b+\Re a_1 x_1+\Re a_2 x_2)^2- 3(\delta^2 |\Im b|^2+|\Im a_1|^2x_1^2+|\Im a_2|^2x_2^2)\,.
\end{align*}
Also, combining Cauchy's product inequality with the last assumption in \eqref{arg comp modulus}, we see that
$$
|x|^2 + (\delta\Re b+\Re a_1 x_1+\Re a_2 x_2)^2 \geq \frac{|x|^2}{2} + \frac{1}{1+2A_0^2}\, \delta^2 (\Re b)^2\,.
$$
By the first two assumptions in \eqref{arg comp modulus}, we then get that
\begin{equation}\label{lower bd}
    \Re|Mx+\delta be_3|_{\C}^2 \geq \frac{|x|^2}{4} + \frac{1}{4A_0^2} \,\delta^2 (\Re b)^2 \geq \frac{1}{4A_0^2} (|x|^2 + \delta^2 (\Re b)^2)\,.
\end{equation}
This implies in particular that
$$
||Mx + \delta b e_3|_\C| \gtrsim|x| + \delta |b|\,.
$$

We now deal separately with the cases $r > 0$ and $r =0$, the first being simpler.

\medbreak 
\noindent \textit{Case 1: $r > 0$}. Note that \eqref{lower bd} implies
\begin{equation} \label{E.pointwiseF}
|f(x)| \lesssim \frac{\delta^r |x|^{m_1+m_2}}{|x|^{r+m_1+m_2+2} + (\delta|b|)^{r+m_1+m_2+2}}\,, \quad \textup{for all } x \in [-\tfrac12,\tfrac12]^2 \textup{ and all } r \geq 0\,.
\end{equation}
Hence, since we are dealing with the case where $r > 0$, we get, using polar coordinates to estimate the integral, that
\begin{align*}
    & \int_{B_{\frac{1}{2}}(0)} |f(x)| \dd x\lesssim \delta^r  \int_{B_{\frac{1}{2}}(0)} \frac{|x|^{m_1+m_2}}{|x|^{r+m_1+m_2+2} + (\delta|b|)^{r+m_1+m_2+2}} \dd x \\
    & \qquad \lesssim \frac{1}{\delta |b|^{r+1}} \int_0^{\frac12}  \frac{\big(\frac{s}{\delta|b|}\big)^{m_1+m_2+1}}{1+\big(\frac{s}{\delta|b|}\big)^{r+m_1+m_2+2}} \dd s  \leq \frac{1}{|b|^r} \int_0^{\infty} \frac{\rho^{m_1+m_2+1}}{1+\rho^{r+m_1+m_2+2}} \dd \rho \lesssim|b|^{-r}\,.
\end{align*}

In particular, the convolution $f *g$ is always well-defined for $g \in L^{\infty}(\T^2,\C)$. Moreover, by Young's convolution inequality, it follows that
\begin{align*}
&\norm{g*f}_{L^\infty(\T^2) }\lesssim |b|^{-r}\norm{g}_{L^\infty(\T^2)}\,,\\
&\norm{g*f(\cdot)-g*f(\cdot+y)}_{L^\infty(\T^2)} \lesssim |b|^{-r} \norm{g(\cdot)-g(\cdot+y)}_{L^\infty(\T^2)}\lesssim |b|^{-r}|y|^\frac{1}{2}\norm{g}_{C^\frac{1}{2}(\T^2)},
\end{align*}
which shows \eqref{E.convolutionHolder} in the case where $r > 0$.

\medbreak
\noindent \textit{Case 2: $r=0$.} In this case, the convolution needs to be estimated as a singular integral. Using that singular integrals are bounded on Hölder spaces (see for instance \cite[Corollary \ 6.7.2 and Remark \ 6.5.2]{Grafakos2} and \cite[Theorem 4.4.1]{Grafakos1}), it suffices to show that \begin{align}
\sup_{0<R\leq \frac12}\frac{1}{R} \int_{B_R(0)}|f(x)||x|\dx &\lesssim 1 \,,\label{sing 1}\\
\sup_{y\neq 0}\int_{|x|\geq 2|y|}|f(x-y)-f(x)|\dx&\lesssim 1 \,, \label{sing 2}\\
\sup_{0<R_1 < R_2\leq \frac12}\left|\int_{R_1 \leq |x| \leq R_2}f(x)\dx\right|&\lesssim 1\,.\label{sing 3}
\end{align}
Arguing as in the case where $r > 0$, \eqref{sing 1} immediately follows from \eqref{E.pointwiseF}. Hence, we focus on \eqref{sing 2} and \eqref{sing 3}.

We first deal with \eqref{sing 2}. Here, we consider separately the cases $|y|\geq c$ and $|y|\leq c$. Note that in the first case
\begin{align*}
    \int_{|x| \geq 2|y|} |f(x-y) - f(x)| \dd x = \int_{|x| \geq 2|y|} \left| \int_{-1}^0 \nabla f(x+\tau y)\cdot y \dd \tau \right| \dd x\,.
\end{align*}
Thus, using that
$$
|x+\tau y| \geq |x|-|y| \geq |y|\,, \quad \textup{for all } x,y \in B_{\frac{1}{2}}(0) \textup{ with }|x| \geq 2|y|\,,
$$
we get that
\begin{align*}
      \int_{|x| \geq 2|y|} |f(x-y) - f(x)| \dd x \leq \int_{|\xi| \geq |y|} |y| |\nabla f(\xi)| \dd \xi\,.
\end{align*}
Also, combining \eqref{arg comp modulus} and \eqref{lower bd}, we get that
$$
|\nabla f(\xi)| \lesssim \frac{|\xi|^{m_1+m_2-1}}{|\xi|^{m_1+m_2+2}+(\delta|b|)^{m_1+m_2+2}} + \frac{|\xi|^{m_1+m_2}}{|\xi|^{m_1+m_2+2}+(\delta|b|)^{m_1+m_2+2}} + \frac{|\xi|^{m_1+m_2}}{|\xi|^{m_1+m_2+3}+(\delta|b|)^{m_1+m_2+3}}\,,
$$
and so that
\begin{align*}
         \int_{|x| \geq 2|y|} |f(x-y) - f(x)| \dd x \lesssim |y| \int_{B_\frac{1}{2}(0)\backslash B_{|y|}(0)} \frac{1}{|\xi|^3} \dd \xi \lesssim 1\,.
\end{align*}
This immediately gives \eqref{sing 2} in the case where $|y|\geq c$.

In the case where $|y| \leq c$, we need to analyze separately the cases where $|x|  \geq 2c$ and $|x| \leq 2c$. Note that here we are implicitly assuming that $c > 0$. If $|x| \leq 2c$, then $|x-y| \leq 3c$ and hence, using \eqref{E.pointwiseF}, we can estimate \begin{align*}
    \int_{2c\geq|x|\geq 2|y|}|f(x-y)-f(x)|\dd x\leq 2\int_{|x|\in [c,3c]}|f(x)|\dx\lesssim \int_{B_{3c}(0)\backslash B_c(0)}\frac{1}{|x|^2}\dd x\lesssim \log(3c)-\log(c)\lesssim 1\,,
\end{align*}
showing \eqref{sing 2}. On the other hand, if $|x| \geq 2c$, we have that
$$
|x + \tau y| \geq |x|-|y| \geq c\,,
$$
and we can argue exactly as we did when $|y| \geq c$ and again obtain \eqref{sing 2}.

We finally deal with \eqref{sing 3}. Let $0 < R_1 < R_2 \leq 1/2$. We consider three different cases separately.

\noindent \textit{Case 2.1: $R_2 \leq 10\delta|b|$.} In this case, using once again \eqref{E.pointwiseF}, we immediately get that
\begin{align*}
    \left|\int_{R_1 \leq |x| \leq R_2}f(x)\dx\right| \leq \int_{B_{10\delta|b|}(0)}|f(x)| dx \lesssim \int_{0}^{10} \frac{\rho^{m_1+m_2+1}}{1+\rho^{m_1+m_2+2}} \dd \rho \lesssim 1\,.
\end{align*}

\noindent \textit{Case 2.2: $R_1 \geq 10\delta |b|$.} Since $m_1+m_2$ is odd, it follows that
\begin{align*}
\mel\left|\int_{R_1 \leq |x| \leq R_2}f(x)\dx\right| =\frac{1}{2} \left| \int_{B_{R_2}(0)\backslash B_{R_1}(0)} (f(x)+f(-x)) \dd x \right| \\
&\leq \int_{R_1 \leq |x| \leq R_2} |x_1|^{m_1}|x_2|^{m_2} \left| \frac{1}{|Mx+\delta b e_3|_\C^{m_1+m_2+2}} - \frac{1}{|-Mx+\delta b e_3|_\C^{m_1+m_2+2}} \right| \dd x \,.
\end{align*}
Moreover, using \eqref{arg comp modulus}, \eqref{lower bd}, and the polarization identity
$$
||Mx+\delta b e_3|_\C^2 - |-Mx+\delta b e_3|_\C^2| = 4 |  \scalar{Mx}{\delta be_3}_\C|\,,
$$
we get that
\begin{align*}
&\left| \frac{1}{|Mx+\delta b e_3|_\C^{m_1+m_2+2}} - \frac{1}{|-Mx+\delta b e_3|_\C^{m_1+m_2+2}} \right| \lesssim \frac{| \scalar{Mx}{\delta b e_3}_\C|}{|x|^{m_1+m_2+4}+(\delta|b|)^{m_1+m_2+4}}\\
&\lesssim \frac{|x|\delta|b|}{|x|^{m_1+m_2+4}+(\delta|b|)^{m_1+m_2+4}}\,, \quad \textup{for } |x| \in [R_1,R_2]\,.
\end{align*}
Thus, we infer that
$$
   \left|\int_{R_1 \leq |x| \leq R_2}f(x)\dx\right| \lesssim \frac{1}{\delta|b|} \int_{10\delta |b|}^{\frac12}  \frac{\big(\frac{s}{\delta|b|}\big)^{m_1+m_2+2}}{1+\big(\frac{s}{\delta|b|}\big)^{m_1+m_2+4}}\,\dd s \lesssim 1\,.
$$

\noindent \textit{Case 2.3: $R_1 < 10\delta|b| < R_2$.} Note that, in this case
$$
   \left|\int_{R_1 \leq |x| \leq R_2}f(x)\dx\right| \leq    \left|\int_{R_1 \leq |x| \leq 10\delta|b|}f(x)\dx\right| +    \left|\int_{10\delta|b| \leq |x| \leq R_2}f(x)\dx\right|\,.
$$
Hence, combining the previous cases, we get that
$$
 \left|\int_{R_1 \leq |x| \leq R_2}f(x)\dx\right| \lesssim 1\,.
$$

We have now proved \eqref{sing 1}--\eqref{sing 3}. Hence, by \cite[Corollary \ 6.7.2 and Remark \ 6.5.2]{Grafakos2} and \cite[Theorem 4.4.1]{Grafakos1}, we conclude that \eqref{E.convolutionHolder} holds. 

Concerning \eqref{E.convolutionDerivativeHolder}, we simply note that the derivatives of $f$ with respect to $a_1$, $a_2$, and $b$ are a linear combination of terms with the same structure as $f$, the only difference being that each derivative with respect to~$b$ introduces a factor~$\delta$. Each additional~$\delta$ can be absorbed in the estimates by picking an additional factor~$|b|$, with the same argument as above. This is the content of the estimate~\eqref{E.convolutionDerivativeHolder}.
\end{proof}

\subsection{Proof of  (\ref{bd K})} Having Lemmas \ref{Lem double diff} and \ref{SingInt} at hand, we can now prove the first conclusion in Proposition \ref{main est}, namely \eqref{bd K}. We start by introducing some more notation. Throughout this whole section, let $|\beta|\leq \rho$. First of all, we set
\begin{align*}
& d_{l,\ell}^{\eps}(s+i\beta,\varsigma+i\beta)=s+i\beta - (\varsigma+i\beta)  +  w_1[\ell,s+i\beta]e_3-w_1[\ell,\varsigma+i\beta]e_3 \\
&\quad- \ep\left( \int_l^{\ell}  w_4[\mu,s+i\beta]  \dd \mu \right) e_3= \digamma(s+i\beta)-\digamma(\varsigma+i\beta)+\eps\zeta(s+i\beta)\,,
\end{align*}
with
\begin{equation} \label{E.digamma}
\digamma(a+ib) := a+ib + w_1[\ell,a+ib]\, e_3\,,
\end{equation}
and
\begin{equation} \label{E.zeta}
 \zeta(a+ib) := - \bigg( \int_l^{\ell}  w_4[\mu,a+ib]  \dd \mu \bigg) e_3\,.
\end{equation}
Then, we have that
\begin{equation} \label{E.holomorphicKepsilon}
\begin{aligned}
    &K_\ep^c\left[w_1[\ell,\cdot],\,\int_l^\ell w_4[\mu,\cdot] \dd \mu,\, w_{5-7}[\ell,\cdot]\right](s+i\be) \\
    & \qquad =\frac{1}{4\pi} \int_{\T^2} w_{5-7}[\ell,\varsigma+i\beta] \times  \frac{d_{l,\ell}^{\eps}(s+i\beta,\varsigma+i\beta)}{\big|d_{l,\ell}^{\eps}(s+i\beta,\varsigma+i\beta)\big|_{\C}^3} \phi(s-\varsigma)\dd \varsigma\,,
    \end{aligned}
\end{equation}
where $\phi$ is the periodic extension to $\T^2$ of the cutoff function introduced in \eqref{E.phicutoff}.

Before going further, let us point out that both $\digamma$ and $\zeta$ depend on $w$, $l$ and $\ell$, but we do not reflect this in the notation. Moreover, taking into account the notation introduced in Section \ref{S.localExistence}, one should notice that $\digamma|_{\T^2} = \ga_{\ep,\ell}$.

Let us also stress that through the rest of this section, we consider $0 \leq \rho < \rho_0$ and $w \in \cY^7_\rho$ satisfying the following:
\begin{align} \label{E.Zrho1}
    & \pd_{s_j} w_1[l,\cdot] = w_{1+j}[l,\cdot]\,, \quad \textup{for all } |l| \leq 1\text{ and $j\in \{1,2\}$}\,, \\
    & w_1[l,\cdot] - w_1[\ell,\cdot] = \ep \int_\ell^l w_4[\mu,\cdot]  \dd \mu \,, \quad \textup{for all } l,\ell \in [-1,1]\,. \label{E.Zrho2}
\end{align}
Moreover, we assume in all the proofs that $l \neq \ell$ (which is not restrictive as $\{l-\ell\}$ is a zero set). At this point, taking into account \eqref{E.Zrho1} and \eqref{E.Zrho2}, one can easily check that
 \begin{equation}\label{direct est}
 \begin{aligned}
& \sup_{|\be| \leq \rho} \Big\{\norm{\digamma(\cdot +i\beta)}_{C^\frac{3}{2}(\T^2)}+ \frac{1}{|l-\ell|}\norm{\zeta(\cdot+i\beta)}_{C^\frac{1}{2}(\T^2)}+\eps\norm{\zeta(\cdot+i\beta)}_{C^\frac{3}{2}(\T^2)} \Big\} \\
& \qquad \lesssim 1+\sum_{j=1}^4 \sup_{|l| \leq 1} \norm{w_j[l,\cdot]}_\rho.
\end{aligned}
\end{equation}

Finally, we set
$$
\tilde{d}(v,s+i\beta) :=  {\rm{D}}_s \digamma(s+i\beta)\cdot v +\eps\zeta(s+i\beta)\,, 
$$
and split the kernel in \eqref{E.holomorphicKepsilon} as \begin{equation}\begin{aligned}
& \frac{d_{l,\ell}^{\eps}(s+i\beta,\varsigma+i\beta)}{\big|d_{l,\ell}^{\eps}(s+i\beta,\varsigma+i\beta)\big|_{\C}^3}\, \phi(s-\varsigma)= J_1(s-\varsigma,s+i\beta)+J_2(s+i\beta,\varsigma+i\beta)\,,\label{splitting J}
\end{aligned}\end{equation}
with
\begin{align*}
    & J_1(s-\varsigma, s+i\be) := \frac{\tilde{d}(s-\varsigma,s+i\beta)}{\big|\tilde{d}(s-\varsigma,s+i\beta)\big|_{\C}^3}\phi(s-\varsigma)\,,\\
    & J_2(s+i\beta,\varsigma+i\beta) := \left(\frac{d_{l,\ell}^{\eps}(s+i\beta,\varsigma+i\beta)}{\big|d_{l,\ell}^{\eps}(s+i\beta,\varsigma+i\beta)\big|_{\C}^3}-\frac{\tilde{d}(s-\varsigma,s+i\beta)}{\big|\tilde{d}(s-\varsigma,s+i\beta)\big|_{\C}^3}\right)\phi(s-\varsigma)\,.
\end{align*}

Having  \eqref{splitting J} at hand,  we analyze the part corresponding to each kernel $J_j$ separately. Before doing so, we prove some elementary estimates that we will need. 

\begin{lemma} \label{L.ab}
    Let \begin{equation*}(a_1,a_2)= a(s+i\be) :=\nabla_s w_1[\ell,s+i\be] \in \C^2\quad \text{and} \quad b e_3 = b(s+i\be)e_3:= |l-\ell|^{-1}\ze(s+i\be) \in \C^3.\end{equation*} If \eqref{E.smallness main est} and \eqref{E.smallness main est 2} hold, then $a_1, a_2$ and $b$ fulfill \eqref{arg comp modulus} with $A_0 = \mathcal{C}_1$ given in \eqref{E.cC1}, and it holds that $|b|\approx 1$ uniformly in the other parameters.  
\end{lemma}

\begin{remark}
    Here and in the rest of this subsection, by abuse of notation, we say that $w \in \cY_\rho^7$ satisfies \eqref{E.smallness main est} and \eqref{E.smallness main est 2} if they both hold with $w = w^1 = w^2 \in \cY_\rho^7$. 
\end{remark}
 
\begin{proof}
    First of all, note that
    $$
    |a_1| + |a_2| = |w_2[\ell,s+i\be]| + |w_3[\ell,s+i\be]| \leq \mathcal{C}_1\,.
    $$
    Similarly, using \eqref{E.smallness main est} and the fact that $\mathcal{C}_1 \geq 1$, it is immediate to check that
    $$
    |\Im a_1| + |\Im a_2| \leq \norm{w}_{\cZ_\rho^{\Im}} \leq \frac{1}{300 \mathcal{C}_1} \leq \frac{1}{300}\,.
    $$
    Next, observe that
    \begin{equation} \label{E.lowerboundb}
    |\Re b| \geq \Big(1-\sup_{|l|\leq 1} \sup_{|\be| \leq \rho} \norm{\Re w_4[l, \cdot + i\be]-1}_{L^\infty(\TT^2)} \Big)\,,
    \end{equation}
    and that
    $$
    |\Im b| \leq  \sup_{|\be| \leq \rho} \norm{\Im w_4[l, \cdot + i\be]}_{L^\infty(\TT^2)}\,.
    $$
    Combining these inequalities, it is straightforward to check that, if the following inequality holds for all $l \in [-1,1]$,
    \begin{equation} \label{E.conclusion mini lemma}
     100 A_0 \sup_{|\be| \leq \rho} \norm{\Im w_4[l, \cdot + i\be]}_{L^\infty(\TT^2)} +  \sup_{|\be| \leq \rho} \norm{\Re w_4[l, \cdot + i\be]-1}_{L^\infty(\TT^2)} \leq 1 \,,
    \end{equation}
    then
    \begin{equation} \label{E.upperboundb}
    0 \leq |\Im b| \leq \frac{1}{100 A_0} |\Re b|\,.
    \end{equation}
    Taking into account that \eqref{E.conclusion mini lemma} with $A_0 = \mathcal{C}_1$ is an immediate consequence of \eqref{E.smallness main est}, the previous inequalities hold. Finally, note that the bounds on $b$ follow from \eqref{E.smallness main est}, \eqref{E.lowerboundb} and \eqref{E.upperboundb}.
\end{proof}

\begin{lemma}\label{bds d}
Let $\mathcal{C}_1$ be as in Proposition \ref{main est}.  If $w \in \cY^7_\rho$ satisfies \eqref{E.smallness main est}, \eqref{E.smallness main est 2}, \eqref{E.Zrho1} and \eqref{E.Zrho2}, it follows that
 \begin{align}
& \textup{\rm a) }\ |d_{l,\ell}^\eps(s+i\beta,\varsigma+i\beta)|+|\tilde{d}(s-\varsigma,s+i\beta)|\lesssim |s-\varsigma|+\eps|l-\ell|\,, \label{bd d up} \\
& \textup{\rm b) }\  |d_{l,\ell}^\eps(s+i\beta,\varsigma+i\beta)-\tilde{d}(s-\varsigma,s+i\beta)|\lesssim |s-\varsigma|^\frac{3}{2}\,, \label{bd diff} \\
&  \textup{\rm c) }\  \min \Big\{\Re |d_{l,\ell}^\eps(s+i\beta,\varsigma+i\beta)|_\C^2\,, \, \Re |\tilde{d}(s-\varsigma,s+i\beta)|_\C^2 \Big\} \gtrsim |s-\varsigma|^2+\eps^2|l-\ell|^2\,,\label{bd d low} 
\end{align}
with implicit constants that depend only on $\mathcal{C}_1$ but not on $\rho$ or $\ep$.
\end{lemma}

\begin{proof}
First, let us recall that
$$
d_{l,\ell}^\ep (s+i\be, \varsigma+i\be) = \digamma(s+i\be) - \digamma(\varsigma+i\be) + \ep \ze(s+i\be)\,, 
$$
and that
$$
\tilde{d}(s-\varsigma,s+i\be) = {\rm D}_s \digamma(s+i\be) (s-\varsigma) + \ep\ze(s+i\be)\,.
$$
We prove each estimate separately. Since the upper bounds in a) are a straightforward consequence of \eqref{direct est}, we focus on b) and c). Using again \eqref{direct est}, and the mean value theorem, we get that
\begin{align*}
& |d_{l,\ell}^\eps(s+i\beta,\varsigma+i\beta)-\tilde{d}(s-\varsigma,s+i\beta)| \\
& = \left|\digamma(s+i\beta)-\digamma(\varsigma+i\beta)-\mathrm{D}_s\digamma(s+i\beta)(s-\varsigma)\right| \lesssim |s-\varsigma|^\frac{3}{2}\norm{\digamma(\cdot+i\beta)}_{C^\frac{3}{2}(\T^2)}\lesssim |s-\varsigma|^\frac{3}{2}\,.
\end{align*}

Finally, we prove the lower bounds in c). For $M$ and $b$ as in Lemma \ref{L.ab}, it follows that
$$
|\tilde{d}(s-\varsigma, \varsigma+i\be)|_\C^2 = |M(s-\varsigma) + \ep |l-\ell| be_3|_\C^2\,.
$$
On the one hand, using \eqref{lower bd}, we get that
$$
\Re |\tilde{d}(s-\varsigma, \varsigma+i\be)|_\C^2 \gtrsim  |s-\varsigma|^2 + \ep^2 |l-\ell|^2 \,.
$$
On the other hand, arguing as in \eqref{lower bd}, we infer that
\begin{align*}
\mel\Re |d_{l,\ell}^\eps(s+i\beta,\varsigma)|_{\C}^2= |s-\varsigma|^2+ \Re\left((\eps\zeta(s+i\be)+w_1[\ell,s+i\be]-w_1[\ell,\varsigma+i\be])^2\right) \\
&\geq |s-\varsigma|^2+(\eps\Re \zeta(s+i\be)+\Re w_1[\ell,s+i\be]-\Re w_1[\ell,\varsigma+i\be])^2\\
&\quad-3\left(\eps^2|\Im \zeta(s+i\be)|^2+|\Im w_1[\ell,s+i\be]-\Im w_1[\ell,\varsigma+i\be]|^2\right)\\
&\gtrsim |s-\varsigma|^2(1-2\norm{\Im \nabla w_1[\ell,\cdot+i\be]}_{L^\infty(\T^2)}^2)+\eps^2 (|\Re \zeta(s+i\beta)|^2-2|\Im \zeta(s+i\beta)|)\,.
\end{align*}
At this point, using Lemma \ref{L.ab}, the result immediately follows from \eqref{arg comp modulus}.
\end{proof}

Having this lemma at hand, and before going any further, let us point out that the fact that $K_\ep^c$ is holomorphic for $l \neq \ell$ is a straightforward consequence of the lower bound in c). More precisely, we have the following:

\begin{lemma} \label{L.holomorphicExtension}
Let $\mathcal{C}_1$ be as in Proposition \ref{main est}.  If $w \in \cY^7_\rho$ satisfies \eqref{E.smallness main est}, \eqref{E.smallness main est 2}, \eqref{E.Zrho1} and \eqref{E.Zrho2}, then for all $\eps\in (0,1)$ and $l\neq \ell$, it holds that
$$
K_\ep^c\bigg[w_1[\ell,\cdot],\int_l^\ell w_4[\mu,\cdot]\dd\mu, w_{5-7}[\ell,\cdot]\bigg] \in \cX_\rho^3\,.
$$
\end{lemma}

\begin{proof} First of all, let us recall that
\begin{align*}
K_\ep^c&\bigg[w_1[\ell,\cdot],\int_l^\ell w_4[\mu,\cdot]\dd\mu,w_{5-7}[\ell,\cdot]\bigg]  \\
    & = \frac{1}{4\pi} \int_{\TT^2}  w_{5-7}[\ell, s+i\be + \varsigma] \times \phi(s-\varsigma)\frac{\digamma(s+i\beta)-\digamma(s+i\beta+ \varsigma)+\eps\zeta(s+i\beta)}{|\digamma(s+i\beta)-\digamma(s+i\beta+\varsigma)+\eps\zeta(s+i\beta)|_\C^3}  \dd \varsigma 
\end{align*}
Moreover, by Lemma \ref{bds d} c) we know that
$$
\Re |\digamma(s+i\beta)-\digamma(s+i\beta+ \varsigma)+\eps\zeta(s+i\beta)|_\C^2 \gtrsim \varsigma^2 + \ep^2 |l-\ell|^2\,,
$$
which implies that the denominator is always nonzero when $l\neq \ell$, and implies the result, since the integrand is in $\cX_\rho^3$.
\end{proof}

We can now deal with the part corresponding to $J_1$. We treat it as a convolution operator on $\T^2$ with an extra dependence on $s + i \be \in \TT^2_\rho$.

\begin{lemma} \label{L.J1}
Let $\mathcal{C}_1$ be as in Proposition \ref{main est}. If $w \in \cY^7_\rho$ satisfies \eqref{E.smallness main est}, \eqref{E.smallness main est 2}, \eqref{E.Zrho1} and \eqref{E.Zrho2}, then
$$
\sup_{l,\ell \in [-1,1]} \left\{ \sup_{|\be| \leq \rho} \norm{ \int_{\T^2} J_1(\cdot-\varsigma,\cdot+i\be) \wedge w_{5-7}[\ell,\varsigma+i\be] \dd \varsigma}_{C^{\frac12}(\T^2)} \right\} \lesssim\, 1 \,, 
$$
with an implicit constant that depends only on $\mathcal{C}_1$, but not on $\rho$ or $\ep$.
\end{lemma}

\begin{proof}
For $a_1,a_2$ and $b$ as in Lemma \ref{L.ab}, and the corresponding $M$ as in Lemma \ref{SingInt} observe that
$$
J_1(x,s+i\be) =  \frac{Mx}{|Mx+\ep|l-\ell| be_3|_\C^3} + \frac{\ep |l-\ell|b e_3}{|Mx+\ep|l-\ell| be_3|_\C^3}\,.
$$
We further split the numerator of the first summand into the contributions of $x_1$ and $x_2$ to see with Lemmas \ref{SingInt}, \ref{L.ab} and $\delta=\eps|l-\ell|$ that
\begin{align} \label{E.J1bounded}
\norm{\int_{\T^2} J_1(\cdot-\varsigma,\cdot+i\beta)\wedge w_{5-7}[\ell,\varsigma+i\beta]\dd\varsigma}_{L^{\infty}(\T^2)} \lesssim \norm{w_{5-7}[\ell,\cdot+i\beta]}_{C^\frac{1}{2}(\T^2)}\,, \quad \textup{for } \be \in B_\rho(0)\,.
\end{align}

Now, observe that for all $s^1,s^2 \in \T^2$  we have
\begin{align*}
& \left|\int_{\T^2} \big( J_1(s^1-\varsigma,s^1+i\beta)\wedge w_{5-7}[\ell,\varsigma+i\beta]-J_1(s^2-\varsigma,s^2+i\beta)\wedge w_{5-7}[\ell,\varsigma+i\beta] \big) \dd\varsigma\right|\\
& \quad \leq \left| \int_{\T^2} J_1(s^1-\varsigma,s^1+i\beta)\wedge w_{5-7}[\ell,\varsigma+i\beta] \dd \varsigma - \int_{\T^2} J_1(s^2-\varsigma,s^1+i\beta)\wedge w_{5-7}[\ell,\varsigma+i\beta] \dd \varsigma \right| \\
&\qquad+\left|\int_{\T^2} \big(
J_1(s^2-\varsigma,s^1+i\beta)-J_1(s^2-\varsigma,s^2+i\beta)\big)\wedge w_{5-7}[\ell,\varsigma+i\beta]\dd\varsigma\right| =: {\rm I}_1 + {\rm I}_2\,.
\end{align*}
Arguing as in the proof of \eqref{E.J1bounded}, using again Lemma \ref{SingInt}, we get that
\begin{equation} \label{E.conclussionI1J1}
     {\rm I}_1 \lesssim |s^1-s^2|^{\frac12} \norm{w_{5-7}[\ell,\cdot+i\beta]}_{C^\frac{1}{2}(\T^2)}\,, \quad \textup{for }  \be \in B_\rho(0)\,.
\end{equation}

Next, we deal with ${\rm I}_2$. For $s^1, s^2 \in \TT^2$ and $a$ and $b$ as in Lemma \ref{L.ab}, we use the notation $a_k^j := a_k(s^j+i\be)$ and $b^j:= b(s^j+i\be)$ with $j,k \in \{1,2\}$, and observe that by \eqref{E.convolutionDerivativeHolder}, the convolution with $J_1$ is Lipschitz in these coefficients (with respect to the $L(C^\frac{1}{2},C^\frac{1}{2})$-topology). Hence, by Lemma \ref{SingInt}, it follows that
\begin{align*}
{\rm I}_2\lesssim \left( |a_1^1-a_1^2|+|a_2^1-a_2^2|+\max_{t\in [0,1]}\frac{|b^1-b^2|}{|b(s^1t+(1-t)s^2+i\be)|} \right) \norm{w_{5-7}[\ell,\cdot+i\beta]}_{C^\frac{1}{2}(\T^2)}, \  \textup{for }  \be \in B_\rho(0)\,.
\end{align*}
Also, observe that, by \eqref{direct est},
\begin{equation}  \label{E.a1a2diff}
|a^1-a^2| = |{\rm D}_s \digamma(s^1+i\be) - {\rm D}_s \digamma(s^2+i\be)| \lesssim |s^1-s^2|^{\frac12}\,,
\end{equation}
and
$$
|b^1-b^2| = |\zeta(s^1+i\be) - \zeta(s^2+i\be)| \lesssim  |s^1-s^2|^{\frac12}\,.
$$
Finally, taking into account \eqref{E.complex-modulus} and the definition of $b$, we infer that
\begin{equation} \label{E.lowerbll}
|b(s^1t+(1-t)s^2+i\be)| \gtrsim  1\,,
\end{equation}
uniformly in $t \in [0,1]$. Combining these estimates, we conclude that
\begin{equation} \label{E.conclussionI2J1}
     {\rm I}_2 \lesssim |s^1-s^2|^{\frac12} \norm{w_{5-7}[\ell,\cdot+i\beta]}_{C^\frac{1}{2}(\T^2)}\,, \quad \textup{for } \be \in B_\rho(0)\,.
\end{equation}
The result follows from \eqref{E.J1bounded}, \eqref{E.conclussionI1J1} and \eqref{E.conclussionI2J1}.
\end{proof}

Next, we deal with the part corresponding to $J_2$.

\begin{lemma} \label{L.J2}
Let $\mathcal{C}_1$ be as in Proposition \ref{main est}. If $w \in \cY^7_\rho$ satisfies \eqref{E.smallness main est}, \eqref{E.smallness main est 2}, \eqref{E.Zrho1} and \eqref{E.Zrho2}, then
$$
\sup_{l,\ell \in [-1,1]} \left\{ \sup_{|\be| \leq \rho} \norm{ \int_{\T^2} J_2(\cdot-\varsigma,\cdot+i\be)\wedge w_{5-7}[\ell,\varsigma+i\be] \dd \varsigma}_{C^{\frac12}(\T^2)} \right\} \lesssim 1\,, 
$$
with an implicit constant that depends only on $\mathcal{C}_1$ but not on $\rho$ or $\eps$.
\end{lemma}

\begin{proof}
First of all, observe that
\begin{align*}
|J_2(s+i\beta,\varsigma+i\beta)|&= \left|\frac{d_{l,\ell}^{\eps}(s+i\beta,\varsigma+i\beta)}{|d_{l,\ell}^{\eps}(s+i\beta,\varsigma+i\beta)|_{\C}^3}-\frac{\tilde{d}(s-\varsigma,s+i\beta)}{|\tilde{d}(s-\varsigma,s+i\beta)|_{\C}^3}\right|\phi(s-\varsigma)\\
&\leq \left|\frac{d_{l,\ell}^{\eps}(s+i\beta,\varsigma+i\beta)-\tilde{d}(s-\varsigma,s+i\beta)}{|d_{l,\ell}^{\eps}(s+i\beta,\varsigma+i\beta)|_{\C}^3}\right|\phi(s-\varsigma)\\
&\quad+\left|\frac{\tilde{d}(s-\varsigma,s+i\beta)\big(|\tilde{d}(s-\varsigma,s+i\beta)|_{\C}^3-|d_{l,\ell}^{\eps}(s+i\beta,\varsigma+i\beta)|_{\C}^3\big)}{|d_{l,\ell}^{\eps}(s+i\beta,\varsigma+i\beta)|_{\C}^3|\tilde{d}(s-\varsigma,s+i\beta)|_{\C}^3}\right|\phi(s-\varsigma)\,.
\end{align*}

\noindent Using the estimate \eqref{bas est2} and the bounds from Lemma \ref{bds d}, we then get that 
\begin{align}
\left|J_2(s+i\beta,\varsigma+i\beta)\right|\lesssim  |s-\varsigma|^{-\frac{3}{2}}\,\mathds{1}_{|s-\varsigma|\leq\frac{1}{4}}\,, \quad \textup{for } \be \in B_\rho(0)\,, \label{bd J2}
\end{align}
and so that 
\begin{equation} \label{E.J2bounded}
\begin{aligned}
& \norm{\int_{\T^2} J_2(\cdot-\varsigma,\cdot+i\beta)\wedge w_{5-7}[\ell,\varsigma+i\beta]\dd\varsigma}_{L^{\infty}(\T^2)} \\ & \qquad  \lesssim \norm{w_{5-7}[\ell,\cdot+i\beta]}_{L^\infty(\T^2)} \int_{|s-\varsigma|\leq\frac{1}{2}} \frac{\dd\varsigma}{|s-\varsigma|^{\frac{3}{2}}}\lesssim \norm{w_{5-7}[\ell,\cdot+i\beta]}_{L^\infty(\T^2)}\,, \quad \textup{for } \be \in B_\rho(0)\,.
\end{aligned}
\end{equation}

Next, we bound the corresponding H\"older seminorm. For arbitrary $s^1, s^2 \in \TT^2$, we use that
\begin{align*}
&\left|\int_{\T^2} \big(J_2(s^1+i\beta,\varsigma+i\beta)-J_2(s^2+i\beta,\varsigma+i\beta)\big)\wedge w_{5-7}[\ell,\varsigma+i\beta]\dd\varsigma \right|\\
& \quad \leq\int_{|\varsigma-s^1|\leq 3|s^1-s^2|}\big(|J_2(s^1+i\beta,\varsigma+i\beta)|+|J_2(s^2+i\beta,\varsigma+i\beta)|\big)\left|w_{5-7}[\ell,\varsigma+i\beta] \right| \dd\varsigma\\
&\qquad+\left|\int_{|\varsigma-s^1|\geq 3|s^1-s^2|}\left(J_2(s^1+i\beta,\varsigma+i\beta)-J_2(s^2+i\beta,\varsigma+i\beta)\right)\wedge w_{5-7}[\ell,\varsigma+i\beta]\dd\varsigma\right| =: {\rm I}_1 + {\rm I}_2\,.
\end{align*}
On the one hand, using \eqref{bd J2} and arguing as in \eqref{E.J2bounded}, we get that
\begin{equation} \label{E.conclussionI1J2}
     {\rm I}_1 \lesssim |s^1-s^2|^{\frac12} \norm{w_{5-7}[\ell,\cdot+i\beta]}_{L^{\infty}(\T^2)}\,, \quad \textup{for } \be \in B_\rho(0)\,.
\end{equation}
On the other hand, to estimate ${\rm I}_2$, we will consider separately the cases where $|s^1 - s^2| \geq \frac{1}{32}$ and where $|s^1 - s^2| \leq \frac{1}{32}$.
 
\noindent \textit{Case 1:} $|s^1-s^2| \geq \frac1{32}$. Arguing as in \eqref{E.J2bounded}, one can easily get that
\begin{equation} \label{E.conclussionI2J2}
     {\rm I}_2 \lesssim |s^1-s^2|^{\frac12} \norm{w_{5-7}[\ell,\cdot+i\beta]}_{L^{\infty}(\T^2)}\,, \quad \textup{for } \be \in B_\rho(0)\,.
\end{equation}

\noindent \textit{Case 2:} $|s^1-s^2| \leq \frac{1}{32}$. First of all, we set
\begin{align*}
& p = p(\varsigma,s^1+i\be):= \tilde{d}(s^1-\varsigma,s^1+i\beta) \,,\\ 
& q = q(\varsigma, s^1+i\be) := d_{l,\ell}^\eps(s^1+i\beta,\varsigma+i\beta)\,,     \\
& r = r(\varsigma,s^2+i\be):= \tilde{d}(s^2-\varsigma,s^2+i\beta)\,,   \\
& t = t(\varsigma,s^2+i\be) := d_{l,\ell}^\eps(s^2+i\beta,\varsigma+i\beta)\,,
\end{align*}
and split the kernel in ${\rm I}_2$ as
\begin{align*}
& J_2(s^2+i\beta,\varsigma+i\beta)-J_2(s^1+i\beta,\varsigma+i\beta)=\left(\frac{p}{|p|_{\C}^3}-\frac{q}{|q|_{\C}^3}-\frac{r}{|r|_{\C}^3}+\frac{t}{|t|_{\C}^3}\right)\phi(s^1-\varsigma)\\
&\quad+\left(\frac{r}{|r|_{\C}^3}-\frac{t}{|t|_{\C}^3}\right)\big(\phi(s^1-\varsigma)-\phi(s^2-\varsigma)\big)\,.
\end{align*}
Then, by the mean value theorem, 
\begin{align*}
    \left|\frac{r}{|r|_{\C}^3}-\frac{t}{|t|_{\C}^3}\right|\big|\phi(s^1-\varsigma)-\phi(s^2-\varsigma)\big| \leq \left|\frac{r}{|r|_{\C}^3}-\frac{t}{|t|_{\C}^3}\right| \left( \int_0^1 |\nabla \phi(\la(s^1-s^2) + s^2-\varsigma)| \dd \la \right) |s^1-s^2|\,. 
\end{align*}
Moreover, observe that
\begin{equation} \label{E.supportPhiprime}
\nabla \phi(\la(s^1-s^2) + s^2-\varsigma) \neq 0 \ \Longrightarrow \ |s^1-\varsigma| \geq \frac{1}{16} \ \Longrightarrow\ |s^2-\varsigma| \geq \frac{1}{32}\,.
\end{equation}
Hence, by \eqref{bd d low}, it follows that
$$
   \left|\frac{r}{|r|_{\C}^3}-\frac{t}{|t|_{\C}^3}\right|\big|\phi(s^1-\varsigma)-\phi(s^2-\varsigma)\big|  \lesssim |s^1-s^2|\,,  \quad \textup{for } \be \in B_\rho(0)\,,
$$
and so that
\begin{equation*}
\begin{aligned}
    & \left| \int_{|\varsigma-s^1|\geq 3|s^1-s^2|}\left(\frac{r}{|r|_{\C}^3}-\frac{t}{|t|_{\C}^3}\right)\big(\phi(s^1-\varsigma)-\phi(s^2-\varsigma)\big)\wedge w_{5-7}[\ell,\varsigma+i\beta]\dd\varsigma \right| \\[0.4cm]
    & \qquad  \lesssim |s^1-s^2| \norm{w_{5-7}[\ell,\cdot+i\beta]}_{L^{\infty}(\T)^2} \lesssim |s^1-s^2|^{\frac12} \norm{w_{5-7}[\ell,\cdot+i\beta]}_{L^{\infty}(\T^2)}  \,, \quad \textup{for } \be \in B_\rho(0)\,.
    \end{aligned}
\end{equation*}

To estimate the remaining part, we rely on Lemma \ref{Lem double diff}. First, observe that
\begin{align}\begin{aligned}\label{fracc dd}
    &\frac{p-q-r+t}{|p|_{\C}^3} = - \frac{\digamma(s^1+i\beta)-\digamma(s^2+i\beta)-{\rm D}_s\digamma(s^2+i\beta)(s^1-s^2)}{|p|_\C^3} \\
    & \qquad +  \frac{1}{|p|_\C^3} \big( {\rm D}_s\digamma(s^1+i\be) - {\rm D}_s\digamma(s^2+i\be) \big)(s^1-\varsigma)\,.
\end{aligned}\end{align}
Now, on the one hand, combining \eqref{direct est} with \eqref{bd d low}, we get that
\begin{equation} \label{E.Digamma1}
\left|\frac{\digamma(s^1+i\beta)-\digamma(s^2+i\beta)-{\rm D}_s\digamma(s^2+i\beta)(s^1-s^2)}{|p|_\C^3}\right| \lesssim \frac{|s^1-s^2|^{\frac32}}{|s^1-\varsigma|^{3}}\,, \quad \textup{for } \be \in B_\rho(0)\,,
\end{equation}
and so that
\begin{align}
& \left| \int_{|\varsigma-s^1|\geq 3|s^1-s^2|} \frac{\digamma(s^1+i\beta)-\digamma(s^2+i\beta)-{\rm D}_s\digamma(s^2+i\beta)(s^1-s^2)}{|p|_\C^3} \phi(s^1-\varsigma) \wedge  w_{5-7}[\ell, \varsigma+i\be] \dd \varsigma \right| \nonumber \\ \label{E.Digamma2}
& \qquad \lesssim \norm{w_{5-7}[\ell,\cdot+i\beta]}_{L^{\infty}(\T^2)} |s^1-s^2|^{\frac32} \int_{|\varsigma-s^1|\geq 3|s^1-s^2|}  \frac{1}{|s^1-\varsigma|^{3}} \dd\varsigma \\ \nonumber
& \qquad \lesssim |s^1-s^2|^{\frac12} \norm{w_{5-7}[\ell,\cdot+i\beta]}_{L^{\infty}(\T^2)}\,, \qquad \textup{for } \be \in B_\rho(0)\,.
\end{align}
On the other hand, using Lemma \ref{SingInt} with $c = 3 |s^1-s^2|$, $r = 0$, $m_1 = 1$ and $m_2=0$ (resp.\ $m_1=0$ and $m_2=1)$, Lemma \ref{L.ab}, and \eqref{E.a1a2diff}, we get that  
\begin{align}
    & \left| \int_{|\varsigma-s^1|\geq 3|s^1-s^2|}  \frac{\big( {\rm D}_s \digamma(s^1+i\be) - {\rm D}_s \digamma(s^2+i\be) \big)(s^1-\varsigma)}{|p|_\C^3} \phi(s^1-\varsigma)  \wedge w_{5-7}[\ell, \varsigma+i\be] \dd \varsigma \right| \nonumber \\
    & \quad = \left| \int_{\T^2}  \frac{(s^1-\varsigma)}{|p|_\C^3} \phi(s^1-\varsigma) \mathds{1}_{|s^1-\varsigma| \geq 3|s^1-s^2|}\, \big( {\rm D}_s\digamma(s^1+i\be) - {\rm D}_s \digamma(s^2+i\be) \big) \wedge w_{5-7}[\ell, \varsigma+i\be] \dd \varsigma \right| \nonumber  \\[0.3cm]
    & \quad \lesssim |s^1-s^2|^{\frac12} \norm{w_{5-7}[\ell,\cdot+i\beta]}_{C^\frac{1}{2}(\T^2)}, \quad \textup{for } \be \in B_\rho(0)\,. \label{E.544}
\end{align}
Hence, we have that
\begin{equation} \label{1st sing}
\begin{aligned}
 & \left| \int_{|\varsigma-s^1|\geq 3|s^1-s^2|} \frac{p-q-r+t}{|p|_{\C}^3} \phi(s^1-\varsigma) \,  \wedge w_{5-7}[\ell, \varsigma+i\be] \dd \varsigma \right| \\[0.3cm]
 & \qquad \lesssim |s^1-s^2|^{\frac12} \norm{w_{5-7}[\ell,\cdot+i\beta]}_{C^\frac{1}{2}(\T^2)}\,, \quad \textup{for }\be \in B_\rho(0)\,.
\end{aligned}
\end{equation}

Next, observe that
\begin{equation}\begin{aligned}
    & \frac{3p\scalar{p}{p-q+r-t}_\C}{|p|_\C^5} =  \frac{3p\scalar{p}{\digamma(s^1+i\beta)-\digamma(s^2+i\beta)-{\rm D}_s\digamma(s^2+i\beta)(s^1-s^2)}_\C}{|p|_\C^5} \\
    & \quad + \frac{3p\scalar{p}{ ( {\rm D}_s\digamma(s^1+i\be) - {\rm D}_s\digamma(s^2+i\be))(s^1-\varsigma) }_\C}{|p|_\C^5}\,.\label{bad part}
\end{aligned}\end{equation}
Moreover, using \eqref{E.complex-modulus}, one can easily check that
\begin{align*}
    &\left|\frac{3p\scalar{p}{\digamma(s^1+i\beta)-\digamma(s^2+i\beta)-{\rm D}_s\digamma(s^2+i\beta)(s^1-s^2)}_\C}{|p|_\C^5}\right| \\
    & \qquad \lesssim \left|\frac{\digamma(s^1+i\beta)-\digamma(s^2+i\beta)-{\rm D}_s\digamma(s^2+i\beta)(s^1-s^2)}{|p|_\C^3}\right|, \quad \textup{for } \be \in B_\rho(0)\,.
\end{align*}
Hence, using \eqref{E.Digamma1} and \eqref{E.Digamma2}, we get that
\begin{equation} \label{2nd sing}
\begin{aligned}
 \Bigg|  \int_{|\varsigma-s^1|\geq 3|s^1-s^2|} &  \bigg( \frac{3p\scalar{p}{\digamma(s^1+i\beta)-\digamma(s^2+i\beta)-{\rm D}_s\digamma(s^2+i\beta)(s^1-s^2)}_\C}{|p|_\C^5}\, \phi(s^1-\varsigma) \\
& \wedge  w_{5-7}[\ell, \varsigma+i\be] \bigg) \dd \varsigma \Bigg| \lesssim |s^1-s^2|^{\frac12} \norm{w_{5-7}[\ell,\cdot+i\beta]}_{L^{\infty}(\T^2)}, \quad \textup{for } \be \in B_\rho(0)\,.
\end{aligned}
\end{equation}

Now, expanding the scalar product in the other fraction in \eqref{bad part} into the different components, one can check that Lemma \ref{SingInt} is applicable to it and see that 
\begin{equation} \label{3rd sing}
\begin{aligned}
\mel\left|\int_{|\varsigma-s^1|\geq 3|s^1-s^2|}\frac{3p\scalar{p}{ ( {\rm D}_s\digamma(s^1+i\be) - {\rm D}_s\digamma(s^2+i\be))(s^1-\varsigma)}_\C}{|p|_\C^5}\phi(s^1-\varsigma) \wedge w_{5-7}[\ell, \varsigma+i\be]\dd\varsigma\right|\\[0.3cm]
&\lesssim \left|{\rm D}_s\digamma(s^1+i\be) - {\rm D} _s\digamma(s^2+i\be)\right|\norm{w_{5-7}[\ell,\cdot+i\beta]}_{C^{\frac{1}{2}}(\T^2)}\\
&\lesssim|s^1-s^2|^\frac{1}{2}\norm{w_5[\ell,\cdot+i\beta]}_{C^{\frac{1}{2}}(\T^2)}, \quad \textup{for } \be \in B_\rho(0)\,.
\end{aligned}
\end{equation}
Note that in the last inequality we are again using \eqref{direct est}.

Finally, we estimate the integrals with kernels corresponding to the right-hand side in Lemma \ref{Lem double diff}. First, observe that, by \eqref{bd diff}, whenever $|s^1-\varsigma| \geq 3 |s^1-s^2|$, it follows that
\begin{align}
|r-t|+|p-q|\lesssim |s^1-\varsigma|^\frac{3}{2}\,, \quad \textup{for }\be  \in B_\rho(0)\,.\label{rt bd}
\end{align}
Likewise, using \eqref{direct est}, we get that
\begin{equation}\begin{aligned}
 |p-r| &\lesssim \left|{\rm D}_s\digamma(s^1+i\beta)(s^1-\varsigma)-{\rm D}_s\digamma(s^2+i\beta)(s^2-\varsigma)\right|+\eps\left|\zeta(s^1+i\beta)-\zeta(s^2+i\beta)\right| \\
&\lesssim |s^1-\varsigma||s^1-s^2|^\frac{1}{2}+|s^1-s^2|\,, \quad \textup{for }\be \in B_\rho(0)\,, \label{pr bd}
\end{aligned}\end{equation}
and that, for $\be \in B_\rho(0)$,
 \begin{equation}\begin{aligned}
|q-t|&\lesssim \left|\digamma(s^1+i\beta)-\digamma(s^2+i\beta)\right|+\eps\left|\zeta(s^1+i\beta)-\zeta(s^2+i\beta)\right| \lesssim |s^1-s^2|. \label{qt bd}
\end{aligned}\end{equation}

We also note, thanks to \eqref{bd d low} and the elementary estimate $\Re \sqrt{z}\geq \sqrt{\Re z}$, which holds for all $z\in \C$ with $\Re z>0$, that \begin{align}
\min(\Re |p|_{\C},\Re |q|_{\C},\Re |r|_{\C},\Re |t|_{\C})\gtrsim |s^1-\varsigma|+\eps|l-\ell|\gtrsim |s^1-\varsigma|.
\end{align}

\noindent Hence, using the shortened notation
\begin{equation} \label{E.Kernelpqrt}
{\rm K}(p,q,r,t) :=  \frac{\max\{|p|,|q|,|r|,|t|\}^7}{\min\{||p|_{\C}|,||q|_{\C}|,||r|_{\C}|,||t|_{\C}|\}^{11}}\big(|p-q|+|r-t|\big)\big(|p-r|+|q-t|\big)\,,
\end{equation}
we infer from  \eqref{bd d up}, \eqref{bd d low}, \eqref{rt bd}, \eqref{pr bd} and \eqref{qt bd} that
\begin{equation}
\begin{aligned} & \left|
\int_{|s^1-\varsigma|\geq 3|s^1-s^2|} {\rm K}(p,q,r,t) \,\phi(s^1-\varsigma) \left|w_{5-7}[\ell, \varsigma+i\be]\right| \dd \varsigma \right|\\
& \qquad \lesssim \norm{w_{5-7}[\ell,\cdot + i \be]}_{L^{\infty}(\T^2)}\, \int_{|s^1-\varsigma| \geq 3|s^1-s^2| }|s^1 - \varsigma|^{-\frac{5}{2}}\big(|s^1-\varsigma||s^1-s^2|^\frac{1}{2}+|s^1-s^2|\big)\dd \varsigma \\
& \qquad\lesssim |s^1-s^2|^\frac{1}{2}\norm{w_{5-7}[\ell,\cdot + i \be]}_{L^{\infty}(\T^2)} \,, \quad \textup{for } \be \in B_\rho(0)\,. \label{rhs bd}
\end{aligned}
\end{equation}
Combining \eqref{1st sing}, \eqref{2nd sing}, \eqref{3rd sing}, \eqref{rhs bd} and Lemma \ref{Lem double diff}, we get that 
\begin{equation} \label{E.conclussionI2J22}
     {\rm I}_2 \lesssim |s^1-s^2|^{\frac12} \norm{w_{5-7}[\ell,\cdot+i\beta]}_{C^{\frac12}(\T^2)}\,, \quad \textup{for } \be \in B_\rho(0)\,.
\end{equation}
The result follows from \eqref{E.J2bounded}, \eqref{E.conclussionI1J2}, \eqref{E.conclussionI2J2} and \eqref{E.conclussionI2J22}.
\end{proof}

Having the previous lemmas at hand, we can now conclude.

\begin{proof}[Proof of \eqref{bd K}]
Having at hand Lemma \ref{L.holomorphicExtension}, and the decomposition in \eqref{splitting J}, \eqref{bd K} immediately follows from Lemmas \ref{L.J1} and \ref{L.J2}.
\end{proof}

\subsection{Proof of (\ref{bd K lip})} We again only need to consider $l\neq \ell$, as the rest is a zero set. 

Since the proof is similar to the one of \eqref{bd K}, we will skip some details. First of all, taking into account \eqref{E.holomorphicKepsilon}, we see that
\begin{equation} \label{E.decompKLipschitz}
\begin{aligned}
    & K_{\ep}^c\left[w_1^1[\ell,\cdot],\, \int_l^\ell w_4^1[\mu,\cdot] \dd \mu,\,w_{5-7}^1[\ell,\cdot]\right](s+i\beta) \\
        & \qquad- K_{\ep}^c\left[w_1^2[\ell,\cdot],\,\int_l^\ell w_4^2[\mu,\cdot] \dd \mu,\,w_{5-7}^2[\ell,\cdot]\right](s+i\beta) \\
    & \quad = K_{\ep}^c\left[w_1^1[\ell,\cdot],\, \int_l^\ell w_4^1[\mu,\cdot] \dd \mu,\,w_{5-7}^1[\ell,\cdot] - w_{5-7}^2[\ell,\cdot]\right](s+i\beta) \\
    & \qquad + K_{\ep}^c\left[w_1^1[\ell,\cdot],\, \int_l^\ell w_4^1[\mu,\cdot] \dd \mu,\,w_{5-7}^2[\ell,\cdot]\right](s+i\beta) \\
    & \qquad- K_{\ep}^c\left[w_1^2[\ell,\cdot],\, \int_l^\ell w_4^2[\mu,\cdot] \dd \mu,\,w_{5-7}^2[\ell,\cdot]\right](s+i\beta)\,.
\end{aligned}
\end{equation}
The first term on the right-hand side can be directly estimated using the proof of \eqref{bd K}. Indeed, we have the following:

\begin{lemma} \label{L.LipschitzTrivial}
Let $\mathcal{C}_1$ be as in Proposition \ref{main est}. If $w^1, w^2 \in \cY^7_\rho$ satisfy \eqref{E.smallness main est}, \eqref{E.smallness main est 2}, \eqref{E.Zrho1} and \eqref{E.Zrho2}, then
$$
\begin{aligned}
\sup_{l,\ell \in [-1,1]} & \norm{K_{\ep}^c\left[w_1^1[\ell,\cdot],\, \int_l^\ell w_4^1[\mu,\cdot] \dd \mu,\,w_{5-7}^1[\ell,\cdot] - w_{5-7}^2[\ell,\cdot]\right]}_{\cX_\rho^3} \lesssim \norm{w_{5-7}^1-w_{5-7}^2}_{\cY_\rho^3},
\end{aligned}
$$
with an implicit constant that depends only on $\cC_1$ but not on $\rho$ or $\ep$.
\end{lemma}

We now focus on the other two terms and analyze them together. First, we prove the following:

\begin{lemma} \label{L.LipschitzLinfty}
Let $\mathcal{C}_1$ be as in Proposition \ref{main est}. If $w^1, w^2 \in \cY^7_\rho$ satisfy \eqref{E.smallness main est}, \eqref{E.smallness main est 2}, \eqref{E.Zrho1} and \eqref{E.Zrho2}, then
$$
\begin{aligned}
\sup_{l,\ell \in [-1,1]}  \bigg\{ \sup_{|\be| \leq \rho} \bigg\| &  K_{\ep}^c\left[w_1^1[\ell,\cdot+i\beta],\, \int_l^\ell w_4^1[\mu,\cdot+i\beta] \dd \mu,\,w_{5-7}^2[\ell,\cdot+i\beta]\right] \\
    & - K_{\ep}^c\left[w_1^2[\ell,\cdot+i\beta],\, \int_l^\ell w_4^2[\mu,\cdot+i\beta] \dd \mu,\,w_{5-7}^2[\ell,\cdot+i\beta]\right] \bigg\|_{L^{\infty}(\T^2)}
 \bigg\} \lesssim\, \|w^1-w^2\|_{\cY^7_\rho}\,,
\end{aligned}
$$
with an implicit constant that depends only on $\cC_1$ but not on $\rho$ or $\ep$. 
\end{lemma}

\begin{proof}
Taking into account \eqref{splitting J}, we analyze the part corresponding to each kernel $J_j$ separately. With a slight abuse of notation, we write
\begin{equation} \label{E.JjwDependece}
\begin{aligned}
    & J_1(s-\varsigma, s+i\be;\, w) := \frac{\tilde{d}(s-\varsigma,s+i\beta;\, w)}{\big|\tilde{d}(s-\varsigma,s+i\beta;\, w)\big|_{\C}^3}\,\phi(s-\varsigma)\,,\\
    & J_2(s+i\beta,\varsigma+i\beta;\, w) := \left(\frac{d_{l,\ell}^{\eps}(s+i\beta,\varsigma+i\beta;\, w)}{\big|d_{l,\ell}^{\eps}(s+i\beta,\varsigma+i\beta;\, w)\big|_{\C}^3}-\frac{\tilde{d}(s-\varsigma,s+i\beta;\, w)}{\big|\tilde{d}(s-\varsigma,s+i\beta;\, w)\big|_{\C}^3}\right)\phi(s-\varsigma)\,,
\end{aligned}
\end{equation}
to emphasize the dependence on $w^1$ and $w^2$. Also, throughout the proof, we assume that $|\be| \leq \rho \leq \rho_0$.

We first analyze the part corresponding to $J_1$. To that end, we set
$$
a^j = a^j(s+i\be) := \nabla_s w_1^j [\ell,s+i\be] \quad \textup{and} \quad b^j e_3 = b^j(s+i\be)e_3 = |l-\ell|^{-1}\ze(s+i\be;\, w^j)\,.
$$
Recall that by Lemmas \ref{SingInt} and \ref{L.ab}, the convolution kernel $J_1$ is Lipschitz in these coefficients with respect to the $L(C^\frac{1}{2},C^\frac{1}{2})$-topology.  Hence, similarly as above in \eqref{E.a1a2diff}-\eqref{E.conclussionI2J1}, it follows from Lemmas \ref{SingInt} and \ref{L.ab} that
\begin{align*}
\mel\sup_{l,\ell \in [-1,1]} \norm{\int_{\T^2} \left(J_1(\cdot-\varsigma,\cdot+i\beta, w^1)-J_1(\cdot-\varsigma,\cdot+i\beta, w^2)\right)\times w_{5-7}^2[\ell,\varsigma+i\beta]\dd\varsigma}_{L^\infty(\TT^2)}\\
&\lesssim \left( |a^1-a^2|+\max_{t\in [0,1]}\frac{|b^1-b^2|}{|tb^1+(1-t)b^2|} \right) \norm{w_{5-7}[\ell,\cdot+i\beta]}_{C^\frac{1}{2}(\T^2)}\,.
\end{align*}
Now observe that, by the definitions \eqref{E.digamma} and \eqref{E.zeta} of $\digamma$ and $\zeta$, it holds that,
\begin{equation}  \label{E.a1a2diff2}
|a^1-a^2| = |{\rm D}_s \digamma(s+i\be, w^1) - {\rm D}_s \digamma(s+i\be, w^2)| \lesssim \norm{w^1-w^2}_{\cY_\rho^7}\,,
\end{equation}
and
$$
|b^1-b^2| = |l-\ell|^{-1}|\zeta(s+i\be, w^1) - \zeta(s+i\be, w^2)| \lesssim \norm{w^1-w^2}_{\cY_\rho^7}\,.
$$
Finally, taking into account \eqref{E.complex-modulus} and the definition of $b$, we infer that
\begin{equation} \label{E.lowerbll2}
|tb^1+(1-t)b^2| \gtrsim  1\,,
\end{equation}
uniformly in $t \in [0,1]$. Combining these estimates, we conclude that \begin{align}\begin{aligned}\label{lipbdj1}
\sup_{l,\ell \in [-1,1]} & \norm{\int_{\T^2} \left(J_1(\cdot-\varsigma,\cdot+i\beta, w^1)-J_1(\cdot-\varsigma,\cdot+i\beta, w^2)\right)\times w_{5-7}^2[\ell,\varsigma+i\beta]\dd\varsigma}_{L^\infty(\TT^2)}\\[0.3cm]
&\lesssim \norm{w^1-w^2}_{\cY_\rho^7}.
\end{aligned}\end{align}

Next, to analyze the part corresponding to $J_2$, we rely on Lemma \ref{Lem double diff} a). First of all, we set
\begin{align*}
    & p = p(\varsigma,s+i\be;\, w^1) := \tilde{d}(s-\varsigma, s+i\be;\, w^1)\,,\\
    & q = q(\varsigma, s+i\be;\, w^1) := d_{l,\ell}^\ep(s+i\be, \varsigma+i\be;\, w^1)\,,\\
    & r = r(\varsigma,s+i\be;\, w^2) :=  \tilde{d}(s-\varsigma, s+i\be;\, w^2)\,, \\
    & t = t(\varsigma,s+i\be;\, w^2) := d_{l,\ell}^\ep(s+i\be, \varsigma+i\be;\, w^2)\,,
\end{align*}
and observe that
\begin{align*}
    J_2(s-\varsigma, \varsigma+i\be;\, w^1)- J_2(s-\varsigma, \varsigma+i\be;\, w^2) = - \left( \frac{p}{|p|_\C^3} - \frac{q}{|q|_\C^3} - \frac{r}{|r|_\C^3} + \frac{t}{|t|_\C^3} \right) \phi(s-\varsigma)\,.
\end{align*}
Having this expression at hand, we argue as in the proof of Lemma \ref{L.J2}. First, observe that
$$
|p-q-r+t| \lesssim |s-\varsigma|^{\frac32} \|w^1-w^2\|_{\cY^7_\rho}\,.
$$
Then, by \eqref{bd d low} (see also the proof of \eqref{E.Digamma2}), we get that
\begin{align} \label{E.J2Lipschitz1}
    \left| \int_{\T^2}  \phi(s-\varsigma)\frac{p-q-r+t}{|p|_\C^3} \, \times w_{5-7}^2[\ell,\varsigma+i\be] \dd \varsigma \right|  \lesssim \|w^1-w^2\|_{\cY^7_\rho}\,.
\end{align}
Likewise, we have that
\begin{align}
   \left| \frac{3p\scalar{p}{p-q+r-t}_\C}{|p|_\C^5} \right| \lesssim \frac{\|w^1-w^2\|_{\cY^7_\rho}}{|s-\varsigma|^{\frac32}}\,,
\end{align}
and so that 
\begin{align} \label{E.J2Lipschitz2}
    \left| \int_{\T^2}  \frac{3p\scalar{p}{p-q+r-t}_\C}{|p|_\C^5} \, \phi(s-\varsigma) \times w_{5-7}^2[\ell,\varsigma+i\be] \dd \varsigma \right|  \lesssim \|w^1-w^2\|_{\cY^7_\rho}\,.
\end{align}

On the other hand, by \eqref{bd diff}, it follows that
$$
|p-q| + |r-t| \lesssim |s-\varsigma|^{\frac32}\,,
$$
and, using \eqref{direct est} together with the linearity of these expressions in $w^1,w^2$, we get that
$$
|q-t|+|p-r| \lesssim \big(|s-\varsigma|+\ep |l-\ell| \big) \|w^1-w^2\|_{\cY^7_\rho}\,.
$$
Using then the shortened notation introduced in \eqref{E.Kernelpqrt}, we conclude that
\begin{align} \label{E.J2Lipschitz3}
  \sup_{l,\ell \in [-1,1]}   \left\| \int_{\T^2} {\rm K}(p,q,r,t) \,\phi(\cdot-\varsigma) |w_{5-7}[\ell, \varsigma+i\be]| \dd \varsigma \right\|_{L^{\infty}(\TT^2)} \lesssim \|w^1-w^2\|_{\cY^7_\rho}.
\end{align}
Combining \eqref{E.J2Lipschitz1}, \eqref{E.J2Lipschitz2} and \eqref{E.J2Lipschitz3} with Lemma \ref{Lem double diff}, we then obtain that
\begin{equation} \label{E.J2LinftyLipschitz}
\begin{aligned}
    \sup_{l,\ell \in [-1,1]} &  \norm{ \int_{\T^2} \big( J_2(\cdot-\varsigma,\cdot+i\be;\, w^1) - J_2(\cdot-\varsigma,\cdot+i\be;\, w^2) \big) \times w_{5-7}^2[\ell,\varsigma+i\be] \dd \varsigma}_{L^{\infty}(\T^2)}  \\[0.3cm]
    &   \lesssim \|w^1-w^2\|_{\cY^7_\rho}.
\end{aligned}
\end{equation}
The result follows combining \eqref{lipbdj1} and \eqref{E.J2LinftyLipschitz}.
\end{proof}

We now turn towards the boundedness of the corresponding Hölder seminorm. Using the notation introduced in \eqref{E.JjwDependece}, we first prove the following:
 
\begin{lemma} \label{L.J1LipschitzHolder}
Let $\mathcal{C}_1$ be as in Proposition \ref{main est}. If $w^1, w^2 \in \cY^7_\rho$ satisfy \eqref{E.smallness main est}, \eqref{E.smallness main est 2}, \eqref{E.Zrho1} and \eqref{E.Zrho2}, then
\begin{equation*} 
\begin{aligned}
    \sup_{l,\ell \in [-1,1]} & \left\{ \sup_{|\be| \leq \rho} \bigg[ \int_{\T^2} \big( J_1(\cdot-\varsigma,\cdot+i\be;\, w^1) - J_1(\cdot-\varsigma,\cdot+i\be;\, w^2) \big) \times w_{5-7}^2[\ell,\varsigma+i\be] \dd \varsigma\bigg]_{C^{\frac12}(\T^2)} \right\} \\[0.3cm]
    & \lesssim \|w^1-w^2\|_{\cY^7_\rho}\,,
\end{aligned}
\end{equation*}
with an implicit constant that depends only on $\cC_1$ but not on $\rho$ or $\ep$.
\end{lemma}

The proof of this result relies on the following elementary lemma, which is an immediate consequence of the fundamental theorem of calculus.
 
\begin{lemma} \label{L.FTC}
Let $B\subset \R^{n}$ be a closed convex set and let $g: B\rightarrow X$ be $C^2$ for some (real) Banach space $X$. Then for $y_1,y_2,y_3,y_4\in B$ it holds that
\begin{align*}
|g(y_1)-g(y_2)-g(y_3)+g(y_4)|_X &\lesssim  \| g\|_{W^{2,\infty}(B,X)}\Big( |y_1-y_2-y_3+y_4|\\
&\quad+ \big(|y_1-y_2|+|y_3-y_4| \big)\big(|y_1-y_3| + |y_2-y_4| \big)\Big)\,.
\end{align*}
\end{lemma}

\begin{proof}[Proof of Lemma \ref{L.J1LipschitzHolder}]
We start by pointing out that there is no loss of generality in considering $s^1,s^2 \in \T^2$ with $|s^1-s^2| \leq \frac1{10}$. The case where $|s^1-s^2|\geq \frac1{10}$ immediately follows from Lemma \ref{L.LipschitzLinfty}. Also, we introduce the shortened notation
$$
a_j^h := a(s^j+i\be;\, w^h) = \nabla_s w_1^h[\ell,s^j+i\be] \,, \quad b_j^h e_3 := b(s^j + i \be;\, w^h)e_3 = |l-\ell|^{-1}\ze(s^j+i\be;\, w^h)\,, 
$$
for $j,h \in \{1,2\}$, and recall that the convolution kernel $J_1$ is $C^2$ (w.r.t.\ $L(C^\frac{1}{2},C^\frac{1}{2})$)  in these parameters by the Lemmas \ref{SingInt} and \ref{L.ab}. Moreover, let us stress that, considering separately the cases where $l < \ell $ and where $l > \ell$, one can easily check the convexity of the set $B$ defined by \eqref{arg comp modulus}.

\allowdisplaybreaks
If we now consider the linear map associated with the convolution with $J_1$ as a function of $a$ and $b$ to $L(C^\frac{1}{2},C^\frac{1}{2})$, then we can apply Lemma \ref{L.FTC} twice and obtain that
\begin{equation}  \label{E.conclussionJ1LipLip}
\begin{aligned}
    & \bigg| \int_{\T^2} \Big( J_1(s^1-\varsigma,s^1+i\be;\, w^1) - J_1(s^1-\varsigma,s^1+i\be;\, w^2) \\
    & \qquad- J_1(s^2-\varsigma,s^2+i\be;\, w^1) + J_1(s^2-\varsigma,s^2+i\be;\, w^2)  \Big)  \times w_{5-7}^2[\ell,\varsigma+i\be] \dd \varsigma \bigg| \\
    & \lesssim \Big(|a_1^1-a_1^2-a_2^1+a_2^2| + (|a_1^1-a_1^2|+|a_2^1-a_2^2|)(|a_1^1-a_2^1|+|a_1^2-a_2^2|) \\
    & \qquad + |b_1^1-b_1^2-b_2^1+b_2^2| + (|b_1^1-b_1^2|+|b_2^1-b_2^2|)(|b_1^1-b_2^1|+|b_1^2-b_2^2|) \\
    & \qquad + |s^1-s^2|^{\frac12} \Big) \norm{w_{5-7}[\ell,\cdot+i\beta]}_{C^\frac{1}{2}(\T^2)}\,.
\end{aligned}
\end{equation}
Let us point out that here we are implicitly using that a lower bound similar to \eqref{E.lowerbll2} holds uniformly in the set $B$ defined by \eqref{arg comp modulus}. 

Next, using  \eqref{direct est} and the fact that the $a$'s and $b$'s are affine in $w$, we get that
\begin{align*}
&|a_1^1-a_1^2-a_2^1+a_2^2|+|b_1^1-b_1^2-b_2^1+b_2^2|\lesssim |s^1-s^2|^\frac{1}{2}\norm{w^1-w^2}_{\cY_\rho^7}\,,\\
&|a_1^1-a_1^2|+|a_2^1-a_2^2| + |b_1^1-b_1^2|+|b_2^1-b_2^2|\lesssim \norm{w^1-w^2}_{\cY_\rho^7}\,,\\
&|a_1^1-a_2^1|+|a_1^2-a_2^2|+ |b_1^1-b_2^1|+|b_1^2-b_2^2|\lesssim |s^1-s^2|^\frac{1}{2}\,.
\end{align*}
The result follows combining \eqref{E.conclussionJ1LipLip} with these estimates. 
\end{proof}

\begin{lemma} \label{L.J2LipschitzHolder}
Let $\mathcal{C}_1$ be as in Proposition \ref{main est}. If $w^1,\, w^2 \in \cY^7_\rho$ satisfy \eqref{E.smallness main est}, \eqref{E.smallness main est 2}, \eqref{E.Zrho1} and \eqref{E.Zrho2}, then 
\begin{equation*}  
\begin{aligned}
    \sup_{l,\ell \in [-1,1]} & \left\{ \sup_{|\be| \leq \rho} \bigg[ \int_{\T^2} \big( J_2(\cdot-\varsigma,\cdot+i\be;\, w^1) - J_2(\cdot-\varsigma,\cdot+i\be;\, w^2) \big) \wedge w_{5-7}^2[\ell,\varsigma+i\be] \dd \varsigma\bigg]_{C^{\frac12}(\T^2)} \right\} \\[0.3cm]
    & \lesssim \|w^1-w^2\|_{\cY^7_\rho}\,,
\end{aligned}
\end{equation*}
with an implicit constant that depends only on $\cC_1$ but not on $\rho$ or $\ep$.
\end{lemma}

\begin{proof}
We start by pointing out that there is no loss of generality in considering $s^1, s^2 \in \T^2$ with $|s^1-s^2| \leq \frac1{32}$. The case where $|s^1-s^2| \geq \frac{1}{32}$ follows from Lemma \ref{L.LipschitzLinfty}. Now, observe that
\begin{equation} \label{E.J2LipschitzBeginning}
\begin{aligned}
    & \bigg| \int_{\T^2} \Big( J_2(s^1-\varsigma,s^1+i\be;\, w^1) - J_2(s^1-\varsigma,s^1+i\be;\, w^2) \\
    & \qquad- J_2(s^2-\varsigma,s^2+i\be;\, w^1) + J_2(s^2-\varsigma,s^2+i\be;\, w^2)  \Big)  \times w_{5-7}^2[\ell,\varsigma+i\be] \dd \varsigma \bigg| \\
    & \leq \bigg|\int_{|\varsigma-s_1| \leq 3|s^1-s^2|} \big( J_2(s^1-\varsigma,s^1+i\be;\, w^1) - J_2(s^1-\varsigma,s^1+i\be;\, w^2)\big)\times w_{5-7}^2[\ell,\varsigma+i\be] \dd \varsigma \bigg| \\
    & \quad + \bigg|\int_{|\varsigma-s_2| \leq 4|s^1-s^2|} \big( J_2(s^2-\varsigma,s^2+i\be;\, w^1) - J_2(s^2-\varsigma,s^2+i\be;\, w^2)\big) \times w_{5-7}^2[\ell,\varsigma+i\be] \dd \varsigma \bigg| \\
    & \quad + \bigg| \int_{|\varsigma-s^1| \geq 3|s^1-s^2|} \Big( J_2(s^1-\varsigma,s^1+i\be;\, w^1) - J_2(s^1-\varsigma,s^1+i\be;\, w^2) \\
    & \qquad- J_2(s^2-\varsigma,s^2+i\be;\, w^1) + J_2(s^2-\varsigma,s^2+i\be;\, w^2)  \Big)  \times w_{5-7}^2[\ell,\varsigma+i\be] \dd \varsigma \bigg|\,.
\end{aligned}
\end{equation}
The first two terms can be handled as we did in Lemma \ref{L.LipschitzLinfty} to prove \eqref{E.J2LinftyLipschitz}, the main difference being the domain where we are integrating. This difference is precisely what makes appear the extra factor $|s_1-s_2|^{\frac12}$, and allows us to prove that
\begin{equation} \label{E.J2LipschitzLipschitz-1}
\begin{aligned}
     & \bigg|\int_{|\varsigma-s_1| \leq 3|s^1-s^2|} \big( J_2(s^1-\varsigma,s^1+i\be;\, w^1) - J_2(s^1-\varsigma,s^1+i\be;\, w^2)\big)\times w_{5-7}^2[\ell,\varsigma+i\be] \dd \varsigma \bigg| \\
    & \quad + \bigg|\int_{|\varsigma-s^2| \leq 4|s^1-s^2|} \big( J_2(s^2-\varsigma,s^2+i\be;\, w^1) - J_2(s^2-\varsigma,s^2+i\be;\, w^2)\big)\times  w_{5-7}^2[\ell,\varsigma+i\be] \dd \varsigma \bigg| \\[0.3cm]
    & \lesssim |s^1-s^2|^{\frac12} \|w^1-w^2\|_{\cY^7_\rho}\,, \quad \textup{for } \be \in B_\rho(0)\,.
\end{aligned}
\end{equation}
Having this estimate at hand, it only remains to deal with the last term on the right-hand side of \eqref{E.J2LipschitzBeginning}. To that end, we first split the kernel there as \allowdisplaybreaks
\begin{align*}
    & J_2(s^1-\varsigma,s^1+i\be;\, w^1) - J_2(s^1-\varsigma,s^1+i\be;\, w^2) - J_2(s^2-\varsigma,s^2+i\be;\, w^1) + J_2(s^2-\varsigma,s^2+i\be;\, w^2) \\
    & = \Bigg( \frac{d_{l,\ell}^{\eps}(s^1+i\beta,\varsigma+i\beta;\, w^1)}{\big|d_{l,\ell}^{\eps}(s^1+i\beta,\varsigma+i\beta;\, w^1)\big|_{\C}^3} - \frac{\tilde{d}(s^1-\varsigma,s^1+i\beta;\, w^1)}{\big|\tilde{d}(s^1-\varsigma,s^1+i\beta;\, w^1)\big|_{\C}^3} - \frac{d_{l,\ell}^{\eps}(s^1+i\beta,\varsigma+i\beta;\, w^2)}{\big|d_{l,\ell}^{\eps}(s^1+i\beta,\varsigma+i\beta;\, w^2)\big|_{\C}^3} \\
    & \quad + \frac{\tilde{d}(s^1-\varsigma,s^1+i\beta;\, w^2)}{\big|\tilde{d}(s^1-\varsigma,s^1+i\beta;\, w^2)\big|_{\C}^3} - \frac{d_{l,\ell}^{\eps}(s^2+i\beta,\varsigma+i\beta;\, w^1)}{\big|d_{l,\ell}^{\eps}(s^2+i\beta,\varsigma+i\beta;\, w^1)\big|_{\C}^3} + \frac{\tilde{d}(s^2-\varsigma,s^2+i\beta;\, w^1)}{\big|\tilde{d}(s^2-\varsigma,s^2+i\beta;\, w^1)\big|_{\C}^3} \\
    & \quad + \frac{d_{l,\ell}^{\eps}(s^2+i\beta,\varsigma+i\beta;\, w^2)}{\big|d_{l,\ell}^{\eps}(s^2+i\beta,\varsigma+i\beta;\, w^2)\big|_{\C}^3} - \frac{\tilde{d}(s^2-\varsigma,s^2+i\beta;\, w^2)}{\big|\tilde{d}(s^2-\varsigma,s^2+i\beta;\, w^2)\big|_{\C}^3} \Bigg) \phi(s^1- \varsigma) \\
    & \quad + \Bigg(\frac{d_{l,\ell}^{\eps}(s^2+i\beta,\varsigma+i\beta;\, w^1)}{\big|d_{l,\ell}^{\eps}(s^2+i\beta,\varsigma+i\beta;\, w^1)\big|_{\C}^3} - \frac{\tilde{d}(s^2-\varsigma,s^2+i\beta;\, w^1)}{\big|\tilde{d}(s^2-\varsigma,s^2+i\beta;\, w^1)\big|_{\C}^3} \\
    & \quad - \frac{d_{l,\ell}^{\eps}(s^2+i\beta,\varsigma+i\beta;\, w^2)}{\big|d_{l,\ell}^{\eps}(s^2+i\beta,\varsigma+i\beta;\, w^2)\big|_{\C}^3} + \frac{\tilde{d}(s^2-\varsigma,s^2+i\beta;\, w^2)}{\big|\tilde{d}(s^2-\varsigma,s^2+i\beta;\, w^2)\big|_{\C}^3} \Bigg) \big( \phi(s^1- \varsigma) - \phi(s^2-\varsigma) \big)\,.
\end{align*}
The part corresponding to the second parenthesis can be analyzed using Lemma \ref{Lem double diff} a). We set
\begin{equation} \label{E.p2q2r2t2}
\begin{aligned}
    & p = p(\varsigma,s^2+i\be;\, w^1) := \tilde{d}(s^2-\varsigma, s^2+i\be;\, w^1)\,,\\
    & q = q(\varsigma, s^2+i\be;\, w^1) := d_{l,\ell}^\ep(s^2+i\be, \varsigma+i\be;\, w^1)\,,\\
    & r = r(\varsigma,s^2+i\be;\, w^2) :=  \tilde{d}(s^2-\varsigma, s^2+i\be;\, w^2)\,, \\
    & t = t(\varsigma,s^2+i\be;\, w^2) := d_{l,\ell}^\ep(s^2+i\be, \varsigma+i\be;\, w^2)\,,
\end{aligned}
\end{equation}
and point out that
\begin{align*}
    &\Bigg|\frac{d_{l,\ell}^{\eps}(s^2+i\beta,\varsigma+i\beta;\, w^1)}{\big|d_{l,\ell}^{\eps}(s^2+i\beta,\varsigma+i\beta;\, w^1)\big|_{\C}^3} - \frac{\tilde{d}(s^2-\varsigma,s^2+i\beta;\, w^1)}{\big|\tilde{d}(s^2-\varsigma,s^2+i\beta;\, w^1)\big|_{\C}^3} \\
    & \quad - \frac{d_{l,\ell}^{\eps}(s^2+i\beta,\varsigma+i\beta;\, w^2)}{\big|d_{l,\ell}^{\eps}(s^2+i\beta,\varsigma+i\beta;\, w^2)\big|_{\C}^3} + \frac{\tilde{d}(s^2-\varsigma,s^2+i\beta;\, w^2)}{\big|\tilde{d}(s^2-\varsigma,s^2+i\beta;\, w^2)\big|_{\C}^3} \Bigg| \big| \phi(s^1- \varsigma) - \phi(s^2-\varsigma) \big| \\
    & = \left| \frac{p}{|p|_\C^3} - \frac{q}{|q|_\C^3} - \frac{r}{|r|_\C^3} + \frac{t}{|t|_\C^3} \right| \big| \phi(s^1- \varsigma) - \phi(s^2-\varsigma) \big| \\
    & \leq \left| \frac{p}{|p|_\C^3} - \frac{q}{|q|_\C^3} - \frac{r}{|r|_\C^3} + \frac{t}{|t|_\C^3} \right| \left( \int_0^1 |\nabla \phi(\la(s^1-s^2) + s^2-\varsigma)| \dd \la \right) |s^1-s^2|\,. 
\end{align*}
Taking into account \eqref{E.supportPhiprime}, one can argue as in the proof of \eqref{E.J2LinftyLipschitz} and conclude that
\begin{equation} \label{E.J2LipschitzLipschitz-2}
    \begin{aligned}
        & \left| \int_{|\varsigma-s_1|\geq 3|s^1-s^2|} \left(\frac{p}{|p|_\C^3} - \frac{q}{|q|_\C^3} - \frac{r}{|r|_\C^3} + \frac{t}{|t|_\C^3} \right) \big( \phi(s^1- \varsigma) - \phi(s^2-\varsigma) \big) \times w^2_{5-7}[\ell,\varsigma+i\be] \dd \varsigma \right| \\[0.3cm]
        & \qquad \lesssim |s_1-s_2|^{\frac12} \norm{w^1-w^2}_{\cY^7_\rho}, \quad \textup{for }\be \in B_\rho(0)\,.
    \end{aligned}
\end{equation}

To analyze the remaining term, we set
\begin{equation} \label{E.p1q1r1t1}
\begin{aligned}
    & p := \tilde{d}(s^1-\varsigma, s^1+i\be;\, w^1)\,, \qquad && p' :=  \tilde{d}(s^1-\varsigma, s^1+i\be;\, w^2)\,,\\ 
    & q := d_{l,\ell}^\ep(s^1+i\be, \varsigma+i\be;\, w^1)\,, \qquad && q' := d_{l,\ell}^\ep(s^1+i\be, \varsigma+i\be;\, w^2)\,,  \\
    & r :=  \tilde{d}(s^2-\varsigma, s^2+i\be;\, w^1)\,, && r' := \tilde{d}(s^2-\varsigma, s^2+i\be;\, w^2)\,, \\
    & t := d_{l,\ell}^\ep(s^2+i\be, \varsigma+i\be;\, w^1)\,, && t' := d_{l,\ell}^\ep(s^2+i\be, \varsigma+i\be;\, w^2)\,.
\end{aligned}
\end{equation}
The corresponding kernel can then be written as
\begin{equation}
    -\bigg( \frac{p}{|p|_\C^3} - \frac{q}{|q|_\C^3} - \frac{r}{|r|_\C^3} + \frac{t}{|t|_\C^3} - \frac{p'}{|p'|_\C^3} + \frac{q'}{|q'|_\C^3} + \frac{r'}{|r'
    |_\C^3} - \frac{t'}{|t'|_\C^3} \bigg) \phi(s_1-\varsigma)\,,
\end{equation}
to which we can apply Lemma \ref{Lem double diff} b). 

Arguing similarly as for \eqref{E.J2Lipschitz3}, one directly obtains from \eqref{direct est} and \eqref{bd d low} that the convolution with the kernel corresponding to the right-hand side in  \eqref{td est} can be bounded (up to a constant) by
$$
|s^1-s^2|^{\frac12} \norm{w^1-w^2}_{\cY_\rho^7}\,.
$$
It thus remains to treat the parts corresponding to the differences 
$$
\frac{p-q-r+t}{|p|_{\C}^3}-\frac{p'-q'-r'+t'}{|p'|_{\C}^3} \quad \textup{and} \quad \frac{3p\scalar{p}{q-p+r-t}_{\C}}{|p|_{\C}^5}-\frac{3p'\scalar{p'}{q'-p'+r'-t'}_{\C}}{|p'|_{\C}^5}\,.
$$

We show in detail how to deal with the first difference, the other one being similar. To that end, we divide the difference as  
\begin{align*}
&\frac{p-q-r+t}{|p|_{\C}^3}-\frac{p'-q'-r'+t'}{|p'|_{\C}^3} \\
&\quad =\frac{p-q-r+t-(p'-q'-r'+t')}{|p|_{\C}^3}+(p'-q'-r'+t')\left(\frac{1}{|p|_{\C}^3}-\frac{1}{|p'|_{\C}^3}\right)
\end{align*}
The convolution corresponding to the first term on the right-hand side can be treated exactly as in \eqref{fracc dd}-\eqref{1st sing} above. The difference $w^1-w^2$ appears because all the terms are affine in $w^1$ and $w^2$. In short, we get that
\begin{align*}
& \left|\int_{|\varsigma-s^1|\geq  3 |s^1-s^2|}\frac{p-q-r+t-(p'-q'-r'+t')}{|p|_{\C}^3} \phi(s^1-\varsigma)\times w_{5-7}^2[\ell,\varsigma+i\beta]\dd\varsigma\right| \\[0.3cm]
& \quad \lesssim |s^1-s^2|^\frac{1}{2}\norm{w^1-w^2}_{\cY_\rho^7}\,.
\end{align*}

Next, similarly to \eqref{fracc dd}, we expand
\begin{align}\begin{aligned}\label{frac ddd}
   p'-q'-r'+t'=&- \big( \digamma(s^1+i\beta,w^2)-\digamma(s^2+i\beta,w^2)-{\rm D}_s\digamma(s^2+i\beta,w^2)(s^1-s^2) \big) \\
    &  +   \big( {\rm D}_s\digamma(s^1+i\be,w^2) - {\rm D}_s\digamma(s^2+i\be,w^2) \big)(s^1-\varsigma)\,.
\end{aligned}\end{align}

Then, observe that by \eqref{bas est2}, \eqref{direct est}, \eqref{bd d up}  and \eqref{bd d low}, it follows that 
\begin{align*}
\cK_1(s^1,s^2,\beta)  := &\ \left|\frac{1}{|p|_{\C}^3}-\frac{1}{|p'|_{\C}^3}\right|\left|\digamma(s^1+i\beta,w^2)-\digamma(s^2+i\beta,w^2)-{\rm D}_s\digamma(s^2+i\beta,w^2)(s^1-s^2)\right| \\
 \lesssim & \ \frac{||p|_{\C}^3-|p'|_{\C}^3|}{||p|_{\C}^3|||p'|_{\C}^3|}|s^1-s^2|^\frac{3}{2} \lesssim \frac{1}{(|s^1-\varsigma|+\eps|l-\ell|)^3}\,|s^1-s^2|^\frac{3}{2} \norm{w^1-w^2}_{\cY_\rho^7}\,,
\end{align*}
which in turn implies that 
\begin{align*}
&\int_{|\varsigma-s^1|\geq  3|s^1-s^2|} \cK_1(s^1,s^2,\beta)  |w_{5-7}^2 [\ell,\varsigma+i\beta]|\dd\varsigma \\
& \quad \lesssim \norm{w_{5-7}^2[\ell,\cdot+i\beta]}_{L^\infty(\T^2)}\,|s^1-s^2|^\frac{3}{2}\norm{w^1-w^2}_{\cY_\rho^7} \int_{|\varsigma-s^1|\geq  3|s^1-s^2|}  \frac{\dd \varsigma}{|\varsigma-s^1|^3} \\
& \quad \lesssim |s^1-s^2|^\frac{1}{2}\norm{w^1-w^2}_{\cY_\rho^7} \,.
\end{align*}

Concerning the other summand in \eqref{frac ddd}, we appeal again to the differentiability statement in Lemma \ref{SingInt} and, arguing as in \eqref{E.544}-\eqref{1st sing},  conclude that 
\begin{align*}
& \biggl|\int_{|\varsigma-s^1|\geq 3|s^1-s^2|} \Big( {\rm D}_s \digamma(s^1+i\be,w^2) \\
& \quad - {\rm D}_s \digamma(s^2+i\be,w^2) \Big)(s^1-\varsigma)\left(\frac{1}{|p|_\C^3}-\frac{1}{|p'|_\C^3}\right) \phi(s^1-\varsigma) \wedge w_{5-7}^2[\ell, \varsigma+i\be] \dd \varsigma \biggr|  \\
& \quad \lesssim |s^1-s^2|^\frac{1}{2}\norm{w^1-w^2}_{\cY_\rho^7}\,.
\end{align*}

Combining these estimates, we conclude that
\begin{align*}
& \left|\int_{|\varsigma-s^1|\geq  3 |s^1-s^2|} \left( \frac{p-q-r+t}{|p|_{\C}^3}-\frac{p'-q'-r'+t'}{|p'|_{\C}^3} \right) \phi(s^1-\varsigma)\times w_{5-7}^2[\ell,\varsigma+i\beta]\dd\varsigma\right| \\[0.3cm]
& \quad \lesssim |s^1-s^2|^\frac{1}{2}\norm{w^1-w^2}_{\cY_\rho^7}\,,
\end{align*}
as desired.  
\end{proof}

Having the previous lemmas at hand, we can now conclude.

\begin{proof}[Proof of \eqref{bd K lip}]
Taking into account \eqref{E.decompKLipschitz} and  \eqref{E.JjwDependece}, the result immediately follows from the Lemmas \ref{L.LipschitzLinfty}, \ref{L.J1LipschitzHolder}, and \ref{L.J2LipschitzHolder}.
\end{proof}

\subsection{Proof of Proposition \ref{est kfar}}\label{sec bd kfar} This subsection is devoted to dealing with the far away part of the kernel. More precisely, we prove Proposition \ref{est kfar}. We deal only with the infinite sum. The corresponding estimate for the first summand is easier and can be handled similarly. We first deal with the proof of \eqref{bd Kfar}.

\begin{proof}[Proof of \eqref{bd Kfar}]
    For every $k \in \ZZ^2 \setminus \{0\}$, we show that the $\cY_\rho^3$--norm of the corresponding summand is bounded by $|k|^{-3}$. This will allow us to conclude the proof by summing up. To simplify the notation, we set
    \begin{align*}
W(z,\varsigma,k):=z-(z-\varsigma-k)+w_1[\ell,z]e_3-w_1[l,z-\varsigma]e_3=\varsigma+k+w_1[\ell,z]e_3-w_1[\ell,z-\varsigma]e_3\,,
\end{align*}
and stress that, for every $k \in \ZZ^2 \setminus \{0\}$, we want to bound the $\cY_\rho^3$--norm  of
\begin{equation} \label{E.aimWk}
\int_{-1}^1 \int_{[-\tfrac12,\tfrac12]^2} \left(\frac{W(s+i\beta,\varsigma,k)}{|W(s+i\be,\varsigma,k)|_{\C}^3}+\frac{W(s+i\beta,\varsigma,-k)}{|W(s+i\be,\varsigma,-k)|_{\C}^3}\right)\wedge w_{5-7}[\ell,s+i\beta-\varsigma]\dd\varsigma\dd\ell\,,
\end{equation}
by $|k|^{-3}$. Before going any further, note that we can rewrite the corresponding kernel as 
\begin{align*}
    K_W(s+i\be,\varsigma,k) : &= \frac{W(s+i\beta,\varsigma,k)}{|W(s+i\be,\varsigma,k)|_{\C}^3}+\frac{W(s+i\beta,\varsigma,-k)}{|W(s+i\be,\varsigma,-k)|_{\C}^3} \\
    & = \frac{2( \varsigma + (w_1[\ell,s+i\be]-w_1[\ell,s+i\be - \varsigma]))\, e_3}{|W(s+i\be,\varsigma,k)|_{\C}^3} \\
    &  \quad + \frac{W(s+i\beta,\varsigma,-k)(|W(s+i\be,\varsigma,k)|_{\C}^3-|W(s+i\be,\varsigma,-k)|_{\C}^3)}{|W(s+i\be,\varsigma,k)|_{\C}^3|W(s+i\be,\varsigma,-k)|_{\C}^3}\,.
\end{align*}
Having this decomposition at hand, it is not difficult to estimate the $L^\infty$--norm of \eqref{E.aimWk}. Indeed, arguing as in the proof of Lemma \ref{bds d}, using now \eqref{im small2}, we get that
$$
\Re |W(s+i\be,\varsigma,k)|_\C^2 \gtrsim |k|^2\,,  \quad\textup{for } k \in \ZZ^2 \setminus \{0\} \textup{ and } \be \in B_\rho(0)\,.
$$
This in particular implies that
\begin{equation} \label{E.lowerk}
||W(s+i\be, \varsigma,k)|_\C| \gtrsim|k|\,, \quad\textup{for } k \in \ZZ^2 \setminus \{0\} \textup{ and } \be \in B_\rho(0)\,.
\end{equation}
Likewise, it follows that
\begin{align}
2 |\varsigma + (w_1[\ell,s+i\be]- w_1[\ell,s+i\be-\varsigma]) e_3| \lesssim 1+ \norm{w_1[\ell,\cdot+i\be]}_{L^{\infty}(\TT^2)}\,.\label{up bd w}
\end{align}
Finally, using these estimates and arguing as in the proof of \eqref{bas est2}, we get that
$$
||W(s+i\be,\varsigma,k)|_{\C}^3-|W(s+i\be,\varsigma,-k)|_{\C}^3| \lesssim \frac{1+|k|^3}{|k|} \quad\textup{for } k \in \ZZ^2 \setminus \{0\} \textup{ and } \be \in B_\rho(0)\,.
$$
Thus, we infer that
$$
|K_W(s+i\be,\varsigma,k)|  \lesssim \frac{1}{|k|^3}\,, \quad\textup{for } k \in \ZZ^2 \setminus \{0\} \textup{ and } \be \in B_\rho(0)\,,
$$
which immediately allows to estimate the $L^{\infty}$--norm of \eqref{E.aimWk} as
\begin{equation*}
    \norm{ \int_{[-\tfrac12,\tfrac12]^2} K_W(s+i\be,\varsigma,k)\wedge w_{5-7}[\ell,s+i\beta-\varsigma]\dd\varsigma}_{L^{\infty}(\T^2)} \lesssim \frac{1}{|k|^3} \,, \quad\textup{for } k \in \ZZ^2 \setminus \{0\} \textup{ and } \be \in B_\rho(0)\,. 
\end{equation*}

Next, we deal with the $C^{\frac12}$--seminorm of \eqref{E.aimWk}. We consider arbitrary $s^1, s^2 \in \TT^2$ and split the difference of the integrals as 
\begin{align*}
    & \int_{-1}^1  \int_{[-\tfrac12,\tfrac12]^2}  K_W(s^1+i\be,\varsigma,k)\wedge w_{5-7}[\ell,s^1+i\beta-\varsigma] \dd \varsigma \dd \ell \\
    & \quad -  \int_{-1}^1  \int_{[-\tfrac12,\tfrac12]^2}  K_W(s^2+i\be,\varsigma,k)\wedge w_{5-7}[\ell,s^2+i\beta-\varsigma] \dd \varsigma \dd \ell \\
    & = \int_{-1}^1  \int_{[-\tfrac12,\tfrac12]^2}  K_W(s^1+i\be,\varsigma,k) \wedge \left(w_{5-7}[\ell,s^1+i\beta-\varsigma] - w_{5-7}[\ell,s^2+i\beta-\varsigma]\right)  \dd \varsigma \dd \ell \\
    & \quad + \int_{-1}^1  \int_{[-\tfrac12,\tfrac12]^2} \Big(K_W(s^1+i\be,\varsigma,k) - K_W(s^2+i\be,\varsigma,k) \Big) \wedge w_{5-7}[\ell,s^2+i\be-\varsigma] \dd\varsigma\dd\ell =: {\rm I}_1 + {\rm I}_2\,.
\end{align*}
We can handle ${\rm I}_1$ exactly as we did with the $L^{\infty}$--norm of \eqref{E.aimWk}, and obtain that
$$
{\rm I}_1 \lesssim \frac{1}{|k|^3} |s^1-s^2|^{\frac12} \sup_{|\ell| \leq 1} \norm{w_{5-7}[\ell,\cdot+i\be]}_{C^{\frac12}(\TT^2)}\,, \quad \textup{for $k \in \ZZ^2 \setminus \{0\}$ and $\be \in B_\rho(0)$}\,.
$$
Finally, to deal with ${\rm I}_2$, we will use Lemma \ref{Lem double diff} with
\begin{align*}
& p=W(s^1+i\beta,\varsigma,k)\,,\\
& q=-W(s^1+i\beta,\varsigma,-k)\,,\\
& r=W(s^2+i\beta,\varsigma,k)\,,\\
& t=-W(s^2+i\beta,\varsigma,-k)\,.
\end{align*}
Note that, arguing as we did to estimate the $L^{\infty}$--norm of \eqref{E.aimWk}, we get that
\begin{align*}
   & |p-q|+|r-t| \lesssim 1 + \|w_1[\ell,\cdot + i \be]\|_{L^\infty(\TT^2)}\,,\\
   & |p-r|+|q-t| \lesssim |s^1-s^2|^{\frac12} \|w_1[\ell,\cdot+i \be]\|_{C^{\frac12}(\TT^2)}\,.
\end{align*}
Then, using \eqref{E.lowerk}, it is straightforward to check that
$$
\left| \frac{p-q-r+t}{|p|^3_{\CC}} \right| + \left| \frac{3p \scalar{p}{p-q-r+t}_{\CC}}{|p|^5_{\CC}} \right| \lesssim \frac{1}{|k|^3} |s^1-s^2|^{\frac12} \|w_1[\ell,\cdot+i \be]\|_{C^{\frac12}(\TT^2)} \,,  
$$
for $k \in \ZZ^2 \setminus \{0\}$ and $\be \in B_\rho(0)$. 

Similarly, it follows that
\begin{align*}
  & \frac{\max\{|p|,|q|,|r|,|t|\}^7}{\min\{\Re|p|_{\C},\Re|q|_{\C},\Re|r|_{\C},\Re|t|_{\C}\}^{11}}\big(|p-q|+|r-t|\big)\big(|p-r|+|q-t|\big) \\ 
  & \qquad \lesssim  \frac{1}{|k|^4} |s^1-s^2|^{\frac12} \|w_1[\ell,\cdot+i \be]\|_{C^{\frac12}(\TT^2)} \Big(1+  \|w_1[\ell,\cdot + i \be]\|_{L^\infty(\TT^2)}\Big)\,,  
\end{align*}
for $k \in \ZZ^2 \setminus \{0\}$ and $\be \in B_\rho(0)$. Hence, using Lemma \ref{Lem double diff}, we conclude that
$$
{\rm I}_2 \lesssim \frac{1}{|k|^3} |s^1-s^2|^{\frac12} \sup_{|\ell| \leq 1} \norm{w_{5-7}[\ell,\cdot+i\be]}_{C^{\frac12}(\TT^2)}\,, \quad \textup{for $k \in \ZZ^2 \setminus \{0\}$ and $\be \in B_\rho(0)$}\,.
$$
This concludes the proof. 
\end{proof} 

We now turn to the proof of \eqref{bd Kfar lip}. Since the proof is similar to the one of \eqref{bd Kfar}, we skip some details.  
\begin{proof}[Proof of \eqref{bd Kfar lip}]
For every $k \in \ZZ^2 \setminus \{0\}$, we show that the $\cY_\rho^3$--norm of the corresponding summand is bounded by $|k|^{-3}\norm{w^1-w^2}_{\cY_\rho^7}$. This will then allow us to conclude the proof by summing up, as in the proof of \eqref{bd Kfar}. To emphasize the dependence in $w$, for $i \in \{1,2\}$, we set
\begin{align*}
W(z,\varsigma,k;\,w^i):=z-(z-\varsigma-k)+w_1^i[\ell,z]e_3-w_1^i[l,z-\varsigma]e_3=\varsigma+k+w_1^i[\ell,z]e_3-w_1^i[\ell,z-\varsigma]e_3\,. 
\end{align*}
We start estimating the $L^\infty$-norm. We may split each summand as
\begin{align}  
& \int_{-1}^1 \int_{[-\tfrac12,\tfrac12]^2} \left(\frac{W(s+i\beta,\varsigma,k;\,w^1)}{|W(s+i\be,\varsigma,k;\,w^1)|_{\C}^3}+\frac{W(s+i\beta,\varsigma,-k;\,w^1)}{|W(s+i\be,\varsigma,-k;\,w^1)|_{\C}^3}\right)\wedge w_{5-7}^1[\ell,s+i\beta-\varsigma] \nonumber \\
&\qquad -\left(\frac{W(s+i\beta,\varsigma,k;\,w^2)}{|W(s+i\be,\varsigma,k;\,w^2)|_{\C}^3}+\frac{W(s+i\beta,\varsigma,-k;\,w^2)}{|W(s+i\be,\varsigma,-k;\,w^2)|_{\C}^3}\right)\wedge w_{5-7}^2[\ell,s+i\beta-\varsigma]\dd\varsigma\dd\ell\,,\nonumber \\
& \label{spint} \quad=\int_{-1}^1 \int_{[-\tfrac12,\tfrac12]^2} \bigg(\frac{W(s+i\beta,\varsigma,k;\,w^1)}{|W(s+i\be,\varsigma,k;\,w^1)|_{\C}^3}+\frac{W(s+i\beta,\varsigma,-k;\,w^1)}{|W(s+i\be,\varsigma,-k;\,w^1)|_{\C}^3}\\
&\qquad-\frac{W(s+i\beta,\varsigma,k;\,w^2)}{|W(s+i\be,\varsigma,k;\,w^2)|_{\C}^3}-\frac{W(s+i\beta,\varsigma,-k;\,w^2)}{|W(s+i\be,\varsigma,-k;\,w^2)|_{\C}^3}\bigg)\wedge w_{5-7}^1[\ell,s+i\beta-\varsigma]\dd\varsigma\dd\ell \nonumber \\
& \qquad+\int_{-1}^1 \int_{[-\tfrac12,\tfrac12]^2} \bigg(\frac{W(s+i\beta,\varsigma,k;\,w^2)}{|W(s+i\be,\varsigma,k;\,w^2)|_{\C}^3}+\frac{W(s+i\beta,\varsigma,-k;\,w^2)}{|W(s+i\be,\varsigma,-k;\,w^2)|_{\C}^3}\bigg)\wedge \Big(w_{5-7}^1[\ell,s+i\beta-\varsigma]\nonumber\\
&\qquad-w_{5-7}^2[\ell,s+i\beta-\varsigma]\Big)\dd\varsigma\dd\ell\,. \nonumber
\end{align}
The second integral can be treated exactly as we did to estimate the $L^{\infty}$-norm of \eqref{E.aimWk}. Hence, we only discuss first one. To bound the first integral, we shall again use Lemma \ref{Lem double diff} a). We now set 
\begin{align*}
    & p=W(s,\varsigma,k;\,w^1)\,, \quad q=-W(s,\varsigma,-k;\,w^1) \,,\quad r=W(s,\varsigma,k;\,w^2)\,, \quad t=-W(s,\varsigma,-k;\,w^2)\,.
\end{align*}
Similarly as in \eqref{E.lowerk}-\eqref{up bd w}, we note that \begin{align}
\min\left(\Re |p|_{\C},\,\Re  |q|_{\C},\, \Re |r|_{\C},\, \Re |t|_{\C}\right)\gtrsim |k|\gtrsim |p|+|q|+|r|+|t|\,.\label{bd norm fa}
\end{align}
Also, we further note that \begin{align}
&|p-q|+|r-t|\lesssim 1+\norm{w_1^1[\ell,\cdot+i\beta]}_{L^\infty(\T^2)}+\norm{w_1^2[\ell,\cdot+i\beta]}_{L^\infty(\T^2)}\label{d fa1}\,,\\
&|p-r|+|q-t|+|p-r-q+t|\lesssim \norm{w_1^1[\ell,\cdot+i\beta]-w_1^2[\ell,\cdot+i\beta]}_{L^\infty(\T^2)}\lesssim \norm{w^1-w^2}_{\cY_\rho^7}\,,\label{d fa2}
\end{align}
as one can directly from the definitions of $p,q,r$ and $t$. Applying Lemma \ref{Lem double diff} a), we then get that
\begin{align*}
& \bigg|\frac{W(s+i\beta,\varsigma,k;\,w^1)}{|W(s+i\be,\varsigma,k;\,w^1)|_{\C}^3}+\frac{W(s+i\beta,\varsigma,-k;\,w^1)}{|W(s+i\be,\varsigma,-k;\,w^1)|_{\C}^3} \\
& \quad-\frac{W(s+i\beta,\varsigma,k;\,w^2)}{|W(s+i\be,\varsigma,k;\,w^2)|_{\C}^3}-\frac{W(s+i\beta,\varsigma,-k;\,w^2)}{|W(s+i\be,\varsigma,-k;\,w^2)|_{\C}^3}\bigg| \lesssim \frac{1}{|k|^3}\norm{w^1-w^2}_{\cY_\rho^7}\,,
\end{align*}
for  $k \in \ZZ^2 \setminus \{0\}$ and $\be \in B_\rho(0)$, yielding the desired bound. 

We move on to the bound of the Hölder seminorm. We take $s, s'\in \T^2$.
In order to bound the Hölder seminorm, one can again split the integral in question as in \eqref{spint}  and in \eqref{E.decompKLipschitz}. Every term except the double difference of $W$ can be handled as above and can be bounded by $|k|^{-3}\norm{w^1-w^2}_{\cY_\rho^7}$, as desired. Hence, we shall only consider the double difference.

Aiming towards using Lemma \ref{Lem double diff} b), we set \begin{align*}
& p=W(s,\varsigma,k;\,w^1)\,, \quad q=-W(s,\varsigma,-k;\,w^1)\,, \quad r=W(s,\varsigma,k;\,w^2)\,, \quad t=-W(s,\varsigma,-k;\,w^2)\,,\\
& p'=W(s',\varsigma,k;\,w^1)\,, \quad  q'=-W(s',\varsigma,-k;\,w^1) \,,\quad r'=W(s',\varsigma,k;\,w^2)\,, \quad t'=-W(s',\varsigma,-k;\,w^2)\,.
\end{align*}
Both of these sets of variables fulfill \eqref{bd norm fa}-\eqref{d fa2}, with $s'$ appropriately replacing $s$.

We then have the following additional estimates for the relevant differences in Lemma \ref{Lem double diff} b): 
\begin{equation}
\begin{aligned}
\mel|p-p'|+|q-q'|+|r-r'|+|t-t'| \\
& \lesssim |s-s'|^\frac{1}{2}\left(\norm{w_1^1[\ell,\cdot+i\beta]}_{C^\frac{1}{2}(\T^2)}+\norm{w_1^2[\ell,\cdot+i\beta]}_{C^\frac{1}{2}(\T^2)}\right) \lesssim |s-s'|^\frac{1}{2}\label{d pp'1}\,,
\end{aligned}
\end{equation}
and 
\begin{equation}
\begin{aligned}
\mel|p-p'-r+r'|+|q-q'-t+t'| \\
& \lesssim |s-s'|^\frac{1}{2}\norm{w_1^1[\ell,\cdot+i\beta]-w_1^2[\ell,\cdot+i\beta]}_{C^\frac{1}{2}(\T^2)}\lesssim |s-s'|^\frac{1}{2}\norm{w^1-w^2}_{\cY_\rho^7}.\label{d pp'}
\end{aligned}
\end{equation}
We may further estimate \begin{align*}
\mel\left|\frac{p-q-r+t}{|p|_{\C}^3}-\frac{p'-q'-r'+t'}{|p'|_{\C}^3}\right|\\
&\leq \frac{1}{||p|_{\C}^3|}\left(|p-p'-r+r'|+|q-q'-t+t'|\right)+|p'-q'-r'+t'|\frac{||p|_{\C}^3-|p'|_{\C}^3|}{||p|_{\C}^3|||p'|_{\C}^3|}\\
&\lesssim \frac{1}{|k|^3}|s-s'|^\frac{1}{2}\norm{w^1-w^2}_{\cY_\rho^7}+\norm{w^1-w^2}_{\cY_\rho^7}|p-p'||k|^{-4}\\
&\lesssim\frac{1}{|k|^{3}}|s-s'|^{\frac12} \norm{w^1-w^2}_{\cY_\rho^7}\, ,
\end{align*}
where we have estimated the first summand in the second line with \eqref{d pp'} and \eqref{bd norm fa} and used \eqref{bas est2}, \eqref{d fa2} and \eqref{d pp'1} to estimate the second one. Similarly, one can estimate \begin{align*}
\mel\left|\frac{p\scalar{p}{p-q-r+t}_{\C}}{|p|_{\C}^5}-\frac{p'\scalar{p'}{p'-q'-r'+t'}}{|p'|_{\C}^5}\right|\lesssim  \frac{1}{|k|^{3}}|s-s'|^{\frac12} \norm{w^1-w^2}_{\cY_\rho^7}\,.
\end{align*}

Combining all these estimates with Lemma \ref{Lem double diff} b), we conclude that 
\begin{align*}
&\left|\frac{p}{|p|_{\C}^3}-\frac{q}{|q|_{\C}^3}-\frac{r}{|r|_{\C}^3}+\frac{t}{|t|_{\C}^3}-\frac{p'}{|p'|_{\C}^3}-\frac{q'}{|q'|_{\C}^3}-\frac{r'}{|r'|_{\C}^3}+\frac{t'}{|t'|_{\C}^3}\right| \lesssim \frac{1}{|k|^{3}}|s-s'|^\frac{1}{2}\norm{w^1-w^2}_{\cY_\rho^7}\,.
\end{align*}
This immediately allows us to bound the corresponding H\"older-seminorm as desired. Thus, we can conclude the proof by summing up. 
\end{proof}

\section{Proof of Proposition \ref{bd F}} \label{S.prop44}

First of all, taking into account \eqref{velo split} and Propositions \ref{main est} and \ref{est kfar}, we immediately deduce that
\begin{align}
&\norm{U_\eps(w^i)}_{\cY_\rho^3}\lesssim \norm{w^i}_{\cY_\rho^7}\,, \quad \textup{for } i \in \{1,2\}\,,\label{bd ueps}\\
&\norm{U_\eps(w^1)-U_\eps(w^2)}_{\cY_\rho^3}\lesssim \norm{w^1-w^2}_{\cY_\rho^7}\,.\label{diff ueps}
\end{align} 
Combining these estimates with Lemma \ref{cauchy est}, for $j \in \{1,\ldots,4\}$ and $i \in \{1,2\}$, one can check that
\begin{equation} \label{E.1234-1}
    \|F_j(w^i)\|_{\cY_{\rho'}} \lesssim \frac{1}{\rho-\rho'}\,, 
\end{equation}
and that
\begin{equation} \label{E.1234-2}
     \|F_j(w^1)- F_j(w^2)\|_{\cY_{\rho'}} \lesssim \frac{1}{\rho-\rho'}\, \|w^1-w^2\|_{\cY_{\rho}^7}\,,
\end{equation}
uniformly in $0 < \rho' < \rho \leq \rho_0$. 

However, when dealing with $F_{5-7}$, the term $A_\ep(w) w_{5-7}$, requires extra care. Let us start with some preliminary lemmas. The first is direct calculus, and can be proved by arguing as in Lemma \ref{Lem double diff}. 

\begin{lemma} \label{lemma abcd}
Let $a,b,c,d\in \C^3$ be such that $\min (\Re |a|_{\C},\Re |b|_{\C}, \Re |c|_{\C}, \Re |d|_{\C})>0$, then 
\begin{equation*}
\begin{aligned}
\left||a|_{\C}-|b|_{\C}-|c|_{\C}+|d|_{\C}\right|\lesssim& \frac{(|a|+|b|+|c|+|d|)^2}{\min (\Re |a|_{\C},\Re |b|_{\C}, \Re |c|_{\C}, \Re |d|_{\C})^3}\Big(|a-b-c+d|(|a|+|b|+|c|+|d|)\\
&\qquad+(|a-c|+|b-d|)|c-d|\Big)\,.
\end{aligned}
\end{equation*}
\end{lemma}

Next, for $\tau_i,\, \hat{\tau}_i$ with $i \in \{1,2\}$ as in \eqref{def taus}, we introduce the notation
$$
\tau_i(w) \quad \textup{and} \quad \hat{\tau}_i(w)\,, 
$$
to emphasize the dependence on $w$. Having this notation at hand, we prove the following key lemma.
 
\begin{lemma}\label{lem taus}
Assume that \eqref{E.cC1} and \eqref{E.smallness main est} hold, then it follows that
\begin{align}
\norm{\tau_1}_{\cY_\rho^3}+\norm{\tau_2}_{\cY_\rho^3}+\norm{\hat{\tau}_1}_{\cY_\rho^3}+\norm{\hat{\tau}_2}_{\cY_\rho^3}+\norm{\frac{1}{|\hat{\tau}_1|_{\C}}}_{\cY_\rho}+\norm{\frac{\scalar{\hat{\tau}_2}{\tau_1}_{\C}}{|\hat{\tau}_1|_{\C}}}_{\cY_{\rho}}+\norm{\frac{1}{|\hat{\tau}_2-\scalar{\hat{\tau}_2}{\tau_1}\tau_1|_{\C}}}_{\cY_\rho}\lesssim 1\,,\label{bound taus}
\end{align}
and that
\begin{equation}\begin{aligned}
\mel\norm{\tau_1(w^1)-\tau_1(w^2)}_{\cY_\rho^3}+\norm{\tau_2(w^1)-\tau_2(w^2)}_{\cY_\rho^3}+\norm{\frac{1}{|\tau_1(w^1)|_{\C}}-\frac{1}{|\tau_1(w^2)|_{\C}}}_{\cY_\rho}\\
&+\norm{\hat{\tau}_1(w^1)-\hat{\tau}_1(w^2)}_{\cY_\rho^3}+\norm{\hat{\tau}_2(w^1)-\hat{\tau}_2(w^2)}_{\cY_\rho^3} \\
&+\norm{\frac{1}{|\hat{\tau}_2(w^1)-\scalar{\hat{\tau}_2(w^1)}{\tau_1(w^1)}_{\C}\tau_1(w^1)|_{\C}}-\frac{1}{|\hat{\tau}_2(w^2)-\scalar{\hat{\tau}_2(w^2)}{\tau_1(w^2)}_{\C}\tau_1(w^2)|_{\C}}}_{\cY_\rho} \\
&+\norm{\frac{\scalar{\hat{\tau}_2(w^1)}{\tau_1(w^1)}_{\C}}{|\hat{\tau}_1(w^1)|_{\C}}-\frac{\scalar{\hat{\tau}_2(w^2)}{\tau_1(w^2)}_{\C}}{|\hat{\tau}_1(w^2)|_{\C}}}_{\cY_{\rho}}\lesssim \norm{w^1-w^2}_{\cY_\rho^7}\,.\label{diff taus}
\end{aligned}\end{equation}
\end{lemma}

\begin{proof}
We shall first show that \begin{align}
\min\left(\frac{1}{||\hat{\tau}_1|_{\C}|},\,\frac{1}{\left||\hat{\tau}_2|\hat{\tau}_1|_{\C}^2-\hat{\tau}_1\scalar{\hat{\tau}_1}{\hat{\tau}_2}_{\C}|_{\C}\right|}\right)\gtrsim 1.\label{claim lbound}
\end{align}
For the first term, using \eqref{E.smallness main est} and that $\cC_1 \geq 1$, we can directly estimate \begin{align}
\Re|\hat{\tau}_1|_{\C}^2=1+(\Re w_2)^2-(\Im w_2)^2 \geq \frac12\,,\label{claim step 0}
\end{align}
which yields the desired lower bound thanks to the elementary estimate $|z|\geq \sqrt{\Re z^2}$.

Concerning the second, note that
$$
|\hat{\tau}_2|\hat{\tau}_1|_{\C}^2-\hat{\tau}_1\scalar{\hat{\tau}_1}{\hat{\tau}_2}_{\C}|_{\C}^2 = 1+ 3(w_2 w_3)^2 + w_2^4 + 2 w_2^2 + w_3^2\,,
$$
and that \eqref{E.cC1} and \eqref{E.smallness main est} imply that
$$
|(\Re w_j) (\Im w_i)| \leq \frac{\cC_1}{300 \cC_1^4} \leq \frac{1}{300 \cC_1^3}\,, \quad \textup{for all } i,j \in \{1,2\}\,.
$$
Hence, using that $\cC_1 \geq 1$, it is immediate to check that
 \begin{align*}
\Re \big|\hat{\tau}_2|\hat{\tau}_1|_{\C}^2-\hat{\tau}_1\scalar{\hat{\tau}_1}{\hat{\tau}_2}_{\C}\big|_{\C}^2\geq \frac{1}{2}\,,
\end{align*}
which in turn implies \eqref{claim lbound}.

Note that \eqref{claim lbound} immediately implies \eqref{bound taus}, since these functions are holomorphic away from $0$, as long as the denominators are bounded.

Similarly, we can estimate the difference of these quantities with respect to different $w$'s as 
\begin{align*}
\norm{\tau_1(w^1)-\tau_1(w^2)}_{\cY_\rho^3}+\norm{\tau_2(w^1)-\tau_2(w^2)}_{\cY_\rho^3}\lesssim \norm{w^1-w^2}_{\cY_\rho^7}\,.
\end{align*}
directly by the definition. 

For the difference of the quantities involving a complexified modulus, we can estimate 
\begin{align*}
    \norm{\frac{1}{|\tau_1(w^1)|_{\C}}-\frac{1}{|\tau_1(w^2)|_{\C}}}_{\cY_\rho}=\norm{\frac{|\tau_1(w^2)|_{\C}-|\tau_1(w^1)|_{\C}}{|\tau_1(w^1)|_{\C}|\tau_1(w^2)|_{\C}}}_{\cY_\rho}\lesssim \norm{|\tau_1(w^2)|_{\C}-|\tau_1(w^1)|_{\C}}_{\cY_\rho}\,,
\end{align*}
where we have again used the lower bound \eqref{claim lbound}. Then, to estimate the right-hand side, we rely on Lemma \ref{lemma abcd}.  Expanding the definitions and using Lemma \ref{lemma abcd}, as well as \eqref{claim step 0}, we get that
\begin{align*}
& \norm{|\tau_1(w^2)|_{\C}-|\tau_1(w^1)|_{\C}}_{\cY_\rho}\\
& \quad \leq\sup_{|\beta|\leq \rho,\, (s^1,s^2)\in \T^4,\, l\in [-1,1]} \bigg\{ |s^1-s^2|^{-\frac{1}{2}} \Big||\tau_1(w^1)[l,s^1+i\beta]|_{\C}-|\tau_1(w^1)[l,s^2+i\beta]|_{\C}\\
&\qquad -|\tau_1(w^2)[l,s^1+i\beta]|_{\C}+|\tau_1(w^2)[l,s^2+i\beta]|_{\C}\Big|\\
&\qquad+\left||\tau_1(w^1)[l,s^1+i\beta]|_{\C}-|\tau_1(w^2)[l,s^1+i\beta]|_{\C}\right| \bigg\} \\
& \quad \lesssim \sup_{|\beta|\leq \rho,\, (s^1,s^2)\in \T^4,\, l\in [-1,1]} \bigg\{ |s^1-s^2|^{-\frac{1}{2}}\Big(\Big|\tau_1(w^1)[l,s^1+i\beta]-\tau_1(w^1)[l,s^2+i\beta]-\tau_1(w^2)[l,s^1+i\beta]\\
&\qquad+\tau_1(w^2)[l,s^2+i\beta]\Big|+\Big(\big|\tau_1(w^1)[l,s^1+i\beta]-\tau_1(w^2)[l,s^1+i\beta]\big|\\
&\qquad+\big|\tau_1(w^1)[l,s^2+i\beta]-\tau_1(w^2)[l,s^2+i\beta]\big|\Big)\left|\tau_1(w^2)[l,s^1+i\beta]-\tau_1(w^2)[l,s^1+i\beta]\right|\Big)\\
&\qquad +\left|\tau_1(w^1)[l,s^1+i\beta]-\tau_1(w^2)[l,s^1+i\beta]\right| \bigg\} \\
&\quad \lesssim \norm{w^1-w^2}_{\cY_\rho^7}\,.
\end{align*}
Note that here  we have also used \eqref{bas est}, and the fact that $\tau_1$ is an affine function in $w$. Combining the two previous estimates, we immediately get that
\begin{align*}
\norm{\frac{1}{|\tau_1(w^1)|_{\C}}-\frac{1}{|\tau_1(w^2)|_{\C}}}_{\cY_\rho}\lesssim \norm{w^1-w^2}_{\cY_\rho^7}.
\end{align*}
The other bounds in \eqref{diff taus} can be shown in a similar manner. We skip the details. 
\end{proof}

Having this lemma at hand, we can now deal with $A_\ep(w)w_{5-7}$. Indeed, taking into account \eqref{def Aeps}, we see that, for $i \in \{1,2\}$, 
 \begin{align*}
\mel\norm{A_\eps(w^i) w_{5-7}^i}_{\cY_{\rho'}^3}\\
&\lesssim \frac{1}{\rho-\rho'}\norm{U_\eps}_{\cY_\rho^3}\norm{w^i}_{\cY_{\rho'}^7}\Biggl(\norm{\tau_1(w^i)}_{\cY_{\rho'}^3}\norm{\frac{1}{|\hat{\tau}_1(w^i)|_{\C}}}_{\cY_{\rho'}}\\
&\qquad+\norm{\tau_2(w^i)}_{\cY_{\rho'}^3}\norm{\frac{1}{|\hat{\tau}_2(w^i)-\scalar{\hat{\tau}_2(w^i)}{\tau_1(w^i)}_\C \tau_1(w^i)|_{\C}}}_{\cY_{\rho'}}\left(1+\norm{\frac{\scalar{\hat{\tau}_2(w^i)}{\tau_1(w^i)}_{\C}}{|\hat{\tau}_1(w^i)|_{\C}}}_{\cY_{\rho'}}\right)\Biggr)\\
&\lesssim \frac{1}{\rho-\rho'}\,,
\end{align*}
uniformly in $0 \leq \rho' < \rho \leq \rho_0$. Note that here we have used Lemma \ref{cauchy est}, \eqref{bd ueps} and \eqref{bound taus}.

Similarly, using Lemma \ref{cauchy est}, \eqref{diff ueps}, \eqref{diff taus},  and the discrete product rule, we can estimate \begin{align*}
\norm{A_\eps(w^1)w^1-A_\eps(w^2)w^2}_{\cY_{\rho'}^3}\lesssim \frac{1}{\rho-\rho'}\norm{w^1-w^2}_{\cY_\rho^7}\,.
\end{align*}
uniformly in $0 \leq \rho' < \rho \leq \rho_0$. At this point, for $j \in \{5,6,7\}$ and $i \in \{1,2\}$, one can easily get that
\begin{equation} \label{E.567-1}
    \|F_j(w^i)\|_{\cY_{\rho'}} \lesssim \frac{1}{\rho-\rho'}\,, 
\end{equation}
and that
\begin{equation} \label{E.567-2}
     \|F_j(w^1)- F_j(w^2)\|_{\cY_{\rho'}} \lesssim \frac{1}{\rho-\rho'}\, \|w^1-w^2\|_{\cY_{\rho}^7}\,,
\end{equation}
uniformly in $0 \leq \rho' < \rho \leq \rho_0$. 

\begin{proof}[Proof of Proposition \ref{bd F}]
    The result immediately follows from \eqref{E.1234-1}, \eqref{E.1234-2}, \eqref{E.567-1} and \eqref{E.567-2}. 
\end{proof}

\section{The case of general closed vortex sheets}
\label{S.closedSurface}

We explain how the proof extends to vortex sheets that are normal graphs over a fixed analytic closed surface in~$\RR^3$.  The flat-graph analysis uses the ambient geometry only through four objects: the graph parametrization, the Biot--Savart kernel, the transport law for the sheet strength, and the jump formula.  All four admit exact intrinsic analogues on a closed analytic surface.  The only genuinely new feature is topological: on a surface of genus greater than zero, the tangential sheet strength is not globally a perp-gradient, and one must work with a divergence-free tangent vector density instead.

\subsection{Geometric setting}

Let $M$ be a closed, oriented, real-analytic two-manifold.  We fix a real-analytic embedding $\widehat{\mathbf{x}} : M \to \RR^3$, and endow~$M$ with the metric induced by the Euclidean metric on~$\R^3$ via the embedding~$\widehat{\mathbf{x}}$. 

We denote the corresponding Levi--Civita connection by~$\nabla^M$, and by $\mathbf{n}:M\to\RR^3$ its unit normal. Likewise, let $\mathcal{S} := -\nabla^M\mathbf{n}$ be the Weingarten map. Then, we set
\[
    P_r(y) := I - r\mathcal{S}(y)\,, \qquad J_r(y) := \det P_r(y)\,.
\]
For small enough $r_0>0$, the tubular map $\mathbf{x}:M\times (-r_0,r_0)\to \R^3$ defined by
\[
    \mathbf{x}(y,r) := \widehat{\mathbf{x}}(y) + r\,\mathbf{n}(y)
\]
is a diffeomorphism onto its image. Moreover,
it satisfies ${\rm D}_y\mathbf{x}(y,r)[v] = {\rm D}\widehat{\mathbf{x}}(y)[P_r(y)v]$, so $J_r(y)$ is the Jacobian of $y\mapsto\mathbf{x}(y,r)$ relative to the reference area form~$dA_M$.

For $h : M \to (-r_0,r_0)$, the \emph{normal graph} and its parametrizing map are
\[
    \Gamma_h := \{\mathbf{x}(y,h(y)) : y\in M\}\,,
    \qquad
    \mathbf{x}_h(y) := \mathbf{x}(y,h(y))\,.
\]
For a tangent vector field $b$ on~$M$, its pushforward by $\mathbf{x}_h$ is the map $M\to\RR^3$ given by
\[
    \mathcal{P}_h b(y) := {\rm D}\mathbf{x}_h(y)[b(y)].
\]
Hence, the vector $ \mathcal{P}_h b(y) \in\RR^3$ is tangent to $\Gamma_h$ at the point $\mathbf{x}_h(y)$.  

The corresponding vector-valued measure on~$\Gamma_h$ is $(\mathcal{P}_h b/G_h)\,dS_{\Gamma_h}$, where $dS_{\Gamma_h}$ is the surface measure on~$\Gamma_h$. Here $G_h(\cdot,t)$ is the area density with respect to the measure $dA_M$, defined as
\[
dS_{\Gamma_h(t)}(\mathbf{x}_{h(\cdot,t)}(y))=G_h(y,t)\,dA_M(y)\,.
\]
Therefore, the aforementioned vector-valued measure is characterized by
\begin{equation}\label{E.cc3d_Ph}
    \int_{\Gamma_h}
    \frac{\mathcal{P}_h b(\mathbf{x}_h^{-1}(x))}{G_h(\mathbf{x}_h^{-1}(x))}\cdot\Phi(x)\,dS_{\Gamma_h}(x)
    = \int_M \mathcal{P}_h b(y)\cdot\Phi(\mathbf{x}_h(y))\,dA_M(y)\,,
\end{equation}
for every $\Phi\in C_c^\infty(\RR^3;\RR^3)$.

\subsection{Layered vorticities}

For a fixed $\ep > 0$ and  for each label $l\in(-1,1)$, let $h_{\ep,l} : M\to(-r_0,r_0)$ be a collection of smooth functions, and let $b_{\ep,l}$ be a divergence-free tangent vector field on~$M$ (i.e., $\operatorname{div}_M b_{\ep,l} = 0$).  Setting
\[
    \mathbf{x}_{\ep,l}(y,t) := \mathbf{x}\bigl(y,h_{\ep,l}(y,t)\bigr)\,,
    \quad \textup{and} \quad
    \Gamma_{\ep,l}(t) := \mathbf{x}_{\ep,l}(M,t)\,,
\]
we will consider \emph{layered vorticities}, that is, distributions of the form
\begin{equation}\label{E.cc3d_omeps}
    \langle\om_\ep(\cdot,t),\Phi\rangle
    := \int_{-1}^1\!\int_M
    \mathcal{P}_{h_{\ep,l}(y,t)} b_{\ep,l}(y,t)\cdot\Phi\!\bigl(\mathbf{x}_{\ep,l}(y,t)\bigr)
    \,dA_M(y)\,\dl,
\end{equation}
which represent the superposition $\om_\ep = \int_{-1}^1 (G_{h_{\ep,l}}^{-1}\mathcal{P}_{h_{\ep,l}}b_{\ep,l})\,dS_{\Gamma_{\ep,l}} \dl$.

A first observation is that this defines a divergence-free field on~$\R^3$:

\begin{lemma}\label{L.cc3d_divfree}
For every $t\in[0,T]$, $\operatorname{div}\om_\ep(\cdot,t)=0$ in $\mathcal{D}'(\RR^3)$.
\end{lemma}
\begin{proof}
For $\varphi\in C_c^\infty(\RR^3)$, the chain rule gives
$\mathcal{P}_{h_{\ep,l}}b_{\ep,l}\cdot\nabla\varphi(\mathbf{x}_{\ep,l})
= b_{\ep,l}\cdot\nabla_M(\varphi\circ\mathbf{x}_{\ep,l})$.
Hence 
$$
\langle\operatorname{div}\om_\ep,\varphi\rangle = -\langle\om_\ep,\nabla\varphi\rangle
= -\int_{-1}^1\!\int_M b_{\ep,l}\cdot\nabla_M(\varphi\circ\mathbf{x}_{\ep,l})\,dA_M\,\dl\,.
$$
Since $M$ is closed, integrating by parts and using $\operatorname{div}_M b_{\ep,l}=0$, gives $\langle\operatorname{div}\om_\ep,\varphi\rangle=0$.
\end{proof}

It is convenient to define the \emph{weighted thickness} as
\begin{equation}\label{E.cc3d_q}
    q_{\ep,l} := J_{h_{\ep,l}}\,\ep^{-1}\pd_lh_{\ep,l}\,.
\end{equation}
This equals the Jacobian of the volume-filling map $F_\ep(y,l,t):=\mathbf{x}(y,h_{\ep,l}(y,t))$ ($\det {\rm D}_{(y,l)}F_\ep = \ep\,q_{\ep,l}$).
Provided that $q_{\ep,l}\geq c_0>0$, the layers are strictly ordered and one can verify that 
\[
    \om_\ep\bigl(\mathbf{x}(y,h_{\ep,l}(y,t)),t\bigr)
    = \frac{\mathcal{P}_{h_{\ep,l}(y,t)} b_{\ep,l}(y,t)}{\ep\,q_{\ep,l}(y,t)}\,.
\]

\subsection{Main statement}

Formal computations suggest that the limiting vortex sheet $(\Gamma_h,b)$ should be governed by the \emph{Birkhoff--Rott system}. For a function $h$ and a divergence-free tangent vector field~$b$ on~$M$, one should have that
\begin{equation} \label{E.cc3d_evh}
\left\{
\begin{aligned}
    \, \pd_t h + V[h,b]\cdot\nabla_M h &= W[h,b]\,, \\
    \pd_t b + \mathcal{L}_{V[h,b]} b &= 0\,.
\end{aligned}
\right.
\end{equation}
Here, $\mathcal L$ denotes the Lie derivative, $W[h,b]:=\mathrm{BR}[h,b]\cdot\mathbf{n}$ and $V[h,B]\in T_yM$ is defined by the equation 
$$
{\rm D}\mathbf{x}_h(y)\,V[h,b](y)=\Pi_{T\Gamma_h}\mathrm{BR}[h,b](y)\,,
$$ 
where $\Pi_{T\Gamma_h}$ denotes the tangential projector, and where
\begin{equation}\label{E.cc3d_BR}
    \mathrm{BR}[h,b](y) := \mathrm{P.V.}\;\frac{1}{4\pi}\int_M
    \frac{\mathcal{P}_{h(z)} b(z)\times\bigl(\mathbf{x}_h(y)-\mathbf{x}_h(z)\bigr)}
    {\bigl|\mathbf{x}_h(y)-\mathbf{x}_h(z)\bigr|^3}\,dA_M(z)\,.
\end{equation}

To obtain an analog of Theorem \ref{main thm} for general closed surfaces in~$\R^3$, we then consider layered initial data $(h^0_{\ep,l},q^0_{\ep,l},b^0_{\ep,l})_{|l|<1}$ as above with the following properties. To state them, given $\rho>0$, let $A_\rho(M)$ be the Banach space of analytic functions on~$M$ with the norm $\|f\|_\rho := \sum_{m\geq 0}(\rho^m/m!)\,\|\nabla^m f\|_{L^\infty(M)}$, and let $A_\rho(TM)$ be its counterpart for vector fields on~$M$. For some fixed $\rho_0 > 0$, we impose the following conditions on the initial data:

\begin{enumerate}[(i)]
\item Support:  $b^0_{\ep,l}\equiv 0$ for $|l|\geq\tfrac{1}{2}$. \smallbreak
\item Regularity:  Each datum belongs to $A_{\rho_0}(M)$ (resp.\ $A_{\rho_0}(TM)$), the maps in $l$ are $C^r((-1,1))$ for some non-integer $r > 1$, $\pd_lh^0_{\ep,l}=\ep\,J_{h^0_{\ep,l}}^{-1}q^0_{\ep,l}$, and $\operatorname{div}_M b^0_{\ep,l}=0$.
\item Divergence-free: With $\omega_\ep^0$ defined as in \eqref{E.cc3d_omeps}, $\div \omega_\ep^0 = 0$ in $\R^3$.  \smallbreak
\item Positivity:  $q^0_{\ep,l}(y)\geq c_0>0$ for all $|l|<1$ and all $y\in M$. 

\item Normalization:  There is some $b_0\in A_{\rho_0}(TM)$ such that $\int_{-1}^1 b^0_{\ep,l}\,\dl \to b_0$ in $A_{\rho_0}(TM)$ as $\ep \to 0$. 
\end{enumerate}

\noindent We are ready to state the result on the desingularization of a general vortex sheet in $\R^3$:

\begin{theorem}\label{T.cc3d_main}
Let $\rho_0 > 0$ be a real number and $r > 1$ a non-integer. For $\ep \in (0,1)$ small enough, consider the initial layer data $(h^0_{\ep,l},q^0_{\ep,l},b^0_{\ep,l})_{|l|\leq 1}$ satisfying the above hypotheses (i)--(v), and such that
$$
C_0 := \sup_{\ep \in (0,1)} \sup_{|l|\leq 1}\bigl(\|h^0_{\ep,l}\|_{\rho_0} + \|\nabla_M h^0_{\ep,l}\|_{\rho_0}+\|q^0_{\ep,l}\|_{\rho_0}+\|b^0_{\ep,l}\|_{\rho_0}\bigr) < \infty\,.
$$
Then, for any $\rho \in (0,\rho_0)$, there exist $\ep_0 \in (0,1)$, a constant $ C > 0$ and a time  $T > 0$ such that, for every $\ep \in (0,\ep_0]$, there is a unique family 
\[
    (h_{\ep,l},q_{\ep,l},b_{\ep,l})_{|l|<1}
    \in C^1\bigl([0,T),\,A_{\rho}(M)^2\times A_{\rho}(TM)\bigr),
\]
which coincide with $(h^0_{\ep,l},q^0_{\ep,l},b^0_{\ep,l})_{|l|<1}$ at time~$0$, such that:
\begin{enumerate}[(i)]
\item \emph{Evolution:}  On $[0,T)$ the triple satisfies the system
\begin{equation}\label{E.cc3d_sys}
    \pd_th_{\ep,\ell} + V_{\ep,\ell}\cdot\nabla_M h_{\ep,\ell} = W_{\ep,\ell},\quad
    \pd_t q_{\ep,\ell} + \operatorname{div}_M\!(q_{\ep,\ell}V_{\ep,\ell}) = 0,\quad
    \pd_t b_{\ep,\ell}+\mathcal{L}_{V_{\ep,\ell}}b_{\ep,\ell}=0,
\end{equation}
with $\operatorname{div}_M b_{\ep,\ell}=0$, $\pd_\ell h_{\ep,\ell}=\ep\,J_{\rho_{\ep,\ell}}^{-1}q_{\ep,\ell}$, and with $V_{\ep,\ell}$, $W_{\ep,\ell}$ being the tangential and normal components of the total velocity $\mathcal{U}_{\ep,\ell}:=\int_{-1}^1 K_\ep^{l,\ell}[h_{\ep,l},b_{\ep,l}]\,\dl$, where
\[
    K_\ep^{l,\ell}[h_{\ep,l},b_{\ep,l}](y,t)
    := \frac{1}{4\pi}\int_M
    \frac{\mathcal{P}_{h_{\ep,l}(z,t)}b_{\ep,l}(z,t)
    \times\bigl(\mathbf{x}(y,h_{\ep,\ell}(y,t))-\mathbf{x}(z,h_{\ep,l}(z,t))\bigr)}
    {\bigl|\mathbf{x}(y,h_{\ep,\ell}(y,t))-\mathbf{x}(z,h_{\ep,l}(z,t))\bigr|^3}
    \,dA_M(z)\,.
\]

\item \emph{Birkhoff--Rott:} There is a unique solution $(h,b)\in A_{\rho}(M) \times A_{\rho}(TM)$ to the system \eqref{E.cc3d_evh} on $[0,T)$, with initial data $(0,b_0)$. \smallbreak




\item \emph{Euler solution:}  The vorticity $\om_\ep$ defined by~\eqref{E.cc3d_omeps} is a weak solution of the 3D incompressible Euler equations on $\RR^3\times[0,T]$, with $\supp \om_\ep(\cdot, t)$ contained in a tubular neighborhood of $\Gamma_{h(\cdot,t)}$ of width $C\ep$.  If in addition the maps $l\mapsto\rho^0_{\ep,l},q^0_{\ep,l},B^0_{\ep,l}$ are $C^r$ for all $r > 1$, then $\om_\ep(\cdot,t)\in C^\infty([0,T]\times\RR^3,\RR^3)$. \smallbreak

\item \emph{Distributional convergence:} If the initial data converges as
    $$
    h_{\ep,0}^0 \to 0 \quad \textup{in } A_{\rho_0}(M)\,, \quad \textup{and} \quad \int_{-1}^1 b_{\ep,l}^0 \dd l \to b_0\quad \textup{in }A_{\rho_0}(TM)\,,\ \textup{ as } \ep \to 0^{+}\,,
    $$
    then the solution converges to the vortex sheet in the following sense: let $(h, b)$ be the unique analytic solution to \eqref{E.cc3d_evh} with initial data $(0, b_0)$. Then, it follows that, as $\ep \to 0^{+},$
    \begin{align*}
    & \circ \ h_{\eps,l}\rightarrow h\quad \text{in $L^\infty([0,T],C^k(M))$ for all $l\in (-1,1)$ and all $k\in \N$}\,,\\
    &\circ\ \int_{-1}^1b_{\eps,l}\dd l\rightarrow b \quad \text{in $L^\infty([0,T],C^k(M))$ for all $k\in \N\,.$}
    \end{align*}
\end{enumerate}
\end{theorem}


\subsection{Comments on the proof}

The proof of Theorem~\ref{T.cc3d_main} follows the same strategy as in the flat case.  We just indicate the three main points where the geometry enters.

\smallbreak
\noindent\textbf{Evolution system.}  The kinematic equation~\eqref{E.cc3d_sys}$_1$ is derived by differentiating $\mathbf{x}(y,h_{\ep,\ell}(y,t))$ in time and projecting onto $\mathbf{n}$. The thickness equation~\eqref{E.cc3d_sys}$_2$ follows from incompressibility and the identity $\operatorname{div} u = J_r^{-1}\operatorname{div}_M(J_r U) + J_r^{-1}\pd_r(J_r U^n)$ in tubular coordinates; the weighted variable $q_{\ep,l}$ is chosen precisely to remove curvature source terms.  Equation~\eqref{E.cc3d_sys}$_3$ is Lie transport of a tangent density, which preserves $\operatorname{div}_M b_{\ep,\ell}=0$ automatically.

\smallbreak
\noindent\textbf{Singular-kernel estimates.}  In a local analytic chart near the diagonal, the geometric chord satisfies
\[
    \mathbf{x}(y,h_{\ep,\ell}(y,t))-\mathbf{x}(z,h_{\ep,l}(z,t))
    = {\rm D}\mathbf{x}_{h_{\ep,l}}(y)(y-z)
    + \bigl(h_{\ep,\ell}(y)-h_{\ep,l}(y)\bigr)\mathbf{n}(y)
    + O(|y-z|^2).
\]
The principal part is the flat tangent-plane chord, and all operator estimates from the periodic case carry over with ${\rm D}\widehat{\mathbf{x}}$, $\mathbf{n}$, $G_h$, $\P_h$ as analytic multiplicative coefficients. 

\smallbreak 
\noindent\textbf{Analytic function spaces.}  The periodic strip analytic norms are replaced by embedding $M$ into a complex manifold and using suitable tubular neighborhoods of $M$ instead of the strips.  These satisfy the same algebra and Cauchy estimates, and since all geometric factors (particularly, the embedding $\widehat{\mathbf{x}}$ and the normal $\mathbf n$ are analytic in $h$ and $\nabla h$ (for $\|h\|_{L^\infty(M)}<r_0$), the nonlinear vector field defined by the system~\eqref{E.cc3d_sys} has exactly the same one-derivative loss as in the flat case.  Nishida's theorem therefore applies without modification.

\vspace{-0.2cm}
\section*{Acknowledgements} 
This work has received funding from the European Research Council (ERC) under the European Union's Horizon 2020 research and innovation programme through the grant agreement~862342 (A.E. and D.M.). A.E. is also partially supported by the grant PID2022-136795NB-I00 of the Spanish Science Agency and the ICMAT--Severo Ochoa grant CEX2019-000904-S. A.J.F. is partially supported by the grants PID2023-149451NA-I00 of MCIN/AEI/10.13039/501100011033/FEDER, UE;  and RYC2024-049142-I, MICINN (Spain).

\appendix
\vspace{-0.1cm}
\section{Cauchy's integral formula}\label{Sec 7}

\begin{lemma}\label{cauchy est}
Let $\rho'\in (0,\rho)$. Then, for any $f\in \cX_\rho$, it follows that $\norm{ \nabla f}_{\cX_{\rho'}^2} \lesssim \frac{1}{\rho-\rho'}\,\norm{f}_{\rho}.$
\end{lemma}
\begin{proof}  
Let $f \in \cX_\rho$ be fixed but arbitrary. The components of $\nabla f(\cdot+x) - \nabla f(\cdot)$ are harmonic for every $x \in \TT^2$. Moreover, for all $z \in \TT_{\rho'}^2$, the polydisk $B_{\frac{1}{2}(\rho-\rho')}^{\C}(z_1)\times B_{\frac{1}{2}(\rho-\rho')}^{\C}(z_2)$ lies in $\T^2_\rho$. Applying then Cauchy's integral formula in polydisks \cite[Chapter I, Theorem 1.3]{Range}, we get that
\begin{align*}
\mel|\nabla f(z+x)-\nabla f(z)|\lesssim \int_{B_{\frac{1}{2}(\rho-\rho')}^{\C}(z_1)\times B_{\frac{1}{2}(\rho-\rho')}^{\C}(z_2)} \left( \frac{|f(\xi+x)-f(\xi)|}{|z_1-\xi_1|^2|z_2-\xi_2|}+\frac{|f(\xi+x)-f(\xi)|}{|z_1-\xi_1||z_2-\xi_2|^2}\right)\dd\xi_1\dd\xi_2\\
&\lesssim \frac{1}{\rho-\rho'}\norm{f(\cdot+x)-f(\cdot)}_{L^\infty(\T^2_\rho)}\leq \frac{1}{\rho-\rho'}\norm{f}_{\cX_\rho}\,.
\end{align*}
Since the same argument can be used for the $L^\infty$-term in the definition of the norm $\|\cdot\|_{\cX_{\rho'}}^2$, the result immediately follows. 
\end{proof}

\begin{lemma}\label{lem lsc}
Let $\rho'\in (0,\rho)$. Then, the embedding $\cX_{\rho}\hookrightarrow \cX_{\rho'}$ is compact. 
\end{lemma}
\begin{proof}
By Cauchy's integral formula, namely Lemma \ref{cauchy est}, it follows that
$
\norm{u}_{C^1(\T_{\rho'}^2)}\lesssim \norm{u}_{\rho}.
$
Hence, by the compactness of the embedding $C^1(\T_{\rho'}^2)\hookrightarrow C^\frac{1}{2}(\T_{\rho'}^2)$, any bounded sequence must have a convergent subsequence. The limit of this subsequence must satisfy the Cauchy--Riemann equations distributionally. It is therefore holomorphic, and it belongs to $\cX_{\rho'}$.
\end{proof}

\bibliographystyle{amsplain}

\end{document}